\documentclass[10pt,reqno,letter]{article}

\usepackage[margin=1in]{geometry}

\usepackage{listings}
 \usepackage[usenames,dvipsnames]{color}
 \usepackage[colorlinks=true, pdfstartview=FitV, linkcolor=blue, citecolor=blue, urlcolor=blue]{hyperref}

\usepackage[nottoc,numbib]{tocbibind}
\usepackage{helvet}

\usepackage{amsmath,amsthm,amsfonts,amssymb,amscd}

\usepackage{mathtools} 
\usepackage{mathabx}
\usepackage{float}
\usepackage{graphicx}
\usepackage{makeidx}
\usepackage{enumerate}
\usepackage{mathtools}
\usepackage{xfrac}

\usepackage[mathscr]{euscript}
\usepackage{stmaryrd}
\usepackage{microtype}
\usepackage{booktabs}
\usepackage{cleveref}
\usepackage{bookmark}
\usepackage{mathrsfs}

\usepackage{emptypage}
\usepackage{tikz}

\DeclareMathAlphabet{\mathpzc}{OT1}{pzc}{m}{it}

\theoremstyle{plain}
\newtheorem{theorem}{\scshape Theorem}[section]
\newtheorem{conjecture}[theorem]{\scshape Conjecture}
\newtheorem{proposition}[theorem]{\scshape Proposition}
\newtheorem{lemma}[theorem]{\scshape Lemma}
\newtheorem{corollary}[theorem]{\scshape Corollary}
\theoremstyle{definition}
\newtheorem{definition}[theorem]{\scshape Definition}

\newtheorem{remark}[theorem]{\scshape Remark}

\renewcommand{\hat}{\widehat}
\renewcommand{\tilde}{\widetilde}

\definecolor{grey}{rgb}{0.5,0.5,0.5}
\definecolor{lightgrey}{rgb}{0.9,0.9,0.9}
\definecolor{darkgreen}{rgb}{0,0.6,0}
\definecolor{orange}{rgb}{1,0.5,0}

\def\intint{\int\!\int}
\def\BB{\mathcal B}
\def\HH{\mathcal H}
\def\NN{\mathbb N}

\def\SS{\mathbb S}
\def\RR{\mathbb R}
\def\PP{\mathbb P}
\def\TT{\mathbb T}
\def\TTT{\mathcal T}
\def\SSS{\mathcal S}
\def\ZZ{\mathbb Z}
\def\CC{\mathbb C}
\def\RSZ{\mathcal R}
\newcommand*{\Id}{\ensuremath{\mathrm{Id}}}
\newcommand*{\Tr}{\ensuremath{\mathrm{Tr}}}
\def\eps{\varepsilon}

\def\ee{\mathfrak h}

\newcommand{\norm}[1]{\left\lVert#1\right\rVert}
\newcommand{\abs}[1]{\left\vert#1\right\vert}

\def\tilde{\widetilde}

\numberwithin{equation}{section} 

\def\div{\mathop{\rm div}\nolimits}    
\def\supp{\mathop{\rm supp}\nolimits}    

\def\p{\partial}

\renewcommand{\sfrac}{\frac}

\usepackage{fancyhdr}
\title{Nonuniqueness of weak solutions to the SQG equation}

\author{
Tristan Buckmaster\footnote{
Courant Institute of Mathematical Sciences,
New York University,
New York, NY 10012, USA.
{\footnotesize \href{mailto:buckmaster@cims.nyu.edu}{buckmaster@cims.nyu.edu}}.
}
\and Steve Shkoller\footnote{
Department of Mathematics,
University of California,
Davis, CA 95616, USA.
{\footnotesize \href{mailto:shkoller@math.ucdavis.edu}{shkoller@math.ucdavis.edu}}.
}
\and Vlad Vicol\footnote{
Department of Mathematics,
Princeton University,
Princeton, NJ 08544, USA.
{\footnotesize \href{mailto:vvicol@math.princeton.edu}{vvicol@math.princeton.edu}}.
}
}
\date{\today}

\begin{document}

\maketitle

\noindent
{\bf Abstract.} 
We prove that weak solutions of the inviscid SQG equations are not unique, thereby answering Open Problem 11 in~\cite{DLSz2012}. Moreover, we also show that weak solutions of the dissipative SQG equation are not unique, even if the fractional dissipation is stronger than the square root of the Laplacian.  

\setcounter{tocdepth}{2} 
\tableofcontents

\section{Introduction}
\label{sec:intro}

The two-dimensional surface quasi-geostrophic (SQG) equation is a fundamental example of  active scalar transport, and is
classically written \cite{CoMaTa1994} as 
\begin{subequations} 
\label{eq:SQG-old}
\begin{align}
&\partial_t \theta + u \cdot \nabla \theta  =0 \,,   \label{eq:SQG-old-a} \\
&u = \RSZ^\perp \theta := \nabla^\perp \Lambda^{-1} \theta \,, \label{eq:SQG-old-b}
\end{align}
\end{subequations}
a transport equation 
for the unknown scalar field $\theta = \theta(x,t)$, where $(x,t)\in \TT^2 \times \RR = [-\pi,\pi]^2 \times \RR$. 
In \eqref{eq:SQG-old},  $\Lambda = (-\Delta)^{1/2}$, $\RSZ = (\RSZ_1,\RSZ_2)$ is the vector of Riesz-transforms, 
$\nabla^\perp = (-\partial_2, \partial_1)$, and   $x^\perp = (-x_2 ,x_1)$ for any vector $x = (x_1,x_2)$.
We consider solutions of the  SQG  equation \eqref{eq:SQG-old} which have zero mean on $\TT^2$, i.e. $\int_{\TT^2} \theta(x,t) dx=0$, a quantity which is conserved in time even for weak solutions.

In the context of geophysical fluid dynamics, the variable $\theta$ denotes the temperature (or surface buoyancy function) in
a rapidly rotating stratified fluid with uniform potential vorticity \cite{HePiGaSw1995} and has applications in both meteorological and oceanic 
flows \cite{Pe1982}.

Mathematically,   two-dimensional SQG flows have the potential for finite-time singularity formation \cite{CoMaTa1994} and
possess striking similarities to three-dimensional Euler solutions; in fact,  the vector $ \nabla^\perp \theta$ is governed by the same evolution equation  as the vorticity of the 3-D Euler flow:
\begin{align}
\partial_t (\nabla^\perp \theta) + u \cdot \nabla (\nabla^\perp \theta) = \nabla u \cdot \nabla^\perp \theta.
\label{eq:SQG:3D:vorticity}
\end{align}
As such, \eqref{eq:SQG-old} has been intensively analyzed over the past two decades. While the local existence of smooth solutions in Sobolev 
spaces $H^s$ with $s>2$, or H\"older spaces $C^{1,\alpha}$ with $\alpha>0$, has been established in \cite{CoMaTa1994}, to date the question of whether a finite-time singularity may develop from smooth initial 
datum
remains open, in analogy to the similar question for the 3-D incompressible Euler system.  When \eqref{eq:SQG-old} is posed
on $ \RR^2 \times \RR $ with datum having infinite kinetic energy, a gradient blowup may occur~\cite{CaCo2010}, but for datum on
 $\TT^2$ and with finite energy,  it is only known that arbitrarily large growth of high Sobolev norms  is possible from arbitrarily small 
initial datum~\cite{KiNa2012}. The collapsing hyperbolic saddle blowup-scenario from~\cite{CoMaTa1994} was ruled out analytically 
in~\cite{Co1998,CoFe2002}, and the modern numerical simulations of~\cite{CoLaShTsWu12} were able to resolve the equations past the initially predicted singular time~\cite{CoMaTa1994}. 
A different blowup scenario via a cascade
of filament instabilities of geometrically decreasing spatial and temporal scales was proposed in~\cite{Sc2011}. 
The first example of a  non-steady global in time smooth solution was obtained only very recently~\cite{CaCoGo2016}.

\subsection{SQG conservation laws}
Fundamental to our subsequent analysis is the fact that sufficiently smooth solutions of \eqref{eq:SQG-old} conserve the square of the
$\dot{H}^{- {1/2} }(  \mathbb{T} ^2 )$ norm of $\theta$.  Upon taking the $L^2$ inner product of \eqref{eq:SQG-old-a} with 
$\Lambda^{-1} \theta$,  integrating by parts in the nonlinear term, and using that $u\cdot \nabla \Lambda^{-1} \theta  
= \nabla^\perp \Lambda^{-1}\theta \cdot \nabla \Lambda^{-1} \theta = 0$, if follows that if $\theta$ is sufficiently smooth 
($\theta \in L^{3}_{t,x}(\TT^2 \times \RR)$ is sufficient, cf.~\cite{IsVi2015}),  then
\begin{align}
\HH(t) := \| \theta(\cdot,t)\|_{\dot{H}^{-1/2}(\TT^2)}^2 = \|\theta_0\|_{\dot{H}^{-1/2}(\TT^2)}^2
\label{eq:Hamiltonian:conserved}
\end{align}
for initial datum $\theta_0 \in \dot{H}^{- 1/2 }(\mathbb{T}^2)$. In fact, the $\dot{H}^{-1/2}$ norm of $\theta$ is the {\em Hamiltonian} $\HH$ associated to an action function (we systematically ignore the factor of $1/2$ that is usual present) from which the SQG equation may be derived via an Euler-Poincar\'e variational principle~\cite{Re1995}.

Additionally, due to the pure transport nature of \eqref{eq:SQG-old-a}, and the fact that the Lagrangian flow induced by the incompressible vector field $u$ preserves volume, sufficiently smooth solutions of the initial value problem for \eqref{eq:SQG-old} conserve the $L^p( \mathbb{T}^2)$ norms of $\theta$, so that
\begin{align}
\| \theta(\cdot,t)\|_{L^p(\TT^2)} = \|\theta_0\|_{L^p(\TT^2)}  
\label{eq:Casimir:conserved}
\end{align}
where $\theta_0 \in L^p(  \mathbb{T} ^2 )$ is the initial datum, and $1 \leq p \leq \infty$. When $p=2$, from elementary properties of the Riesz transform it follows that the {\em kinetic energy} is conserved  for smooth solutions $\| u(\cdot,t)\|_{L^2(\TT^2)} = \|u_0\|_{L^2(\TT^2)} $.\footnote{
In addition to \eqref{eq:Casimir:conserved}, writing the SQG equation in Lagrangian coordinates has further geometric advantages. For instance, in~\cite{CoViWu2015} it is shown that the solutions $\theta \in C^0_t C^{1,\alpha}_x$ have Lagrangian trajectories which are real-analytic functions of time.}

As noted in~\cite{Ta2014}, formally the conservation laws \eqref{eq:Hamiltonian:conserved}--\eqref{eq:Casimir:conserved} are immediate consequences of Noether's Theorem and the fact that the SQG equation belongs to a general class of active scalar equations satisfied by the vorticity of a generalized two-dimensional Euler equation on a Lie algebra with a specific inner product~\cite[Section 2.2]{Re1995} (see also~\cite{Ta2016,Wa2016,Con2016} for a more recent account). We address this point of view in more detail in Section~\ref{sec:geodesic} and Appendix~\ref{appendix:EP} below, where we also present the momentum equation for the incompressible velocity field $v$ whose vorticity is the function $\theta$ in~\eqref{eq:SQG-old}.

While \eqref{eq:SQG:3D:vorticity} suggests that the problem of finite-time singularities for SQG is similar to that for 3-D Euler, the 
aforementioned variational point-of-view justifies a direct analogy between the conservation laws for SQG and those for the 2-D Euler equations: 
\eqref{eq:Hamiltonian:conserved} plays the role of the conservation of kinetic energy in 2-D Euler, while \eqref{eq:Casimir:conserved} is 
analogous to the conservation of the Casimir functions in 2-D Euler. Therefore, we expect that a turbulent SQG solution exhibits a dual cascade 
of energy, as predicted by the Batchelor-Kraichnan theory for two dimensional Euler flows~\cite{Con1998,Con2002,CoTaVi2014}. Motivated by two-dimensional turbulence, we are thus naturally lead to consider weak solutions of the SQG equation.

\subsection{Weak solutions of the SQG equation are not unique}
Motivated by \eqref{eq:Casimir:conserved} with $p=2$ one may define $\theta \in L^2_{\rm loc}(\RR, L^2(\TT^2))$ to be a weak solution of \eqref{eq:SQG-old} if 
\begin{align}
\intint_{\RR\times \TT^2} \left( \theta \partial_t \phi + \theta u \cdot \nabla \phi \right) \, dx\, dt = 0
\label{eq:weak:L2}
\end{align}
holds for any smooth test function $\phi\in C^\infty(\TT^2 \times \RR)$, so that \eqref{eq:SQG-old} holds in the sense of distributions on 
$\TT^2 \times \RR$. Using this definition, it was 
established in \cite{Re1995} that for any $\theta_0$ in $L^2$, there exists a global-in-time weak solution to the Cauchy problem for 
\eqref{eq:SQG-old}, with $\theta \in L^\infty([0,\infty), L^2(\TT^2))$. See also ~\cite{PuVa2015} for the global existence of weak solutions to the 
3-D quasi-geostrophic system. 
 We note the stark contrast here with 3-D Euler, for which the existence of weak solutions for any $L^2$ initial datum remains open.

Moreover, we note that the proof of the  existence of global  weak solutions to SQG is quite different that the proof of global solutions to
 2-D Euler (for which
the vorticity is one derivative smoother than the velocity ~\cite{MaBe2002}); in particular, the proof of~\cite{Re1995} relies on a 
special structure of the nonlinear term in \eqref{eq:SQG-old}, which arises from the fact that the Fourier multiplier 
relating $\theta$ to $u$ is an odd function of the frequency. More precisely,  
\begin{align}
\int_{\TT^2} \theta u\cdot \nabla \phi dx = \int_{\TT^2} \theta \RSZ^\perp \theta \cdot \nabla \phi dx = - \int_{\TT^2} \theta \RSZ^\perp 
\cdot (\theta \nabla \phi) dx = - \int_{\TT^2} \theta u\cdot \nabla \phi dx - \int_{\TT^2} \theta \left[ \RSZ^\perp \cdot, \nabla \phi\right] \theta dx
\label{eq:SQG:magic}
\end{align}
for any smooth test function $\phi$. Here and throughout the paper,  we denote by $[A,B]$ the commutator of the operators $A$ and $B$. 
Since the commutator $\left[ \RSZ^\perp \cdot, \nabla \phi\right]$ is an operator of order $-1$, it maps $\dot{H}^{-1/2}$ into $\dot{H}^{1/2}$ 
(see Appendix~\ref{sec:commutator}), so that weak solutions to \eqref{eq:SQG-old} may be defined for distributions $\theta \in \dot{H}^{-1/2}$.
We thus have the following
\begin{definition}[Weak solution of SQG]
\label{def:weak:solution}
A distribution $\theta \in L^2_{\rm loc}(\RR;\dot{H}^{-1/2}(\TT^2))$ is a weak solution of \eqref{eq:SQG-old} if
\begin{align*}
\int_{\RR} \langle \RSZ_i^\perp  \theta, \partial_t \Lambda^{-1}  \phi^i  \rangle + \langle \RSZ_j^\perp \theta , \RSZ_i^\perp \Lambda^{-1} \theta \partial_j \phi^i \rangle - \frac 12 \langle \RSZ_i \RSZ_j^\perp \theta, [\Lambda , \phi^i] \RSZ_j^\perp \Lambda^{-1} \theta \rangle \, dt = 0
\end{align*}
for any $\phi \in C^\infty_0(\TT^2 \times \RR)$ such that $\div \phi = 0$, where $\langle \cdot , \cdot \rangle$ denotes the 
$\dot{H}^{-1/2}$-$\dot{H}^{1/2}$ duality pairing.
\end{definition}
See also Definition~\ref{def:weak:sol:momentum} for a more intuitive and equivalent formulation of Definition~\ref{def:weak:solution}. 
For $\theta \in L^2_{\rm loc}(\RR, L^2(\TT^2))$, Definition~\ref{def:weak:solution} agrees with that given via \eqref{eq:weak:L2}. The boundedness of the Calder\'on commutator $[\Lambda,\phi^i]$ on $\dot{H}^{1/2}$ is used implicitly in Definition~\ref{def:weak:solution} (cf.~Appendix~\ref{sec:commutator}). 

Using the cancellation property~\eqref{eq:SQG:magic}, it was further shown in~\cite{Ma2008} that for $\theta_0 \in L^p$, with $p>4/3$, there exists a global weak solution $\theta \in L^\infty(\RR,L^p(\TT^2))$ to the Cauchy problem for \eqref{eq:SQG-old}.

The question of uniqueness of these weak solutions  has remained,  to date,  open and has been isolated as a challenging open problem  
in~\cite[Problem 11]{DLSz2012} (see also~\cite{Ma2008,Ru2011,AzBe2015,Co2015}).\footnote{One may also consider another class of weak solutions, the so-called patch-solutions, or sharp-fronts~\cite{Ro2005,Ga2008}. 
These  solutions are given by $\theta = {\bf 1}_{\Omega(t)}$,  where $\Omega(t)$ is a compact, simply-connected domain, 
with smooth boundary,  evolving  with the fluid. Although  these weak solutions are not smooth as functions on $\RR^2$, 
$\theta \in L^\infty_{\rm loc}(\RR,L^\infty(\RR^2))$, the boundary $\partial \Omega (t)$ is smooth and thus  local-in-time existence and 
uniqueness of 
such solutions is known.  For such patch solutions,  an important question is if  $\partial\Omega(t)$ can self-intersect in finite time~\cite{CoFoMaRo2005,Sc2011,GaSt2014}.  See also~\cite{CoFeRodrigo2004,FeRo2011,FeRo2012} for the existence of 
almost-sharp-fronts.} 
One of our main results is  that {\em below a certain regularity threshold,  weak solutions to the SQG equation are not unique}, thereby answering the question posed in \cite{DLSz2012}. We state this result as the following
\def\ee{{\mathscr H}}
\begin{theorem}[Nonuniqueness of weak solutions to SQG]
\label{thm:main:inviscid}
Suppose $\ee:\RR\rightarrow \RR^+$ is a smooth function with compact support. Then for every $1/2<\beta<4/5$ and $\sigma< \beta/(2-\beta)$, there exists a weak solution $\theta$, with $\Lambda^{-1}\theta \in C_t^{\sigma}C_x^{\beta}$, satisfying
\begin{align*}
\HH(t) = \int_{\TT^2}\abs{\Lambda^{- \sfrac 12}\theta(x,t)}^2\,dx=\ee(t)\,.
\end{align*}
for $t \in \RR$. 
\end{theorem}
Indeed, due to the compact support in time of the function $\theta$ in Theorem~\ref{thm:main:inviscid}, it follows that the trivial solution $\theta \equiv 0$ is not the only weak solution to \eqref{eq:SQG-old} which vanishes on the complement of a given time interval. 

Theorem~\ref{thm:main:inviscid} leaves open the question of whether the exponent $\beta$ can be taken arbitrarily close to $1$, which is the Onsager conjecture for the SQG equation (cf.~Conjecture~\ref{conj:Onsager} below).

The proof of Theorem~\ref{thm:main:inviscid} relies on a modification of the convex integration scheme employed by~\cite{DLSz2012b,DLSz2013,BuDLeIsSz2015} to study the Onsager conjecture for the 3-D Euler equations.\footnote{
See also~\cite{DLSz2012,DLSz2016} for  excellent review papers on the applicability of convex integration techniques in fluid dynamics, and connections to the $h$-principle.
}
It was suggested in~\cite{DLSz2012,Sh2011,IsVi2015} that the structure of the SQG nonlinearity is non-amenable to convex integration methods, because the multiplier relating $u$ to $\theta$ is an odd function of frequency. Herein,  we overcome this difficulty by rephrasing the equation in terms of a potential velocity $v$ (whose vorticity is the scalar $\theta$, see Section~\ref{sec:geodesic}), which allows us to apply Fourier analysis techniques to construct nontrivial high-high-low frequency interactions, crucial to the method of convex integration. We discuss these details in Section~\ref{sec:results} below.

\subsection{Weak solutions of the dissipative SQG equation are not unique}
Note that while weak solutions of the SQG equation may be defined for $\theta \in L^2_{t} \dot{H}^{-1/2}_x$, the existence of weak solutions obtained in~\cite{Re1995,Ma2008} requires an initial datum which is more regular (e.g. $\theta_0 \in L^p_x$ for $p>4/3$). One may thus ask a natural question: is it possible for that in a given (low) regularity regime one can both construct weak solutions via compactness arguments (viscosity solutions), and also construct weak solutions via convex integration? 

In order to answer this question in the positive, we consider the fractionally dissipative SQG system 
\begin{subequations} 
\label{eq:dissipative:SQG}
\begin{align}
&\partial_t \theta + u \cdot \nabla \theta  + \Lambda ^ \gamma \theta = 0 ,\\
&u = \RSZ^\perp \theta := \nabla^\perp \Lambda^{-1} \theta, 
\end{align}
\end{subequations} 
with $\gamma \in (0,2]$. 

Strong solutions of \eqref{eq:dissipative:SQG} have been considered extensively. 
The dissipative SQG equation has a 
natural scaling symmetry: if $\theta(x,t)$ is a $\TT^2$-periodic solution to the Cauchy problem for \eqref{eq:dissipative:SQG} with datum $\theta_0(x)$,  then $\theta_\lambda(x,t) = \lambda^{\gamma-1} \theta( \lambda x,\lambda^\gamma t)$ is a 
$\TT_\lambda^2=  [-\pi/\lambda,\pi/\lambda]^2$-periodic solution of \eqref{eq:dissipative:SQG} with initial datum $\theta_{0,\lambda}(x) = \lambda^{\gamma-1} \theta(\lambda x)$. In view of this scaling symmetry, the $L^\infty$ norm is scale invariant for 
$\gamma=1$.\footnote{For strong solutions, this is the strongest norm on which we have an a priori global in time bound: for any $\gamma>0$ 
we have that $\|\theta(t)\|_{L^\infty(\TT^2)}$ decays exponentially~\cite{CoCo2004,CoTaVi2015}. In fact, the argument in~\cite{CaVa2010} 
(see~\cite{CZVi2016}) shows that  $L^\infty_t L^2_x \cap L^2_t \dot{H}^{\gamma/2}_x$ {\em suitable weak solutions} (weak solutions 
which obey the local energy inequality) are bounded $L^\infty(\TT^2)$ for positive time. Thus, in the class of suitable weak solutions with
finite kinetic energy, the $L^\infty$ norm is the most important for a priori control.} \footnote{The Hamiltonian $\HH$ is scaling invariant for 
$\gamma =3/2$.} Therefore, for $\gamma>1$  \eqref{eq:dissipative:SQG} becomes semilinear and subcritical, and the global 
well-posedness of smooth solutions was established in~\cite{CoWu1999}. The critical case $ \gamma =1$ is a quasilinear problem, and the 
global regularity of smooth solutions for large initial datum was established in \cite{KiNaVo2007,CaVa2010}, with different proofs 
given in~\cite{KiNa2010, CoVi2012,CoTaVi2015}. For $\gamma<1$, the supercritical case, the question of finite-time singularities from 
smooth initial datum remains completely open, in analogy to the inviscid equations \eqref{eq:SQG-old}.\footnote{The finite-time blowup 
cannot occur for sufficiently smooth and small initial datum~\cite{CoCo2004, Wu2004, Mi2006, Ju2007}, and as $\gamma \to 1^-$ it also 
cannot happen for datum that is large, but not exceedingly large~\cite{CZVi2016}.} It is known, however,  that the global-in-time weak solutions 
 eventually become  smooth~\cite{Si2010,Da2011,Ki2011,CZVi2016}, which leads us to consider weak solutions of \eqref{eq:dissipative:SQG}.
In analogy to Definition~\ref{def:weak:solution}, we may define a weak solution to \eqref{eq:dissipative:SQG} as follows.
\begin{definition}[Weak solution of the dissipative SQG equation]
\label{def:dissipative:weak:solution}
The distribution $\theta \in L^2_{\rm loc}(\RR;\dot{H}^{-1/2}(\TT^2))$ is a weak solution of \eqref{eq:dissipative:SQG} if
\begin{align*}
\int_{\RR} \langle \RSZ_i^\perp \theta, \partial_t  \Lambda^{-1}  \phi^i  \rangle + \langle \RSZ_j^\perp \theta , \RSZ_i^\perp \Lambda^{-1} \theta \partial_j \phi^i \rangle - \frac 12 \langle \RSZ_i \RSZ_j^\perp \theta, [\Lambda , \phi^i] \RSZ_j^\perp \Lambda^{-1} \theta \rangle - \langle \RSZ_i^\perp  \theta, \Lambda^{\gamma-1} \phi^i  \rangle \, dt = 0
\end{align*}
for any $\phi \in C^\infty_0(\TT^2 \times \RR)$ such that $\div \phi = 0$, where $\langle \cdot , \cdot \rangle$ denotes the $\dot{H}^{-1/2}$-$\dot{H}^{1/2}$ duality pairing.
\end{definition}
Note that Definition~\ref{def:dissipative:weak:solution} does not require that solutions verify the  local energy inequality nor that they
possess the additional regularity
$\theta \in L^2_{\rm loc}(\RR;\dot{H}^{(\gamma-1)/2}(\TT^2))$; the weak solutions we consider need not be suitable weak solutions 
(see~\cite{CaVa2010}). In contrast to the  inviscid SQG equation, for the dissipative SQG equation it was shown in~\cite{Ma2008} 
that for any $\gamma>0$ and any $\theta_0 \in \dot{H}^{-1/2}$ there exists a  global-in-time weak solution $\theta$ 
to \eqref{eq:dissipative:SQG}.\footnote{This solution additionally obeys the energy inequality 
$\norm{\theta(t)}_{\dot{H}^{-1/2}}^2 + 2 \int_0^t \norm{\theta(s)}_{\dot{H}^{(\gamma-1)/2}}^2 ds \leq  \norm{\theta_0}_{\dot{H}^{-1/2}}^2$ for 
any $t\geq 0$.}   See also \cite{BaGr2015} for
the global existence of weak solutions when $\theta_0 \in L^{1+}$.

Using the convex integration scheme developed to prove Theorem~\ref{thm:main:inviscid}, we establish the nonuniqueness of weak solutions to the $\gamma$-dissipative SQG equation \eqref{eq:dissipative:SQG}, even for a range of  values for $ \gamma$ above $1$, the subcritical regime.

\begin{theorem}[Nonuniquess of weak solutions to dissipative SQG]
\label{thm:main:dissipative}
Suppose $\ee: \RR \rightarrow \RR^+$ is a smooth function with compact support. Then for every $1/2<\beta<4/5$, $0< \gamma < 2-\beta$ and $\sigma<\beta/(2-\beta)$, there exists a weak solution $\theta$, with $\Lambda^{-1} \theta \in C_t^{\sigma}C_x^{\beta}$, satisfying
\begin{align*}
\int_{\TT^2}\abs{\Lambda^{- \sfrac 12}\theta(x,t)}^2\,dx=\ee(t)\,
\end{align*}
for all $t \in \RR$.\footnote{The restriction $\gamma+\beta<2$ is sharp, in the sense that the $C^0_tC^\beta_x$ norm for $\Lambda^{-1} \theta$ is scale invariant precisely when $\gamma + \beta = 2$. We show here that for datum which is supercritical for the scaling of the equations, parabolic smoothing does not hold.}
\end{theorem}
Therefore, for the dissipative SQG equation, convex integration can coexist with weak compactness. This flexibility of the 
PDE~\eqref{eq:dissipative:SQG} is both due to the low regularity of the weak solution, and that enforcement of the local energy inequality is not
required. To the best of our knowledge, this is the first instance when the  convex integration scheme can be employed  
for an evolution equation arising in fluid dynamics, which is parabolic, and even semi-linear. The main ideas used in the proof of Theorem~\ref{thm:main:dissipative} are discussed in Section~\ref{sec:results}.

Note that as $\beta \to 1^{-}$, Theorem~\ref{thm:main:dissipative} holds with $\gamma \to 1^{+}$. This motivates Conjecture~\ref{conj:critical:SQG} below, which states that for the critical SQG equation ($\gamma=1$) we have a dichotomy of regularity exponents, whereby for $\beta\geq 1$ the energy equality holds, the uniqueness and global regularity of solutions holds; while for $\beta<1$ the uniqueness of weak solutions breaks down and the equation becomes flexible, i.e. amenable to convex integration constructions.

\subsection{Hydrodynamical systems as geodesic equations}
\label{sec:geodesic}

At least since the work of Poincar\'{e} \cite{Po1901}, it has been
well-known that the equations of motion of the finite-dimensional mechanical systems governed by Newtonian mechanics can be interpreted as the geodesic equations of a Riemannian metric on configuration space  or Lie group (see, for example, \cite{Br1946}).  The motion of
a rigid body, for example, is governed by a left-invariant Riemannian metric on the Lie group $SO(3)$ \cite{MaRa1999, KhAr1999}.  In his
seminal paper, \cite{Ar1966} showed that infinite-dimensional hydrodynamical systems could also be represented by geodesic equations
on the infinite-dimensional group of volume preserving diffeomorphisms $ \mathcal{D} _\mu$.  Specifically, Arnold proved that the 
incompressible Euler equations are geodesics with the respect to the $L^2$-right invariant metric on the Lie group $ \mathcal{D} _\mu$.   The
Euler-Poincar\'{e} variational principle \cite{MaRa1999, KhAr1999} asserts that such hydrodynamical  geodesic equations can be computed
for a rather general metric specified on the associated  Lie algebra $ \mathcal{V} $, the space of divergence-free vector fields, with Lie bracket
given by $[u,w] = \partial_j u w^j - \partial _j w u^j$.

For a positive-definite, self-adjoint operator $A$, we define the metric on $ \mathcal{V} $ by
\begin{equation}\label{metric}
(u, w) = \int_{ \mathbb{T}  ^2 } Au \cdot  w \, dx \,.
\end{equation} 
The metric \eqref{metric} is then right-translated over the Lie group $ \mathcal{D} _\mu$.  As we shall explain Appendix~\ref{appendix:EP}, the geodesic equations
associated to this metric are extrema of the action function
\begin{align}
\label{eq:action:*}
s(u) = \int_{\RR} \int_{\TT^2} A u \cdot u \, dx dt
\end{align}
for incompressible Lie-advected variations $\delta u$ that obey suitable boundary conditions, which  gives rise to the following hydrodynamical system:
\begin{subequations}
\label{eq:hydro}
\begin{align}
\partial_t v + u \cdot \nabla v + ( \nabla u)^T\cdot v & = - \nabla \tilde p,  \\
\div u & = 0, \\
v & = A u  \,.
\end{align}
\end{subequations}
Here $\tilde p$ denotes the pressure function,  a Lagrange multiplier enforcing the incompressibility of $u$.  When the operator $A$ is
the identity matrix, then  \eqref{eq:hydro} is the incompressible Euler equations (with pressure function $p = \tilde p + {\frac{1}{2}} |u|^2$); however,
the operator $A$ can be
differential operator, a nonlocal Fourier multiplier, or even a more general operator satisfying the positivity and symmetry conditions noted above.\footnote{When $A= (1- \alpha ^2 \Delta )$, $ \alpha >0$, the metric \eqref{metric} on $ \mathcal{V} $ is equivalent to the $H^1$-metric, and
 the system \eqref{eq:hydro} corresponds the well-studied Lagrangian Averaged Euler or  Euler-$ \alpha $ equations \cite{HoMaRa1998}.   More generally, if $A= (1- \Delta )^s$, $s \in \mathbb{R}  $, then \eqref{metric} is an
$H^s$ metric on $ \mathcal{V} $,  and
 it was shown in \cite[Section 3.3]{MaRaSh2000} that the vorticity $ \omega = \nabla ^\perp \cdot u$ satisfies
$\partial_t  (1- \Delta )^s \omega + u \cdot \nabla (1- \Delta )^s \omega =0$.}~\footnote{Letting $A$  be a specially chosen order $-2$ operator which is self-adjoint and positive-definite, \cite{Ta2016} proved that 
\eqref{eq:hydro} admits locally in time smooth solutions which blow up in finite time.}

As observed in~\cite{Re1995} (see also~\cite{Ta2016,Wa2016,Con2016}), if $A= \Lambda ^{-1} $, then \eqref{metric} is the $\dot{H}^{-1/2}$ metric on 
$ \mathcal{V} $, and it follows from \eqref{eq:hydro} that $\omega = \nabla^\perp \cdot u$ obeys
\[
\partial_t  \Lambda ^{-1}  \omega + u \cdot  \Lambda^{-1}  \omega =0  \,.
\]
A simple computation (see Section~\ref{sec:SQG:momentum} below) shows that $\theta = - \Lambda^{-1} \omega$ solves the SQG equation~\eqref{eq:SQG-old}.

\subsection{The Onsager conjecture and the uniqueness of weak solutions}
As discussed in Section~\ref{sec:geodesic},  solutions to the SQG equation \eqref{eq:SQG-old} are  extrema of the action function $s(u)$ in \eqref{eq:action:*}, with $A = \Lambda^{-1}$. Since $s(u)$ does not explicitly depend on $t$, for any such system the corresponding Hamiltonian 
\begin{align}
\HH(t) = \int_{\TT^2} \Lambda^{-1} u \cdot u \, dx = \int_{\TT^2} \Lambda^{-1} \theta \, \theta \, dx
\label{eq:Hamiltonian:2}
\end{align}
is formally conserved (see~\eqref{eq:Hamiltonian:conserved}).
Inspired by the Onsager conjecture for the incompressible  Euler equations (see the discussion in Section~\ref{sec:results} below), a 
fundamental question arises for the SQG equations: do weak solutions of \eqref{eq:SQG-old} conserve the Hamiltonian $\HH(t)$?  One may conjecture the following dichotomy:
\begin{conjecture}[Onsager conjecture for SQG]
\label{conj:Onsager}
Let $v = \Lambda^{-1} u = \Lambda^{-1} \RSZ^\perp \theta$. Define $\alpha_O = 1$.\footnote{Here the subindex $O$ of $\alpha_O$ stands for Onsager, as was suggested in~\cite{Kl2016}.}
\begin{itemize}
\item[\bf{(a)}] If $v  \in C(\RR;C^\alpha (\TT^2))$ is a weak solution of the SQG equation, with $\alpha > \alpha_O$, then  \eqref{eq:Hamiltonian:conserved} holds on $[0,T]$.
\item[\bf{(b)}] For any $1/2<\alpha <\alpha_O$, there exist infinitely many weak solutions of the SQG equation, with $v \in C ( \RR ; C^\alpha(\TT^2))$, such that \eqref{eq:Hamiltonian:conserved} fails.
\end{itemize}
\end{conjecture}

The rigid side (a) of this conjecture was resolved in~\cite{IsVi2015}, following the classical work of~\cite{CoETi1994}, by proving that $\theta \in L^{3}(\RR ; L^3(\TT^2))$ implies the conservation of $\HH$.  In this paper we address the flexible side (b) of the conjecture, and prove in~Theorem~\ref{thm:main:inviscid}, that (b) holds for $\alpha < 4/5$.  The regularity gap,  $\alpha \in [4/5,1)$,  for nonconservative weak 
solutions  of SQG remains for the  same reason that the Onsager conjecture for the 2-D Euler equations remains open, with an open range of values for
$ \alpha $ in  $[1/5,1/3)$.\footnote{Indeed, the construction in~\cite{BuDLeIsSz2015}, combined with ideas from~\cite{Ch2013,ChSz2014}, or 
from this paper, shows that $h$-principles are available for 2-D Euler with velocity fields in $C_t^\alpha C^\alpha_x$, for any $\alpha < 1/5$. 
The recent construction in~\cite{Is2016} does not apply in two dimensions, since any two non-parallel infinite lines on the plane intersect.} This issue is discussed further in Section~\ref{sec:results} below. Proving part (b) of Conjecture~\ref{conj:Onsager} for any $\alpha \in [4/5,1)$ appears to be challenging.

\begin{remark}[Onsager conjecture for extrema of a norm-inducing action functional]
\label{rem:Onsager}
The Euler and SQG equations are particular cases of equations obeyed by extrema of the action functional $s(u)$ in \eqref{eq:action:*}. For Euler $A = \Id$, while for SQG, $A=\Lambda^{-1}$. In general, let $A$ be any positive-definite self-adjoint operator which is translation invariant and 
acts on scalar periodic functions with zero mean, such that $\norm{A w}_{L^2(\TT^2)} \approx \norm{w}_{\dot{H}^a(\TT^2)}$ for some 
$a \in \RR$. Let $u$ be an extremum of the corresponding action functional, such that $v = Au$ obeys the hydrodynamical equation 
\eqref{eq:hydro}.
In analogy with Conjecture~\ref{conj:Onsager}, it is natural to determine the Onsager exponent $\alpha_O$, such that solutions with regularity 
above $\alpha_O$ conserve the Hamiltonian $\HH = \int_{\TT^2} u \cdot v dx$, while solutions with regularity below $\alpha_O$ do not.
Upon rewriting the nonlinear term in \eqref{eq:hydro} as $u \cdot \nabla v - (\nabla v)^T \cdot u = u^\perp (\nabla^\perp \cdot A u)$, taking an inner product with a mollified version of $u$ and integrating over $\TT^2$, an argument similar to~\cite{CoETi1994} shows that for $v = A u \in L^3_t C^{\alpha}_x$, with $\alpha > \alpha_O =: -a + (1+a)/3$, the Hamiltonian is conserved.    If indeed this choice of $\alpha_O$ determines
an Onsager dichotomy remains to be shown.
\end{remark}

\begin{remark}[An $L^3_{t,x}$ based Banach scale]
Naturally, the value of the Onsager exponent $\alpha_O$ discussed in Conjecture~\ref{conj:Onsager} and Remark~\ref{rem:Onsager} depends on the precise Banach scale $X^\alpha$ considered. Above, we have only mentioned the scale of H\"older spaces $X^\alpha = C_t C^\alpha_x$. On the other hand the Hamiltonian $\HH$ is quadratic in $u$, and the nonlinear term in \eqref{eq:hydro} is also quadratic in $u$, so that proving the conservation in time of $\HH$ for the Euler~\cite{CoETi1994,ChCoFrSh2008} and SQG equations~\cite{IsVi2015} only requires control of the solution in the  Banach scale $X^\alpha = L^3_t B^{\alpha}_{3,\infty}$, with $\alpha > \alpha_O$. Thus, Conjecture~\ref{conj:Onsager} may be alternatively posed on this $L^3$-based Banach scale, without changing the value of $\alpha_O$. It is however conceivable that for an Onsager regularity threshold $\alpha_O$ defined in terms of an $L^2$-based Banach scale, such as $X^\alpha = L^2_t \dot{H}^\alpha_x$, the sharp value may be different from the one discussed in Remark~\ref{rem:Onsager}, which is computed in terms of $L^\infty_{t,x}$ or $L^3_{t,x}$. 
\end{remark} 
 
\begin{remark}[Other important threshold exponents] 
\label{rem:Klainerman}
In a recent survey article on the work of J.~Nash~\cite{Kl2016},  other {\em threshold exponents} are discussed
for which a dichotomy in the behavior of solutions holds, depending on whether the regularity index of the weak solution is greater than or less
than this exponent. 
For simplicity, fix the Banach scale $X^\alpha = C_t C^\alpha_x$.  In analogy to the {\em Onsager exponent} $\alpha_O$, we define
the following important regularity exponents: the
 {\em Nash exponent} $\alpha_N$ determines whether the nonlinear evolution is flexible or rigid (in the sense that $h$-principles are available); 
 the {\em uniqueness exponent} $\alpha_U$ determines the uniqueness of solutions; 
 the {\em well-posedness exponent} $\alpha_{WP}$ determines the local well-posedness of the system; 
 and the {\em scaling exponent} $\alpha_*$ which determines the space $X^{\alpha_*}$ whose norm is invariant under the natural 
 scaling symmetries of the equation (see Page 11 in~\cite{Kl2016}). 
 For instance, in the case of the Euler equations with  the H\"older scale $ C_t C^\alpha_x$ for the velocity field $u$, we have that 
 $\alpha_{WP} = 1$ (cf.~\cite{Ho1933,BaTi2010,ElMa2014,BoLi2015}), $\alpha_U$ is also conjectured to be equal to $1$ (only $\alpha_U \leq 1$ is known), $\alpha_O = 1/3$ (cf.~\cite{CoETi1994,ChCoFrSh2008,Is2016}), and $\alpha_* = 0$, and $\alpha_O \leq \alpha_N \leq \alpha_U$ (since the convex integration constructions also prove $h$-principles and nonuniqueness). We note that these exponents are not the same,
 and  one expects them to be linearly ordered $\alpha_* \leq \alpha_O \leq \alpha_N \leq \alpha_U \leq \alpha_{WP}$ (cf.~\cite[Equation (0.7)]{Kl2016}). 
\end{remark}

For the inviscid SQG equation, on the H\"older scale $ C_t C^\alpha_x$ for the {\em potential velocity field} $v = \Lambda^{-1} u$, this gap between the exponents remains, and one may conjecture that $\alpha_O = 1$, while $\alpha_{WP} = 2$. However, in view of Theorem~\ref{thm:main:dissipative}, for the dissipative SQG equation with $ \gamma =1$ (informally called the {\em critical SQG equation}), one may conjecture that all the exponents discussed in~Remark~\ref{rem:Klainerman} are the same, thereby justifying the adjective ``critical''. 

\begin{conjecture}[Exponents for critical SQG]
\label{conj:critical:SQG}
Consider the dissipative SQG equation~\eqref{eq:dissipative:SQG} with $\gamma=1$, and fix the Banach scale $X^\alpha = C_t C^\alpha_x$ as a way to measure the regularity of the potential velocity $v = \Lambda^{-1} u = \Lambda^{-1} \RSZ^\perp \theta$. Then $1 = \alpha_* = \alpha_O  = \alpha_U = \alpha_{WP}$.
\end{conjecture}
That $\alpha_*=1$ follows from the fact that the $L^\infty$ norm is scaling invariant. 
The fact that $\alpha_{WP}, \alpha_{U} \leq 1$ follows e.g. from~\cite{CaVa2010}, while Theorem~\ref{thm:main:dissipative} shows that 
$\alpha_O, \alpha_U \geq 4/5$. Establishing the remaining inequalities in Conjecture~\ref{conj:critical:SQG} remains open. The above conjecture 
provides the first nonlinear hydrodynamical PDE for which the exponents of Remark~\ref{rem:Klainerman} are all the same.

\section{Outline of the proof}
 
\subsection{The SQG momentum equation}  
\label{sec:SQG:momentum}
We shall make use of two velocity fields to describe the SQG equations: 
we define the {\em potential velocity}
\begin{align} \label{eq:pv}
v = \Lambda^{-1} u  \,,
\end{align}
which is thus one derivative smoother than the SQG {\em transport velocity} $u = \RSZ^\perp \theta$.  
From \eqref{eq:hydro}, it follows that
the potential velocity $v$ satisfies
\begin{subequations}
\label{eq:SQG}
\begin{align}
\partial_t v + u \cdot \nabla v - ( \nabla v)^T\cdot u & = - \nabla p,  \label{eq:SQG-a} \\
\div v & = 0, \label{eq:SQG-b} \\
u & = \Lambda v \label{eq:SQG-c} \,,
\end{align}
\end{subequations}
where $p = \tilde p + u \cdot v$.
The SQG momentum equation \eqref{eq:SQG-a} can be equivalently written as\footnote{
There is yet another form of the SQG equations which should play an important role in the analysis of smooth solutions, which we write as
\[
\p_t u + (u\cdot \nabla )u -\Lambda \left( [ \Lambda^{-1} , u^\perp]\, \nabla ^\perp\cdot u \right) =  -\nabla P  \,, \qquad 
 \operatorname{div} u = 0 \,.
\] 
This form of SQG is written as a zeroth-order perturbation of the incompressible Euler equations for 
sufficiently smooth vector-fields $u$.}
$$
\p_t v + u^\perp \, ( \nabla ^\perp \cdot v) = -\nabla p \,.
$$
Upon defining the temperature function $\theta$ as minus the vorticity of the potential velocity:
$$
\theta = -  \nabla ^\perp \cdot v \,,
$$
\eqref{eq:SQG-a} becomes
\begin{align}\label{eq:SQG-p-theta}
\partial_t v - \theta u^\perp = - \nabla p.
\end{align}
A direct computation confirms that $\theta$ is indeed a solution of (\ref{eq:SQG-old}); 
taking the scalar product of $\nabla ^\perp$ with \eqref{eq:SQG-p-theta}, we find that
$\partial_t (-\theta) - \nabla^\perp \cdot (\theta u^\perp) = 0$ and hence $\partial_t \theta + u \cdot \nabla \theta = 0$, 
since $\nabla^\perp \cdot u^\perp = - \nabla \cdot u = 0$. Note that the dissipative SQG equation \eqref{eq:dissipative:SQG} also can be written as in \eqref{eq:SQG}, by adding $\Lambda^\gamma v$ to the right side of \eqref{eq:SQG-a}.

As we are primarily interested in  weak solutions of SQG, we shall need some basic commutator identities.    For 
 all test functions $\phi \in C^ \infty ( \mathbb{T}  ^2)$, we have that
\begin{align*} 
-  \int_{\TT^2} \partial_i v^j \Lambda v^j \partial_i \phi  dx & = \int_{\TT} \left( v^j  \Lambda \partial_i v^j \partial_i \phi + v \cdot  \Lambda v  \Delta \phi\right)\, dx \,,
\end{align*} 
and thus,
\begin{align*} 
-  \int_{\TT^2} \partial_i v^j \Lambda v^j \partial_i \phi  dx & = \frac 12 \int_{\TT^2} \partial_i v^j \left[ \Lambda, \partial_i \phi \right] v^j\, dx \, + \frac 12 \int_{\TT^2} v \cdot \Lambda v \Delta \phi dx.
\end{align*} 
This motivates a  convenient and equivalent definition of a  weak solution, which is clearly equivalent to Definition~\ref{def:weak:solution} above.
\begin{definition}[Weak solution of SQG, momentum form]
\label{def:weak:sol:momentum}
We say that $v \in L^2_{\rm loc}(\RR;H^{1/2}(\TT^2))$ is a  weak solution of \eqref{eq:SQG} if
\begin{align*}
\int_{\RR} \langle v^i , \partial_t \phi^i  \rangle + \langle \Lambda v^j , v^i \partial_j \phi^i \rangle - \frac 12 \langle \partial_i v^j , [\Lambda , \phi^i] v^j \rangle \, dt = 0
\end{align*}
holds for any $\phi \in C_0^\infty(\RR\times \TT^2)$ such that $\div \phi = 0$. Here $\langle \cdot , \cdot \rangle$ denotes the $\dot{H}^{-1/2}$-$\dot{H}^{1/2}$ duality pairing.  
\end{definition}
Note that for smooth $\phi$ the operator $[\Lambda , \phi^i]$ is a zeroth order operator on $v$ (cf.~Lemma~\ref{lem:Calderon} below). Moreover, adding the term $\int_{\RR} \langle v^i, \Lambda^\gamma \phi^i\rangle \, dt$ to Definition~\ref{def:weak:sol:momentum} gives an equivalent form of Definition~\ref{def:dissipative:weak:solution}.

\begin{remark}[Pressure]
By taking the divergence of \eqref{eq:SQG-a} and using \eqref{eq:SQG-b}, we obtain that  $p$ solves
\begin{equation}\label{eq:pressure}
- \Delta p = \text{Tr}\left( \nabla v \ \nabla u -  \nabla v ^T \ \nabla u  \right) - \Delta v \cdot u \,.
\end{equation} 
For weak solutions, we must interpret $p$ as a distribution, and
the elliptic equation \eqref{eq:pressure} has the following distributional formulation:  for all test functions $\phi \in C^ \infty ( \mathbb{T}  ^2)$,
\begin{equation}\label{eq:pressure-weak0}
\int_{\TT^2}  p \ \Delta  \phi \, dx =  \int_{ \TT^2} \left(   v^i  \Lambda v^j \partial_{ij}^2 \phi - \partial_i v^j \Lambda v^j \partial_i \phi  \right)\, dx \,.
\end{equation} 
It follows that
\begin{equation}
\label{eq:pressure-weak}
\int_{\TT^2}  p \ \Delta  \phi \, dx =  \int_{ \TT^2} \left(   v^i  \Lambda v^j \partial_{ij}^2 \phi  +  \frac 12 v \cdot  \Lambda v  \Delta \phi  
+ \partial_i v^j \left[\Lambda \, , \, \partial_i \phi \right] v^j     \right)\, dx \,.
\end{equation} 
Therefore, given $v \in L^2(\RR;\dot{H}^{1/2}(\TT^2))$ formula \eqref{eq:pressure-weak} defines $p$ (and therefore also $\nabla p$) as a distribution on $\TT^2$
via
\begin{align*}
\langle p , \Delta \phi \rangle
&= \langle  \Lambda v^j , v^i \partial_{ij}^2 \phi \rangle +  \frac 12 \langle \Lambda v^i , v^i  \Delta \phi  \rangle
+ \langle \partial_i v^j  , \left[\Lambda \, , \, \partial_i \phi \right] v^j \rangle.
\end{align*}
In particular, for $v \in C^{1/2-}_t C^{4/5-}_{x}$, the $\nabla p$ term in \eqref{eq:SQG-a} is a well-defined distribution.
\end{remark}

\subsection{The main result and the main ideas of the proof} \label{sec:results}

Employing the potential velocity formulation of SQG, we will prove the following theorem which is easily seen to imply Theorem \ref{thm:main:inviscid} and Theorem \ref{thm:main:dissipative} (we use the convention that $\gamma=0$ is the inviscid SQG equation):
\begin{theorem}[Nonuniqueness of weak solutions, momentum form]
\label{thm:main}
Suppose $\ee:[0,T]\rightarrow \RR^+$ is a smooth function with compact support. Then for every $1/2<\beta<4/5$, $0\leq \gamma < 2-\beta$ and $\sigma<\beta/(2-\beta)$, there exists a weak solution $v \in C_t^{\sigma}C_x^{\beta}$ satisfying
\begin{align*}
\int_{\TT^2}\abs{\Lambda^{ \sfrac 12}v(x,t)}^2\,dx= \int_{\TT^2}\Lambda v(x,t) \cdot v(x,t) \,dx =\ee(t)\,
\end{align*}
for all $t \in [0,T]$.
\end{theorem}

The proof will employ a convex integration scheme,  similar in style to that presented in \cite{BuDLeIsSz2015} (cf.\  \cite{DLSz2012b,DLSz2013}). In \cite{BuDLeIsSz2015}, highly oscillatory Beltrami waves formed the principle building block in the construction. It was noted in \cite{ChDeSz2012,Ch2013} that Beltrami waves can be replaced by Beltrami plane waves (see Section \ref{sec:Beltrami_plane_waves}) in order to prove analogous results for the 2-D Euler equations. Such Beltrami planes waves form a large class of stationary solutions to both the 2-D Euler and the inviscid SQG equation; as such, they will form the principle building block in the construction presented here. As a side remark, we note that it is not difficult to see from the analysis in the present paper that the results in \cite{BuDLeIsSz2015,Is2013,Bu2015,BuDLeSz2016} for the 3-D Euler equations can be extended to the 2-D setting by replacing Beltrami waves by Beltrami plane waves.

The fundamental aim in any convex integration scheme is to introduce high frequency oscillations that self-interact due to the nonlinearity in order to produce low frequency modes that cancel error terms. For SQG this so called high-high-low interaction is highly nontrivial. 
Indeed, as was already noted in~\cite{IsVi2015}, if one works with the usual formulation of the SQG equations \eqref{eq:SQG-old} in terms of the active scalar $\theta$, and considers a perturbation $\Theta = \sum_k \Theta_k$, with $\Theta_k(x,t) = a_k(t,x) e^{i \lambda k \cdot x}$ and an amplitude $a_k$ which lives at a frequency much smaller than $\lambda$,  with $\lambda \gg 1$, then to leading order the corresponding velocity field is given by $U = \sum_k U_k$, where $U_k (x,t)= \RSZ^\perp \Theta_k(x,t) = i k^\perp \Theta_k(x,t) + o(\lambda^{-1})$. This implies that the high-high interactions in the nonlinear term $\div(\Theta U) = \sfrac 12 \div(\sum_k \Theta_k U_{-k} + \Theta_{-k} U_k)$ vanish to leading order, since $\Theta_k U_{-k}  + \Theta_{-k} U_k = a_k^2 ( i (-k)^\perp + i k^\perp) + o(\lambda^{-1}) = o(\lambda^{-1})$.\footnote{
This is the reason why in \cite{IsVi2015}, one may only consider active scalar equations with non-odd constitutive laws (that is, the Fourier multiplier relating $u$ to $\theta$ in \eqref{eq:SQG-old} is a non-odd function of frequency), extending prior results in~\cite{CoFaGa2011,Sh2011,Sz2012} for the incompressible porous media equation. See also the recent work~\cite{CaCoFa2016}.
}
To overcome this difficulty we work at the level of the momentum equation for $v$ and employ a bilinear pseudo-product operator (see~\cite{CoMe1978} or Appendix~\ref{app:pseudo:product} below) to rewrite the nonlinearity $u \cdot \nabla v - ( \nabla v)^T\cdot u$ as the sum of a divergence of a $2$-tensor, and a gradient of a scalar function. 
Expanding in frequency around our Beltrami plane waves, we will show that the principal term in high-high-low interactions  is of the correct form to cancel low frequency modes.  This is achieved in Section~\ref{sec:osc} below.

We remark that recently, \cite{Is2016} proved the full Onsager's conjecture for the 3-D Euler equations, employing  a novel technique
involving gluing exact solutions to the Euler equations, along with the use of Mikado flows, introduced in \cite{DaSz2016} as a replacement to 
Beltrami waves. Mikado flows have the advantage of satisfying better oscillation-error estimates (see Section \ref{sec:inductive_step} for 
the definition of the oscillation error in the case of SQG), since they have disjoint spatial support in a thin cylinder. 
Unfortunately, the construction is inherently three-dimensional as it requires that the Mikado flows do not intersect, which is impossible in 2-D. 
Finding a suitable replacement for Mikado flows for the case of the SQG equations or the 2-D Euler equations is an interesting open problem.

\subsection{Notation}
Throughout this paper,  we make use of  the Einstein summation convention, in which  repeated indices are summed from $1$ to $2$. 
For $s \in \RR$, the {homogeneous} Sobolev space norm is 
$\|u\|^2_{\dot H^{s}(  \mathbb{T}^2 )} = \sum_{k \in \mathbb{Z}^2  \setminus \{0\}} |\hat u(k)|^2  |k|^{2s} \,.$ Here it is important that we work with functions of zero mean on $\TT^2$. The fractional Laplacian $\Lambda^s$ may be defined in this context as the Fourier multiplier with symbol $|k|^s$, for all $s\in \RR$.

For a function $f:  \mathbb{T}^2\times \RR \to \mathbb{R}  $, we use the notation 
\[
\|f\|_{C^\beta} \quad \mbox{to denote the space-time norm} \quad \|f\|_{C^0(\RR;C^\beta (  \mathbb{T} ^2 ))}
\]
where the spatial H\"older norm is defined as the sum of the $C^0$ norm and the H\"older seminorm $[\cdot]_{C^\beta}$.
To distinguish functions with higher regularity in time, we use the norm $\norm{f}_{C_t^{\sigma}C_x^{\beta}} = \norm{f}_{C^\sigma(\RR;C^\beta(\TT^2))}$. That is, $t\mapsto f(x,t)$ has $\sigma$-H\"older regularity in time and $x \mapsto f(x,t)$ has $\beta$-H\"older regularity in space. For $f\colon \TT^2\to \RR$ which is just a function of the space variable $x$, we denote by abuse of notation $\norm{f}_{C^0}$ its $C^0(\TT^2)$ norm.

Throughout the manuscript we abuse notation and denote by $f$ the periodic extension to all of $\RR^2$ of a $\TT^2$-periodic function $f$. Consequently, throughout the proof we work with $\RR^2$ convolution kernels and $\RR^2$ Fourier multiplier operators, instead of working with their $\TT^2$, respectively $\ZZ^2$ counterparts, which are obtained via the Poisson summation formula from their $\RR^2$ analogues. See e.g.~\cite[pp. 256--261]{CaZy1954}, or \cite[Chapter VII]{StWe1971} for the main ideas behind this transference principle. In particular, since we work with functions on $\TT^2$ which have zero mean, the resulting $\RR^2$ functions have support in frequency in the complement of a small neighborhood of the origin.

We will use $a \lesssim b$ to denote  $a \leq C b$ for a universal constant $C\geq 1$. Moreover, for an integer $N \geq 1$ we will use $D^N$ to denote any spatial derivative $\partial_x^\alpha$, where $|\alpha|=N$.

\section{Convex integration scheme}
\label{sec:convex-integration}

We use a convex integration scheme inspired by~\cite{BuDLeIsSz2015}. We shall construct a sequence of solutions $(v_q,p_q,\mathring R_q)$ to the {\em relaxed SQG momentum equation}  
\begin{subequations}
\begin{align}
\partial_t v_q + u_q \cdot \nabla v_q - (\nabla v_q)^T \cdot u_q + \nabla p_q + \Lambda^\gamma v_q &= \div \mathring R_q \\
\div v_q &=0 \\
u_q &= \Lambda v_q
\end{align}
\label{eq:relaxed-SQG}%
\end{subequations}
where  $\mathring R_q$ is a symmetric trace-free $2 \times 2$ matrix. The goal is to obtain $\mathring R_q \to 0$ as $q\to \infty$ (in a suitable topology), and show that a limiting function $v_q \to v$ exists, and solves \eqref{eq:SQG}.

\subsection{Parameters}
We fix $\beta >1/2$ to be the H\"older exponent that we expect to obtain for our weak solution $v$, and write it as
\begin{align*}
\beta = \frac{4}{5} - \eps
\end{align*}
for some $0< \eps  \ll 1$. 
For this $\eps>0$ fixed, we also define 
\begin{align*}
0 \leq \gamma <  2 - \beta 
\end{align*}
to be the power of the dissipation in the equation. When $\gamma=0$ it is understood that the equation is inviscid, i.e. that the dissipative term $\Lambda^\gamma$ is absent from the equations.

Define the frequency parameter
\begin{align*}
\lambda_q =  \lambda_0^q  
\end{align*}
for some integer $\lambda_0 \gg 1$ that is sufficiently large integer which is a multiple of $5$. 
Note thus that the spatial frequency, i.e. wavenumber, parameter $\lambda_q$ is strictly increasing in $q$ and grows  exponentially.
We also define the amplitude parameter
\begin{align}
\delta_q = \lambda_0^2\lambda_q^{-2\beta} \, .
\label{eq:delta-q-def}
\end{align}
 
\subsection{Inductive assumption}\label{s:inductive_assump}

We shall inductively assume that the potential velocity $v_q$ has compact support in frequency, contained inside the ball of radius $2\lambda_q$ and has size
\begin{align}
\|v_q\|_{C^1}+\|u_q\|_{C^0} \leq C_0 \delta_q^{1/2}\lambda_q.
\label{eq:ind:q:1}
\end{align}
where $C_0\geq 1$ is a  universal constant, independent of any of the other parameters in the construction.
Similarly, we shall inductively assume that $ \mathring R_q$ has compact support in frequency, inside the ball $\{ \xi \colon  |\xi| \leq 4\lambda_q\}$,  and has amplitude given by
\begin{align}
\| \mathring R_q\|_{C^0} \leq \eps_R \lambda_{q+1} \delta_{q+1}
\label{eq:ind:q:2}
\end{align}
holds, where $\ee(t)$ is the prescribed energy profile and $\eps_R$ is a small constant to be chosen precisely in the construction.
We also make the inductive  assumption that material derivatives for $w_q$ and $\mathring  R_q$ are bounded as
\begin{align}
\norm{ (\partial_t + u_q \cdot \nabla) v_q}_{C^0} &\leq  C_0 \lambda_q^2 \delta_q 
\label{eq:ind:q:3}
\\
\norm{ (\partial_t + u_q \cdot \nabla) u_q}_{C^0} &\leq  C_0 \lambda_q^3 \delta_q 
\label{eq:ind:q:3:b}
\\
\norm{ (\partial_t + u_q \cdot \nabla) \mathring R_q}_{C^0} &\leq  \lambda_q^2 \delta_q^{1/2} \lambda_{q+1} \delta_{q+1}.
\label{eq:ind:q:4}
\end{align}
Here $C_0$ is the same as in \eqref{eq:ind:q:1}. 
Additionally, we assume that for the given prescribed energy profile
\begin{equation}
\label{eq:energy_ind}
0 \leq \ee(t) - \int_{\TT^2}\abs{\Lambda^{\sfrac12} v_q}^2~dx\leq  \lambda_{q+1}\delta_{q+1}
\end{equation}
and 
\begin{equation}\label{eq:zero_reynolds}
\ee(t) - \int_{\TT^2}\abs{\Lambda^{\sfrac12} v_q}^2~dx\leq  \frac{\lambda_{q+1}\delta_{q+1}}{8}~\Rightarrow \mathring R_q(\cdot,t)\equiv 0\,.
\end{equation}

\subsection{Inductive step}\label{sec:inductive_step}
The convex integration scheme consists of correcting the potential velocity $v_q$ with an increment $w_{q+1}$ and an associated transport velocity increment $\Lambda w_{q+1}$ and obtain new velocity fields
\begin{align}
v_{q+1} = w_{q+1} + v_{q} \qquad \mbox{and} \qquad 
u_{q+1} = \Lambda w_{q+1} + u_q
\label{eq:velocity:increment}
\end{align}
such that the following holds:
\begin{proposition}[Main Proposition]
\label{prop:main}
Let $\ee:[0,T]\rightarrow \RR^+$ be a given smooth Hamiltonian profile.
Then, for sufficiently large $\lambda_0\in5 \NN$,  if the pair $(v_q,\mathring R_q)$ satisfy assumptions \eqref{eq:ind:q:1}--\eqref{eq:zero_reynolds} specified above then there exists a new pair $(v_{q+1}, \mathring R_{q+1})$ that satisfy these assumptions with $q$ replaced by $q+1$. Moreover, the difference $w_{q+1}=v_{q+1}-v_q$ has frequency support contained in the annulus  $\{ \xi \colon \lambda_q/2 \leq |\xi| \leq 2\lambda_q\}$ and has size
\begin{equation}\label{eq:shell_est}
\norm{w_{q+1}}_{C^0} \leq C_0 \delta_{q+1}^{\sfrac 12}
\end{equation}
for a fixed universal constant $C_0 \geq 0$.
\end{proposition}
We note here that if $(v_q,\mathring{R}_q)$ solves \eqref{eq:relaxed-SQG} at step $q$, and the new velocity $v_{q+1}$ is given by \eqref{eq:velocity:increment}, then upon implicitly\footnote{Equation \eqref{eq:R:new:split} only defines $\div \mathring{R}_{q+1}$. The Reynolds stress $\mathring R_{q+1}$ itself is obtained from \eqref{eq:R:new:split} once we invert the divergence operator, cf.~Definition~\ref{def:BB} below, for the contributions that have large frequency, and we write the low frequency part of the oscillation error in divergence form.}  defining $\mathring R_{q+1}$ by
\begin{align}
\div \mathring R_{q+1}
&= \Big( \partial_t w_{q+1} + u_q \cdot \nabla w_{q+1} \Big) \notag\\
&\qquad + \Big( \Lambda w_{q+1} \cdot \nabla v_q - (\nabla v_q)^T \cdot \Lambda w_{q+1} - (\nabla u_q)^T \cdot w_{q+1}  \Big) \notag\\
&\qquad + \Lambda^\gamma w_{q+1} \notag \\
&\qquad  + \left( \div \mathring R_q + \Lambda w_{q+1} \cdot \nabla w_{q+1} - (\nabla w_{q+1})^T \cdot \Lambda w_{q+1} \right)  \notag\\
&\qquad + \nabla \tilde p_{q+1}
 \notag \\
&=: \div R_{T} + \div R_{N} + \div R_{D} + \div R_{O} + \nabla \tilde p_{q+1}
\label{eq:R:new:split} 
\end{align}
we have that $(v_{q+1},\mathring{R}_{q+1})$ solves \eqref{eq:relaxed-SQG} at step $q+1$. In \eqref{eq:R:new:split} we have split up the Reynolds stress into a Transport, Nash, Dissipation, and Oscillation part, and have denoted by $\tilde p_{q+1}$ a dummy scalar pressure (which is different from the $p_{q+1}$ pressure in equation \eqref{eq:relaxed-SQG}). Note that once $w_{q+1}$ is constructed to have frequency support inside the annulus $\{ \lambda_q/2 \leq |\xi|\leq 2 \lambda_q\}$, it follows from \eqref{eq:R:new:split} and the inductive assumptions on the frequency support of $v_q$ and $\mathring R_q$, that $\mathring R_{q+1}$ has frequency support inside the ball $\{ |\xi|\leq 4 \lambda_{q+1}\}$.

The proof of Proposition~\ref{prop:main} is the main part of the paper, and is achieved in three steps. 
The first step, achieved in Section~\ref{sec:perturbation}, is to construct the velocity increment $w_{q+1}$ which obeys the estimate \eqref{eq:shell_est}, and verify that with this perturbation the bounds \eqref{eq:ind:q:1} and \eqref{eq:ind:q:3}--\eqref{eq:ind:q:3:b}  hold with $q$ replaced with $q+1$. The second step, achieved in Section~\ref{sec:stress}, is to show that the induced Reynolds stress $\mathring{R}_{q+1}$ given by \eqref{eq:R:new:split} obeys estimates \eqref{eq:ind:q:2} and \eqref{eq:ind:q:4} with $q$ replaced with $q+1$. The third step, achieved in Section~\ref{sec:hamiltonian}, is to show that the new velocity field is sufficiently close to the desired Hamiltonian profile, i.e. that bounds \eqref{eq:energy_ind}--\eqref{eq:zero_reynolds} hold with $q$ replaced with $q+1$. Together, these three steps give the proof of the proposition.

Theorem \ref{thm:main} is simple consequence of Proposition~\ref{prop:main}, as we show next.

\subsection{Proof of Theorem~\ref{thm:main}}
\begin{proof}[Proof of Theorem \ref{thm:main}]
We start the iteration by setting $(v_{0}, \mathring R_{0})$ to be the trivial zero solution.  Then \eqref{eq:ind:q:1}-\eqref{eq:ind:q:4} and \eqref{eq:zero_reynolds} follow trivially. Moreover, choosing $\lambda_0$ sufficiently large we can ensure
\begin{align*}
\ee(t)\leq  \lambda_1\delta_1=\lambda_0^{3-2\beta}\,,
\end{align*}
and thus \eqref{eq:energy_ind} holds. Then, we apply Proposition \ref{prop:main} iteratively to obtain a sequence $v_q$ converging in $C^{\beta}$ to a  weak solution 
\[
v = v_0 + \sum_{q\geq 0} (v_{q+1} - v_q) = v_0 + \sum_{q\geq 0} w_{q+1}
\]
of SQG. The convergence in $C^{\beta}$ follows directly from the frequency support of the perturbations $w_{q}$ and the estimate \eqref{eq:shell_est}. Moreover, the estimate \eqref{eq:energy_ind} implies that
\begin{align*}
\int_{\TT^2}\abs{\Lambda^{\sfrac 12}v(x,t)}^2\,dx=\ee(t)\,.
\end{align*}

Note, as a consequence of  \eqref{eq:ind:q:3} and  \eqref{eq:ind:q:1}, it follows that
\begin{align*}
\norm{\partial_t v_q}_{C^0}\leq &\norm{(\partial_t+u_q\cdot \nabla)v_q}_{C^0}+\norm{u_q}_{C^0}\norm{v_q}_{C_1}\\
\lesssim & \lambda_q^2\delta_q\,.
\end{align*}
Thus, by interpolation, using the decomposition $w_q=v_{q}-v_{q-1}$ we obtain
\begin{align*}
\norm{w_q}_{C_t^{\sigma}C_x^0}\lesssim &\norm{w_q}_{C_t^{0}C_x^0}^{1-\sigma}\norm{v_q-v_{q-1}}_{C_t^{1}C_x^0}^{\sigma}\\
\lesssim &\norm{w_q}_{C_t^{0}C_x^0}^{1-\sigma}\left(\norm{v_q}_{C_t^{1}C_x^0} + \norm{v_{q-1}}_{C_t^{1}C_x^0}\right)^{\sigma}\\
\lesssim & \delta_q^{\frac{1-\sigma}{2}} \left(\lambda_q^2\delta_q\right)^{\sigma}\\
=&\lambda_0^2\lambda_q^{-(1-\sigma)\beta +2\sigma(1-\beta)}\\
=&\lambda_0^2\lambda_q^{-\beta+\sigma(2-\beta)}\,
\end{align*}
Hence, if $\sigma<\frac{\beta}{2-\beta}$, then $v_q$ convergences uniformly in ${C_t^{\sigma}C_x^0}$.
\end{proof}

\section{The velocity perturbation}
\label{sec:perturbation}
\subsection{Technical preliminaries}
\subsubsection{Inverse of the divergence}
In defining $R_T, R_D, R_O$, and $R_N$, we need to use the fact that any divergence free vector function $f$ with zero mean on $\TT^2$ may be written as a divergence. More precisely:
\begin{definition}[Inverse divergence]
\label{def:BB}
Let $f$ be divergence free and with zero mean on $\TT^2$. Then we have
\begin{align*}
f = \div (\BB f),  \quad \mbox{or in components} \quad   f^i = \partial_j (\BB f)^{ij}
\end{align*}
where
\begin{align*}
(\BB f)^{ij} := - \partial_j \Lambda^{-2}   f^i - \partial_i  \Lambda^{-2} f^j.
\end{align*}
For $f$ which is not necessarily divergence free, we define
\begin{align*}
\BB f := \BB \PP f\,,
\end{align*}
where $\PP = \Id + \RSZ\otimes \RSZ$ is the Leray projector. Lastly, when $f$ does not have zero mean on $\TT^2$, we define 
\begin{align*}
\BB f := \BB \left(f - \frac{1}{|\TT^2|}\int_{\TT^2} f(x) dx \right).
\end{align*}
In particular, we have that $\div(\BB f) = \PP f$ is divergence free, and $\BB f$ is a symmetric trace free matrix. Properties of the operator $\BB$ are discussed in Appendix~\ref{sec:BB} below.
\end{definition}

\subsubsection{Beltrami plane waves}\label{sec:Beltrami_plane_waves}
For $k \in \SS^1$, we define 
\begin{align}
b_k(\xi) = i k^\perp e^{i k \cdot \xi} \qquad \mbox{and} \qquad c_k(\xi) = e^{i k \cdot \xi}
\label{eq:bk:ck:def}
\end{align}
where we notice that since $k \in \SS^1$, we have
\begin{align*}
b_k = \nabla_\xi^\perp c_k \qquad \mbox{and} \qquad c_k = - \nabla_\xi^\perp \cdot b_k.
\end{align*}
It is also worth noting here that  $b_k$ is an eigenfunction of $\Lambda$ with eigenvalue $1$, that is
\begin{align}
\Lambda_\xi b_k(\xi) = b_k(\xi)\,,
\label{eq:bk-eigenvalue}
\end{align}
since $k \in \SS^1$. Also it will be sometimes useful to note that since $(k^\perp)^\perp =  - k$, we have
\begin{align}
(b_k(\xi))^\perp = - i k c_k(\xi)  = - \nabla_\xi c_k(\xi).
\label{eq:bk-perp}
\end{align}

\subsubsection{Geometric Lemma}
For any finite family of vectors $\Omega\subset \SS^1$ and constants $a_k\in \CC$, such that $a_{-k} = \overline{a_k}$, if we set
\begin{align*}
W(\xi):=\sum_{k\in\Omega}a_k b_k(\xi)
\qquad \mbox{and}\qquad 
V(\xi):=\sum_{k\in\Omega}a_k c_k(\xi)\,,
\end{align*}
then we have the following identity
\begin{equation}\label{eq:Beltrami_identity}
\begin{split}
\div_{\xi}(W\otimes W)=&\frac12 \nabla_{\xi}\abs{W}^2+(\nabla^{\perp}_{\xi}\cdot W)W^{\perp}\\
=&\frac12 \nabla_{\xi}\abs{W}^2-V\nabla V\\
=&\frac12 \nabla_{\xi}\left(\abs{W}^2-\abs{V}^2\right).
\end{split}
\end{equation}
Adopting the notation $W_k(\xi)=a_k b_k(\xi)$, we also note that
\begin{align*}
\sum_{k\in \Omega} W_k\otimes W_{-k}=
\sum_{k\in \Omega} \abs{a_k}^2 k^{\perp}\otimes k^{\perp}\,.
\end{align*}

\begin{lemma}
\label{lem:split}
Let $B_{\eps}(\Id)$ denote the ball of symmetric $2\times 2$ matrices, centered at $\Id$ of radius $\eps$. We can choose $\eps_{\gamma}>0$  such that there exist  disjoint finite subsets 
$$
\Omega_j\subset \SS^1, \qquad j\in \{1, 2\}~,
$$
and smooth positive functions 
\[
\gamma_k\in C^{\infty}\left(B_{\eps_{\gamma}} (\Id)\right), \qquad j\in \{1,2\}~, \qquad k\in\Omega_j~,
\]
such that
\begin{itemize}
\item[\bf{(a)}] For each $j$ we have $5 \Omega_j\subset \ZZ^2$.
\item[\bf{(b)}] If $k\in \Omega_j$ then $-k\in \Omega_j$ and $\gamma_k = \gamma_{-k}$.
\item[\bf{(c)}] For each $R\in B_{\eps_\gamma} (\Id)$ we have the identity 
\begin{equation}\label{e:split}
R = \frac{1}{2} \sum_{k\in\Omega_j} \left(\gamma_k(R)\right)^2 (k^{\perp}\otimes k^{\perp}),
\end{equation}
for all $R\in B_{\eps_{\gamma}}(\Id)$.
\item[\bf{(d)}] For $k, k' \in \Omega_j$, with $k+k' \neq 0$, we have that {$|k+k'| \geq \frac 12$}.
\end{itemize}
\end{lemma}
\begin{proof}[Proof of Lemma~\ref{lem:split}]
First consider the case of constructing $\Omega_1$.  Define $\Omega_1^+=\{k_1,k_2,k_3\}$ where $k_1:=(1,0)$, $k_2:=(\sfrac 35,\sfrac 45)$ and $k_3:=(\sfrac 35,-\sfrac45)$. With these choices we make the following observations.
First, the matrices $k_i^{\perp}\otimes k_i^{\perp}$ for $i=1,2,3$ are linearly independent.
Second, we have the identity
\begin{equation*}
\begin{split}
\Id=&\left(1- \frac{5^2\cdot3^2}{ 4^2\cdot 5^2}\right)k_1\otimes k_1+\frac{5^2}{2\cdot 4^2}k_2\otimes k_2+\frac{5^2}{2\cdot 4^2} k_3\otimes k_3\\
=&\frac{7}{16}k_1\otimes k_1+\frac{25}{32}k_2\otimes k_2+\frac{25}{32} k_3\otimes k_3\\
=&\frac{7}{16}k_1^{\perp}\otimes k_1^{\perp}+\frac{25}{32}k_2^{\perp}\otimes k_2^{\perp}+\frac{25}{32} k_3^{\perp}\otimes k_3^{\perp}.
\end{split}
\end{equation*}
Hence, setting $\Omega_1=-\Omega_1^+\cup \Omega_1^+$, and applying the inverse function theorem to construct $\gamma_k$ we obtain properties (a)--(d). Similarly, setting $\Omega_2=\Omega_1^{\perp}= \{k^\perp \colon k \in \Omega_1\}$, we may construct $\gamma_k$ for $k\in \Omega_2$ and obtain properties (a)--(d).
\end{proof}

\subsubsection{Time cutoffs and the back-to-labels map}
We let $0 \leq \chi \leq 1$ be a smooth cutoff function which is identically $1$ on $[1,2]$,  vanishes on the complement of $[1/2,4]$, and defines a partition of unity according to 
\begin{align*}
\sum_{j \in \ZZ} \chi^2(t - j) = 1
\end{align*}
for all $t \in \RR$. We shall also define
\begin{align}
\chi_{j}(t) = \chi(t \tau_{q+1}^{-1} - j)
\label{eq:chi:q+1:j:def}
\end{align}
where for ease of notation we suppress the dependence of $\chi_j$ on $q$.

For every $j \in \ZZ$, we define the following back-to-labels map $\Phi_j(x,t)$ by solving the transport equation
\begin{align*}
\left( \partial_t + u_q \cdot \nabla \right) \Phi_j &=  0 \, ,\\
\Phi_j(x,j \tau_{q+1}) &= x \, ,
\end{align*}
where we define  the time step parameter $\tau_{q+1}$  by
\begin{align}
\tau_{q+1}^{-1} = \lambda_{q} \lambda_{q+1} \delta_{q}^{1/4} \delta_{q+1}^{1/4} \, .
\label{eq:tau:q}
\end{align}
The motivation for this scaling of $\tau_{q+1}$ comes from balancing the oscillation and the transport error (cf.~estimates \eqref{eq:sharp:1} and \eqref{eq:sharp:2} below).
In particular, we note that 
\begin{align}
\tau_{q+1} \norm{\nabla u_q}_{C_0} \leq \tau_{q+1} \lambda_q^2 \delta_q^{1/2}=\lambda_0^{-1+\beta/2},
\label{eq:CFL}
\end{align}
so that since $\beta<2$, then for $\lambda_0$ large, on a time interval of length $2 \tau_{q+1}$, the flow $\Phi_j$ induced by $u_q$ does not depart substantially from the identity.

\subsubsection{Leray projector and a frequency localizer}
Lastly, we define $\PP = \Id + \RSZ \otimes \RSZ$ to be the Leray projector, and for $k \in \SS^1$ and $\lambda_{q+1}$ as above we set
\begin{align}
\PP_{q+1,k} = \PP P_{\approx k \lambda_{q+1}}
\label{eq:PP:q+1:k:def}
\end{align}
where $P_{\approx \lambda_{q+1} k}$ is a zero order Fourier multiplier operator with symbol $ \hat K_{\approx k \lambda_{q+1}} = \hat K_{\approx 1} (\xi/\lambda_{q+1} - k)$. That is, 
\begin{align*} 
(P_{\approx k \lambda_{q+1}} f)^{\hat \ }(\xi)  =  \hat K_{\approx k \lambda_{q+1}}(\xi) \hat f(\xi) = \hat{K}_{\approx 1}\left( \frac{\xi}{\lambda_{q+1}} - k \right) \hat{f}(\xi),
\end{align*} 
where the function $\hat K_{\approx 1}$ is a smooth bump function supported on the ball $\{ \xi \colon |\xi|\leq \frac{1}{8}\}$, and such that  $\hat K_{\approx 1}(\xi) = 1$ on the smaller ball $ \left\{ \xi \colon   |\xi| \leq  \frac{1}{16} \right\}$.
Note in particular that $0 \leq a,b$ we have 
\[
\sup_{\xi \in \RR^2} |\xi|^a |\nabla_{\xi}^b \hat K_{\approx k \lambda_{q+1}} | \leq C_{a,b} \lambda_{q+1}^{a-b}
\] 
for a suitable constant $C_{a,b}$ that is independent of $\lambda_{q+1}$, and similarly
\begin{align*}
\norm{ |x|^{b} \nabla_x^a K_{\approx k \lambda_{q+1} }  }_{L^1_x(\RR^2)} \leq C_{a,b} \lambda_{q+1}^{a-b}
\end{align*}
holds for $0 \leq a,b \leq 2$, and the constant $C_{a,b}$ is independent of $\lambda_{q+1}$. 

We note here that  $\PP_{q+1,k}$ is a convolution operator with kernel $K_{q+1,k}(x)$, i.e., for $f$ which is $\TT^2$-periodic we may write
\[
\PP_{q+1,k} f(x) = \int_{\RR^2} K_{q+1,k}(y) f(x-y) dy\,,
\]
with a kernel $K_{q+1,k}$ that obeys
\begin{align}
\norm{ |x|^{b} \nabla_x^a K_{q+1,k}(x) }_{L^1_x(\RR^2)} \leq C_{a,b} \lambda_{q+1}^{a-b}
\label{eq:projection:kernel}
\end{align}
for $0 \leq a,b$, and the constant $C_{a,b}$ is independent of $\lambda_{q+1}$.
Here we have implicitly used that $\RSZ \otimes \RSZ = \PP - \Id$ is a matrix zero order Fourier multiplier, whose symbol is smooth away from the origin.

\subsection{Construction of the perturbation}
\label{sec:w:q+1:construction}

With these notations in hand, we now define the potential velocity perturbation $w_{q+1}$ as
\begin{align}
w_{q+1}(x,t) = \sum_{j \in \ZZ,k \in \Omega_{j}} \chi_{j}(t) \PP_{q+1,k} \big( a_{k,j}(x,t) b_k( \lambda_{q+1} \Phi_j(x,t) ) \big)\,,
\label{eq:w:q+1:def}
\end{align}
where the functions $a_{k,j}(x,t)$ are to be defined in \eqref{eq:ak:def} below, and $\Omega_j =  \Omega_1$ if $j$ is odd, while $\Omega_j = \Omega_2$ if $j$ is even. The definition of $\PP_{q+1,k}$ implies that the increment $w_{q+1}$ has compact support in frequency space inside $\{ \xi \colon \lambda_{q+1}/2 \leq |\xi|\leq 2 \lambda_{q+1}\}$, as required in the inductive step.

Let $\mathring R_{q,j}$ define the solution to the transport equation:
\begin{subequations}
\label{eq:R:q:transported}  
\begin{align}
\left( \partial_t + u_q \cdot \nabla \right) \mathring R_{q,j} &=  0 \, ,\\
\mathring R_{q,j}(x,j \tau_{q+1}) &= \mathring R_q(x,j \tau_{q+1})  \, ,
\end{align}
\end{subequations}
and set
\begin{equation}
R_{q,j}:= {\lambda_{q+1}} \rho_j\Id - \mathring R_{q,j}
\label{eq:R:q:def}
\end{equation}
where 
\begin{equation}\label{eq:rho_def}
\rho(t) := \frac{1}{(2\pi)^2}\min\left(\ee(t) - \int_{\TT^2}\abs{\Lambda^{\sfrac12} v_q}^2~dx-\frac{\lambda_{q+
2}\delta_{q+2}}{2},\, 0\right)~\mbox{and}~
\rho_j = \rho(\tau_{q+1}j)\,.
\end{equation} 
The constants $\rho_j$ are chosen in order to ensure convergence of the Hamiltonian to the desired profile. Note that by the inductive assumption~\eqref{eq:energy_ind} we have that
\begin{align*}
\rho_j  \leq \delta_{q+1}\,.
\end{align*}
Then to conclude our definition of the perturbation $w_{q+1}$, we define
\begin{align}
a_{k,j}(x,t):= \begin{cases}
\rho_j^{\sfrac12}\gamma_k\left(\frac{ R_{q,j}(x,t)}{{\lambda_{q+1}} \rho_j}\right)  &\mbox{if } \rho_j\neq 0 \\ 
0 & \mbox{if } \rho_j=0.
\end{cases}
\label{eq:ak:def}
\end{align}
Note that in order that $a_{k,j}$ is well defined we need to ensure that if $\rho_j\neq 0$ then
\[\frac{ R_{q,j}(x,t)}{{\lambda_{q+1}} \rho_j}\in B_{\eps_{\gamma}(\Id)}\,. \]
Since $R_{q,j}$ satisfies a transport equation it suffices to prove that  
\begin{equation}\label{eq:rho_cond}
\frac{\norm{\mathring R_{q}(\cdot,j\tau_{q+1})}_{C^0)}}{{\lambda_{q+1}} \rho_j}\leq \eps_{\gamma}\,. 
\end{equation}
By \eqref{eq:zero_reynolds} is suffices to consider the case when
\[e(j\tau_{q+1}) - \int_{\TT^2}\abs{\Lambda^{\sfrac12} v_q(x,j\tau_{q+1})}^2~dx\geq \frac{\lambda_{q+1}\delta_{q+1}}{8}\,,\]
which implies (assuming $\lambda_0$ is chosen sufficiently large)
\[\rho_j\geq \frac{\lambda_q\delta_{q+1}}{8}.\]
Applying \eqref{eq:ind:q:2} yields
\begin{align*}
\frac{\|\mathring R_q(\cdot, j\tau_{q+1})\|_{C^0} }{\lambda_{q+1}\rho_j} &\leq 8\eps_R
\end{align*}
Then as long as $2\eps_R\leq \eps_{\gamma}$ we obtain \eqref{eq:rho_cond}.

Throughout the paper it will be sometimes convenient to denote
\begin{align}
\tilde w_{q+1,j,k} = \chi_{j}(t)  a_{k,j}(x,t) b_k( \lambda_{q+1} \Phi_j(x,t) ) 
\label{eq:tilde:w:q+1:def}
\end{align}
so that in adopting the abuse of notation $\sum_{j ,k}=\sum_{\{j:\rho_j\neq 0\},k \in \Omega_{j}}$, the equation \eqref{eq:w:q+1:def} reads
\begin{align*}
w_{q+1}(x,t) = \sum_{j ,k} \PP_{q+1,k} \tilde w_{q+1,j,k}\,.
\end{align*}
We also adopt the notation
\begin{align}
\psi_{q+1,j,k}(x,t)= \frac{c_k(\lambda_{q+1} \Phi_j(x,t))}{c_k(\lambda_{q+1} x)} = e^{i \lambda_{q+1} (\Phi_j(x,t) - x) \cdot k}
\label{eq:psi:q+1:j:k:def}
\end{align}
so that 
\begin{align*}
b_k( \lambda_{q+1} \Phi_j(x,t) ) = b_k(\lambda_{q+1} x) \psi_{q+1,j,k}(x,t)~.
\end{align*}

\subsection{Bounds on the perturbation}
\begin{lemma}
\label{lem:w:q+1:bounds}
With $w_{q+1}$ as defined in \eqref{eq:w:q+1:def}, we have that
\begin{subequations}
\begin{align}
\norm{w_{q+1}}_{C^0} &\leq C_0 \delta_{q+1}^{1/2} 
\label{eq:w:q+1}\\
\norm{v_{q+1}}_{C^1} +\norm{u_{q+1}}_{C^0} &\leq C_0 \delta_{q+1}^{1/2} \lambda_{q+1}
\label{eq:v+u:q+1}\\
\norm{D_{t,q} w_{q+1}}_{C^0} &\leq C_0 \tau_{q+1}^{-1} \delta_{q+1}^{1/2}
\label{eq:Dt:w:q+1}\\
\norm{D_{t,q+1} v_{q+1}}_{C^0} &\leq C_0 \lambda_{q+1}^2 \delta_{q+1} 
\label{eq:Dt:v:q+1}\\
\norm{D_{t,q+1} u_{q+1}}_{C^0} &\leq C_0 \lambda_{q+1}^3 \delta_{q+1} 
\label{eq:Dt:u:q+1}
\end{align}
\end{subequations}
for a universal constant $C_0 \geq 1$, which is the same as the constant in Section~\ref{s:inductive_assump}.
\end{lemma}
In particular, the bounds \eqref{eq:v+u:q+1}, \eqref{eq:Dt:v:q+1}, and \eqref{eq:Dt:u:q+1} show that the inductive estimates \eqref{eq:ind:q:1}, \eqref{eq:ind:q:3}, and \eqref{eq:ind:q:3:b} hold with $q$ replaced with $q+1$.

\begin{proof}[Proof of Lemma~\ref{lem:w:q+1:bounds}]
From \eqref{eq:projection:kernel} it follows that $\norm{\PP_{q+1,k}}_{C^0 \to C^0} \lesssim 1$, and hence
\begin{align*}
\norm{w_{q+1}}_{C^0} 
&\lesssim \sum_{j,k} \norm{\tilde w_{q+1,j,k}} \notag\\
&\lesssim \sum_{j,k} {\bf 1}_{\supp (\chi_{j})}  \norm{a_{k,j}}_{C^0} \lesssim \sum_{j} {\bf 1}_{\supp (\chi_{j})}  \rho_j^{1/2}  \lesssim \delta_{q+1}^{1/2}
\end{align*}
in view of \eqref{eq:ak:def}, and the fact that $\rho_j \lesssim \delta_{q+1}^{1/2}$. Estimate \eqref{eq:v+u:q+1} follows from \eqref{eq:w:q+1}, the frequency support of $w_{q+1}$, and the inductive estimates \eqref{eq:ind:q:1}.

In order to estimate the material derivative $D_{t,q} = \partial_t + u_q \cdot\nabla$ of $w_{q+1}$, we note that  $D_{t,q} a_{k,j}(x,t) = 0$ since $\mathring R_{q,j}$ obey the transport equation for the vector field $u_q$, and similarly $D_{t,q} b_k(\lambda_{q+1} \Phi_j(x,t)) = 0$ since the phases $\Phi_j$ also obey the same transport equation.
We thus have that
\begin{align*}
D_{t,q} w_{q+1}  
&= D_{t,q} \left( \sum_{j,k} \PP_{q+1,k} \tilde w_{q+1,j,k} \right) \notag\\
&= \sum_{j,k} \PP_{q+1,k} \big( (\partial_t \chi_{j} ) a_{k,j} b_k(\lambda_{q+1} \Phi_j ) \big) + \sum_{j,k} [D_{t,q},\PP_{q+1,k}] \tilde w_{q+1,j,k} .
\end{align*}
Therefore, by appealing to the boundedness on $C^0$ of $\PP_{q+1,k}$, the definition of $\chi_{j}$ in \eqref{eq:chi:q+1:j:def}, and the commutator estimate \eqref{eq:comm:material:convolution} in Corollary~\ref{cor:Dt:commutator} with $s=0$, we arrive at
\begin{align*}
\norm{D_{t,q} w_{q+1}}_{C^0}
&\lesssim \tau_{q+1}^{-1} \sum_{j,k} {\bf 1}_{\supp (\chi_{j})} \norm{a_{k,j}}_{C^0} + \sum_{j,k} \norm{\nabla u_q}_{C^0} \norm{\tilde w_{q+1,j,k}}_{C^0} \notag\\
&\lesssim \tau_{q+1}^{-1} \sum_j {\bf 1}_{\supp (\chi_{j})}  \rho_j^{1/2} + \lambda_q^2 \delta_q^{1/2} \sum_j {\bf 1}_{\supp (\chi_{j})}  \rho_j^{1/2} \notag\\
&\lesssim \left(\tau_{q+1}^{-1} + \lambda_q^2 \delta_q^{1/2} \right) \delta_{q+1}^{1/2} \notag\\
&\lesssim \tau_{q+1}^{-1} \delta_{q+1}^{1/2}
\end{align*}
since $\rho_j \lesssim \delta_{q+1} $ and  by \eqref{eq:tau:q} we have $\tau_{q+1}^{-1} = \lambda_{q} \lambda_{q+1} \delta_{q}^{1/4} \delta_{q+1}^{1/4} \geq \lambda_q^{2} \delta_{q}^{1/2}$, where here we used the fact that $\beta<2$.

Using the inductive estimate \eqref{eq:ind:q:3} and the bound established above, since
\begin{align*}
D_{t,q+1} v_{q+1} = D_{t,q} (v_q + w_{q+1}) + \Lambda w_{q+1} \cdot \nabla (v_q + w_{q+1})
\end{align*}
we obtain
\begin{align*}
 \norm{D_{t,q+1} v_{q+1}}_{C^0} \lesssim \lambda_q^2 \delta_q + \tau_{q+1}^{-1} \delta_{q+1}^{1/2} + \lambda_{q+1}^2 \delta_{q+1} \lesssim \lambda_{q+1}^2 \delta_{q+1}.
\end{align*}

Similarly,
\begin{align*}
D_{t,q+1} u_{q+1} = D_{t,q} (u_q + \tilde P_{\approx \lambda_{q+1}} \Lambda w_{q+1}) + \Lambda w_{q+1} \cdot \nabla (u_q + \Lambda w_{q+1} ), 
\end{align*}
the inductive estimate \eqref{eq:ind:q:3:b}, and Lemma~\ref{lem:u:grad:commutator} (with $s=1$ and $\lambda = \lambda_{q+1}$) implies
\begin{align*}
 \norm{D_{t,q+1} u_{q+1}}_{C^0} 
 &\lesssim \lambda_q^3 \delta_q + \norm{\tilde P_{\approx \lambda_{q+1}} \Lambda (D_{t,q} w_{q+1})}_{C^0} + \norm{[u_q \cdot \nabla,\tilde P_{\approx \lambda_{q+1}}  \Lambda] w_{q+1}}_{C^0} + \lambda_{q+1}^3 \delta_{q+1}\\
 &\lesssim \lambda_q^3 \delta_q + \lambda_{q+1} \tau_{q+1}^{-1} \delta_{q+1}^{1/2} + \lambda_{q+1} \lambda_q^2 \delta_{q}^{1/2} \delta_{q+1}^{1/2} + \lambda_{q+1}^3 \delta_{q+1}
 \notag\\
&\lesssim \lambda_{q+1}^3 \delta_{q+1}
\end{align*}
which concludes the proof of the lemma.
\end{proof}

In estimating the Reynolds stress error, the following bounds concerning the derivatives of $a_{k,j}$ and $\psi_{q+1,j,k}$ are very useful.
\begin{lemma}
\label{lem:D:N:a:psi}
With $a_{k,j}$ as defined in \eqref{eq:ak:def} and $\psi_{q+1,j,k}$ as defined in \eqref{eq:psi:q+1:j:k:def}, we have that 
\begin{align}
\norm{D^N a_{k,j}}_{C^0(\supp \chi_{j})} &\lesssim \lambda_{q}^N \delta_{q+1}^{1/2}
\label{eq:D:N:ak}
\end{align}
for all $N \geq 0$, and 
\begin{align}
\norm{D^N \psi_{q+1,j,k}}_{C^0(\supp \chi_{j})} &\lesssim  \left(\tau_{q+1} \lambda_{q+1} \lambda_q^2 \delta_q^{1/2}\right)^N = \left(\lambda_q \delta_q^{1/4} \delta_{q+1}^{-1/4}\right)^N
\label{eq:D:N:psi:j:k}
\end{align}
for all $N \geq 1$, where the implied constant depends on $N$.
\end{lemma}

\begin{proof}[Proof of Lemma~\ref{lem:D:N:a:psi}] 
Since $R_q$ is supported on frequencies less than $4 \lambda_{q}$, it follows from the chain rule estimate~\eqref{eq:chain} and the smoothness of the $\gamma_k$ functions that
\begin{align*}
\norm{D^N  a_{k,j}}_{C^0} 
&\lesssim \lambda_{q+1}^{-1} \rho_j^{-1/2}  \norm{D^N R_{q,j}}_{C^0} + \rho_j^{1/2-N} \lambda_{q+1}^{-N}  \norm{D R_{q+1,j}}_{C^0}^N\\
&\lesssim \delta_{q+1}^{ 1/2} \lambda_q^N   + \delta_{q+1}^{1/2} \lambda_{q+1}^{-N}  (\lambda_q  \lambda_{q+1})^N
\end{align*}
where we have also used that $\rho_j \lesssim \delta_{q+1}$.  This proves \eqref{eq:D:N:ak}.

In order to prove \eqref{eq:D:N:psi:j:k}, we appeal to the chain rule estimate~\eqref{eq:chain} and the transport estimates~\eqref{eq:Dphi:near:id}--\eqref{eq:Dphi:N} and find that
\begin{align*}
 \norm{D^N \psi_{q+1,j,k}}_{C^0(\supp \chi_{j})} 
 &\quad \lesssim \lambda_{q+1} \norm{D^{N-1}\left( D\Phi_j-\Id\right)}_{C^0(\supp \chi_{j})} + \lambda_{q+1}^{N} \norm{D \Phi_j  - \Id }_{C^0(\supp \chi_{j})}^N \\
 &\quad \lesssim \lambda_{q+1} \tau_{q+1} \norm{D^N u_q}_{C^0} e^{C \tau_{q+1} \|D u_q\|_{C^0}}
 + \lambda_{q+1}^{N} \left( \tau_{q+1} \norm{D u_q}_{C^0} e^{C \tau_{q+1} \|D u_q\|_{C^0}} \right)^N
\end{align*}
for a suitable constant $C$.
From \eqref{eq:CFL}, $\tau_{q+1} \norm{D u_q}_{C^0} \leq 1$, and thus we conclude that
\begin{align*}
\norm{D^N \psi_{q+1,j,k}}_{C^0(\supp \chi_{j})}
&\lesssim \lambda_{q+1} \tau_{q+1} \lambda_q^{N+1} \delta_q^{1/2} + \lambda_{q+1}^{N} \left( \tau_{q+1}  \lambda_q^2 \delta_q^{1/2} \right)^N \notag\\
&\lesssim \lambda_q^N \delta_q^{1/4} \delta_{q+1}^{-1/4} + \lambda_q^N \left(\delta_q^{1/4} \delta_{q+1}^{-1/4}\right)^N \notag\\
&\lesssim  \left(\lambda_q \delta_q^{1/4} \delta_{q+1}^{-1/4}\right)^N
\end{align*}
since $\delta_{q+1}\leq \delta_q$. This estimate shows that $\psi_{q+1,j,k}$ lives at spatial frequency $\lambda_q \delta_q^{1/4} \delta_{q+1}^{-1/4}$. 
\end{proof}

\section{The Reynolds stress error}
\label{sec:stress}

\subsection{Transport error}
\begin{lemma}[Transport error]
\label{lem:transp:error} For any $\eps>0$, if $\lambda_0$ is sufficiently large then for $R_T$ as defined in \eqref{eq:R:new:split}, we have that
\begin{align*}
\norm{R_{T}}_{C^0} &\leq \eps \lambda_{q+2} \delta_{q+2} \\
\norm{D_{t,q} R_T}_{C^0} &\leq \eps \lambda_{q+1}^2 \delta_{q+1}^{1/2} \lambda_{q+2} \delta_{q+2}\,.
\end{align*}
\end{lemma}

\subsubsection{Amplitude of the Transport error}

By definition, we have that
\begin{align}
R_{T}&=\BB \left( D_{t,q} w_{q+1} \right)
\notag \\
&=\BB\left(\sum_{j,k} \PP_{q+1,k} \big( (\partial_t \chi_{j} ) a_{k,j} b_k(\lambda_{q+1} \Phi_j ) \big) + \sum_{j,k} [D_{t,q},\PP_{q+1,k}] \tilde w_{q+1,j,k}\right) 
\notag\\
&=\BB \tilde P_{\approx \lambda_{q+1}} \left(\sum_{j,k} \PP_{q+1,k} \big( (\partial_t \chi_{j} ) a_{k,j} b_k(\lambda_{q+1} \Phi_j ) \big) + \sum_{j,k} [u_q \cdot \nabla,\PP_{q+1,k}] \tilde w_{q+1,j,k}\right) 
\end{align}
where in the last equality we used the compact support of $u_q$ to conclude that $u_q \cdot \nabla \PP_{q+1,k} = \tilde P_{\approx \lambda_{q+1}} \left( u_q \cdot \nabla \PP_{q+1,k}\right)$.
Using Lemma~\ref{lem:u:grad:commutator} we obtain
\begin{align}
 \norm{R_T}_{C^0} 
 &\lesssim \lambda_{q+1}^{-1} \sum_{j,k} \norm{(\partial_t \chi_{j} ) a_{k,j} b_k(\lambda_{q+1} \Phi_j )}_{C^0} + \lambda_{q+1}^{-1} \sum_{j,k} \norm{\nabla u_q}_{C^0} \norm{\chi_j a_{k,j} b_k(\lambda_{q+1} \Phi_j )}_{C^0} \notag\\
 &\lesssim \lambda_{q+1}^{-1} \tau_{q+1}^{-1} \delta_{q+1}^{1/2} + \lambda_{q+1}^{-1} \lambda_q^{2} \delta_q^{1/2} \delta_{q+1}^{1/2} \notag\\
 &\lesssim \lambda_{q+1}^{-1} \tau_{q+1}^{-1} \delta_{q+1}^{1/2}\notag\\
   &= \lambda_q \delta_q^{1/4} \delta_{q+1}^{3/4} \notag\\
 &=  \lambda_0^{-2+\frac{5\beta}{2}}\lambda_{q+2} \delta_{q+2}.
 \label{eq:sharp:1}
\end{align}
Assuming $\lambda_0$ is sufficiently large and $\beta<\sfrac 45$ we obtain our claim.

\subsubsection{Material derivative of the Transport error}

Using the frequency support of $u_q$ and $w_{q+1}$ we write
\begin{align}
D_{t,q}R_{T}
&=[D_{t , q}, \BB \tilde P_{\approx_{\lambda_{q+1}}}] D_{t,q} w_{q+1} +\BB \tilde P_{\approx_{\lambda_{q+1}}}D_{t,q}\left(\sum_{j,k} \PP_{q+1,k} \big( (\partial_t \chi_{j} ) a_{k,j} b_k(\lambda_{q+1} \Phi_j ) \big)\right) \notag\\
&\quad   +\BB \tilde P_{\approx_{\lambda_{q+1}}}D_{t,q} \left( \sum_{j,k} [D_{t,q},\PP_{q+1,k}] \tilde w_{q+1,j,k}\right) \notag \\
&=T_1+T_2+T_3
\label{eq:RT:Dt}
\end{align}

The first term in \eqref{eq:RT:Dt} is bounded directly using Corollary~\ref{cor:Dt:commutator} and the bound \eqref{eq:Dt:w:q+1} as
\begin{align}
\norm{T_1}_{C^0} 
&\lesssim \lambda_{q+1}^{-1} \norm{\nabla u_q}_{C^0} \norm{D_{t,q} w_{q+1}}_{C^0} \notag\\
&\lesssim \lambda_{q+1}^{-1} \lambda_q^2 \delta_q^{1/2} \tau_{q+1}^{-1} \delta_{q+1}^{1/2} \notag\\
&= \lambda_q^3 \delta_{q+1}^{3/4} \delta_q^{3/4}. \label{eq:T_1_est}
\end{align}

We decompose the second term in \eqref{eq:RT:Dt} as
\begin{align*}
T_2=
&\BB \tilde P_{\approx_{\lambda_{q+1}}} \left(\sum_{j,k}[D_{t,q},\PP_{q+1,k}]\big( (\partial_t \chi_{j} ) a_{k,j} b_k(\lambda_{q+1} \Phi_j ) \big)+\sum_{j,k}\PP_{q+1,k}(\partial_t^2 \chi_{j} ) a_{k,j} b_k(\lambda_{q+1} \Phi_j )\right)
\end{align*}
which allows it to be estimated by
\begin{align}
\norm{T_2}_{C^0} 
&\lesssim  \lambda_{q+1}^{-1} \sum_{j,k} \left( \norm{\nabla u_q}_{C^0} \norm{(\partial_t \chi_{j} ) a_{k,j} b_k(\lambda_{q+1} \Phi_j )}_{C^0} + \norm{(\partial_t^2 \chi_{j} ) a_{k,j} b_k(\lambda_{q+1} \Phi_j )}_{C^0} \right) \notag\\
&\lesssim \lambda_{q+1}^{-1} \left( \lambda_q^2 \delta_{q}^{1/2} \tau_{q+1}^{-1} \delta_{q+1}^{1/2} + \tau_{q+1}^{-2} \delta_{q+1}^{1/2} \right) \notag\\
& \lesssim \lambda_{q+1}^{-1}   \tau_{q+1}^{-2} \delta_{q+1}^{1/2}  
\notag\\
&= \lambda_q^2\lambda_{q+1} \delta_{q}^{1/2} \delta_{q+1}.
\label{eq:T_2_est}
\end{align}
Next,
\begin{align*}
T_3
&=\BB \tilde P_{\approx_{\lambda_{q+1}}} D_{t,q} \left( \sum_{j,k}u_q\cdot \nabla\PP_{q+1,k} \tilde w_{q+1,j,k}-\sum_{j,k}\PP_{q+1,k} \left(u_q\cdot \nabla\tilde w_{q+1,j,k}\right)\right)
\notag \\
&=\BB\tilde P_{\approx_{\lambda_{q+1}}} 
\left(\sum_{j,k}(D_{t,q}u_q)\cdot \nabla\left(\PP_{q+1,k} \tilde w_{q+1,j,k}\right)+\sum_{j,k} u_q\cdot \nabla  \left([D_{t,q},\PP_{q+1,k}] \tilde w_{q+1,j,k}\right)\right)
\notag \\
&\quad+\BB\tilde P_{\approx_{\lambda_{q+1}}}
\left(\sum_{j,k}u_q\cdot \nabla  \left(\PP_{q+1,k} \big( (\partial_t \chi_{j} ) a_{k,j} b_k(\lambda_{q+1} \Phi_j ) \big)\right) - \sum_{j,k} \left((u_q \cdot \nabla) u_q \right) \cdot \nabla  \PP_{q+1,k} \tilde w_{q+1,j,k}   \right)
\notag \\
&\quad-\BB\tilde P_{\approx_{\lambda_{q+1}}}
\left(\sum_{j,k}[D_{t,q},\PP_{q+1,k}] \left(u_q\cdot \nabla\tilde w_{q+1,j,k}\right)+\sum_{j,k}\PP_{q+1,k}\left((D_{t,q}u_q)\cdot \nabla\tilde w_{q+1,j,k}\right)\right)
\notag \\
&\quad-\BB\tilde P_{\approx_{\lambda_{q+1}}}
\left(\sum_{j,k}\PP_{q+1,k}\left(u_q\cdot \nabla\big( (\partial_t \chi_{j} ) a_{k,j} b_k(\lambda_{q+1} \Phi_j ) \big)\right) - \sum_{j,k} \PP_{q+1,k} \left( \left( (u_q \cdot \nabla) u_q \right) \cdot \nabla \tilde w_{q+1,j,k} \right)\right).
\end{align*}
In order to estimate $T_3$ we appeal to the inductive bound \eqref{eq:ind:q:3:b},  Lemma~\ref{lem:D:N:a:psi},   the commutator estimate in Corollary~\ref{cor:Dt:commutator} (with $s = 0$ and $\lambda = \lambda_{q+1}$), 
and to the fact that $\tau_{q+1} \lambda_{q}^2 \delta_q^{1/2} \lesssim 1$ and the estimate \eqref{eq:Dphi:near:id}, imply 
\[
\norm{\tilde w_{q+1,j,k}}_{C^1} \lesssim \norm{a_{k,j}}_{C^1} + \norm{a_{k,j}}_{C^0} \lambda_{q+1} \norm{\nabla \Phi_j}_{C^0} \lesssim \lambda_q \delta_{q+1}^{1/2} + \lambda_{q+1} \delta_{q+1}^{1/2} \lesssim \lambda_{q+1} \delta_{q+1}^{1/2}.
\] 
All these yield
\begin{align}
\norm{T_3}_{C^0}
&\lesssim 
\lambda_{q+1}^{-1} \sum_{j,k} \left( \norm{D_{t,q} u_q}_{C^0} \lambda_{q+1} \norm{\tilde w_{q+1,j,k}}_{C^0} + \norm{u_q}_{C^0} \lambda_{q+1} \norm{u_q}_{C^1} \norm{\tilde w_{q+1,j,k}}_{C^0} \right)
\notag\\
&\quad 
+ \lambda_{q+1}^{-1} \sum_{j,k}  \left( \norm{u_q}_{C^0} \lambda_{q+1} \tau_{q+1}^{-1} \norm{a_{k,j}}_{C^0(\supp \chi_j)} 
+ \norm{u_q}_{C^0} \norm{u_q}_{C^1} \lambda_{q+1}  \norm{\tilde w_{q+1,j,k}}_{C^0} \right)
\notag\\
&\quad + \lambda_{q+1}^{-1} \sum_{j,k}\left( \norm{u_q}_{C^1} \norm{u_q}_{C^0} \norm{\tilde w_{q+1,j,k}}_{C^1} + \norm{D_{t,q} u_q}_{C^0} \norm{\tilde w_{q+1,j,k}}_{C^1}\right)
\notag\\
&\quad + \lambda_{q+1}^{-1} \sum_{j,k}\left( \norm{u_q}_{C^0} \tau_{q+1}^{-1} \norm{a_{k,j} b_k(\lambda_{q+1} \Phi_j)}_{C^1} + \norm{u_q}_{C^0} \lambda_{q} \norm{u_q}_{C^1} \norm{\tilde w_{q+1,j,k}}_{C^1}\right) \notag
\\
&\lesssim    \lambda_{q}^3 \delta_q   \delta_{q+1}^{1/2} + \lambda_q  \delta_q^{1/2} \tau_{q+1}^{-1} \delta_{q+1}^{1/2}  \notag\\
&\lesssim \lambda_q  \delta_q^{1/2} \tau_{q+1}^{-1} \delta_{q+1}^{1/2} \notag\\
&= \lambda_q^2\lambda_{q+1} \delta_q^{3/4}\delta_{q+1}^{3/4}. \label{eq:T_3_est}
\end{align}

Combining \eqref{eq:T_1_est}, \eqref{eq:T_2_est} and \eqref{eq:T_3_est} we obtain
\begin{align*}
\norm{D_{t,q}R_T}_{C^0}\lesssim &\lambda_q^3\delta_q^{3/4}\delta_{q+1}^{3/4}+\lambda_q^2\lambda_{q+1} \delta_{q}^{1/2}\delta_{q+1}+\lambda_q^2\lambda_{q+1}\delta_q^{3/4}\delta_{q+1}^{3/4}\\
 \lesssim &
\lambda_q^2\lambda_{q+1}\delta_q^{3/4}\delta_{q+1}^{3/4}\\
=&\lambda_0^{-3+7\beta/2}\lambda_{q+1}^{2}\ \delta_{q+1}^{1/2} \lambda_{q+2} \delta_{q+2},.
\end{align*}
This completes the proof of Lemma~\ref{lem:transp:error}.

\subsection{Nash error}

\begin{lemma}[Nash error]
\label{lem:Nash:error}
For any $\eps>0$, if $\lambda_0$ is sufficiently large then for $R_N$ as defined in \eqref{eq:R:new:split}, we have that
\begin{align}
\norm{R_{N}}_{C^0} &\leq \eps \lambda_{q+2} \delta_{q+2} ,\notag \\
\norm{D_{t,q} R_N}_{C^0} &\leq \eps \lambda_{q+1}^2\delta_{q+1}^{1/2} \lambda_{q+2} \delta_{q+2}\,.\label{eq:Dt_Nash}
\end{align}
\end{lemma}

\subsubsection{Amplitude of the Nash error}
We recall from \eqref{eq:R:new:split} and the definition of $\BB$ that we may write
\begin{align}
R_N 
&=- \BB \left( (\nabla u_q)^T \cdot w_{q+1}  \right) + \BB \left(\Lambda w_{q+1} \cdot \nabla v_q - (\nabla v_q)^T \cdot \Lambda w_{q+1}\right) \notag\\
&= - \BB \left( (\nabla u_q)^T \cdot w_{q+1}  \right)  + \BB \left((\nabla^\perp \cdot v_q) \Lambda w_{q+1}^\perp \right) \notag\\
&=: N_1 + N_2.
\label{eq:RN:def}
\end{align}
The bound on $N_1$ is direct. Recalling the definition of $w_{q+1}$, and recalling that $u_q$ is supported on frequencies $|\xi| \leq 2\lambda_q \leq \lambda_{q+1}/8$, upon applying \eqref{eq:BB:L:infty} with $\lambda = \lambda_{q+1}$ (which is inherent in $\PP_{q+1,k}$), we obtain
\begin{align*}
\norm{N_1}_{C^0} 
&\lesssim \frac{\norm{u_q}_{C^1} }{\lambda_{q+1}} \sum_{j,k} \norm{\tilde w_{q+1,j,k}}_{C^0} \notag\\
&\lesssim \frac{\lambda_q^2 \delta_q^{1/2} \delta_{q+1}^{1/2}}{\lambda_{q+1}}.
\end{align*}

It is convenient to rewrite the term $N_2$, exploring the special structure of the perturbations $w_{q+1}$.
We note that since $b_k^\perp(\xi) = i (k^\perp)^\perp e^{i k\cdot \xi} = - i k e^{i k \cdot \xi} = - \nabla_\xi c_k(\xi)$, we may write
\begin{align*}
\Lambda w_{q+1}^\perp 
&= \sum_{j,k} \Lambda \PP_{q+1,k} \big(\chi_j a_{k,j} \psi_{q+1,j,k} b_k(\lambda_{q+1} x)^\perp\big) \notag\\
&= -\frac{1}{\lambda_{q+1}} \sum_{j,k} \Lambda \PP_{q+1,k} \big(\chi_j  a_{k,j} \psi_{q+1,j,k} \nabla c_k(\lambda_{q+1} x)  \big) \notag\\
&= -\frac{1}{\lambda_{q+1}} \nabla \sum_{j,k} \Lambda \PP_{q+1,k} \big(\chi_j  a_{k,j} \psi_{q+1,j,k}c_k(\lambda_{q+1} x)  \big)   + \frac{1}{\lambda_{q+1}}  \sum_{j,k} \Lambda \PP_{q+1,k} \big(\chi_j  \nabla \big( a_{k,j} \psi_{q+1,j,k} \big) c_k(\lambda_{q+1} x)  \big). 
\end{align*}
Therefore, recalling that $\BB$ has incorporated into it the Leray projector $\PP$, we have that
\begin{align}
N_2 
&= \frac{1}{\lambda_{q+1}} \BB\left( \nabla (\nabla^\perp \cdot v_q) \sum_{j,k} \Lambda \PP_{q+1,k} \big(\chi_j  a_{k,j}  c_k(\lambda_{q+1} \Phi_j )  \big) \right)\notag\\
&\quad +\frac{1}{\lambda_{q+1}}  \BB\left( (\nabla^\perp \cdot v_q) \sum_{j,k} \Lambda \PP_{q+1,k} \big( \nabla \big(\chi_j  a_{k,j} \psi_{q+1,j,k} \big) c_k(\lambda_{q+1} x)  \big) \right).
\label{eq:N2:rewriting}
\end{align}
In order to bound $N_2$ we again use  \eqref{eq:BB:L:infty} with $\lambda = \lambda_{q+1}$ and obtain 
\begin{align*}
\norm{N_2}_{C^0}
&\lesssim \frac{1}{\lambda_{q+1}^2} \norm{v_q}_{C^2} \sum_{j,k} \norm{\Lambda P_{\approx \lambda_{q+1}} (\chi_j  a_{k,j} c_k(\lambda_{q+1} \Phi_j))}_{C^0} 
\notag\\
&\quad + \frac{1}{\lambda_{q+1}^2} \norm{v_q}_{C^1} \sum_{j,k} \norm{\Lambda P_{\approx \lambda_{q+1}} 
\Big(\chi_j  \nabla a_{k,j} c_k(\lambda_{q+1} \Phi_j) + \chi_j  a_{k,j} \nabla \psi_{q+1,j,k} c_k(\lambda_{q+1} x) \Big)}_{C^0} \notag\\
&\lesssim \frac{1}{\lambda_{q+1}^2} \lambda_q^2 \delta_{q}^{1/2} \lambda_{q+1} \delta_{q+1}^{1/2}
  + \frac{1}{\lambda_{q+1}^2} \lambda_q \delta_{q}^{1/2} \lambda_{q+1} \left(\lambda_q \delta_{q+1}^{1/2} +  \delta_{q+1}^{1/2}  \lambda_q\delta_q^{1/4}\delta_{q+1}^{-1/4} \right) \\
&\lesssim \frac{\lambda_q^2 \delta_q^{1/2} \delta_{q+1}^{1/2}}{\lambda_{q+1}} + 
\frac{\lambda_q^2}{\lambda_{q+1}} \delta_{q}^{3/4}  \delta_{q+1}^{1/4} 
\end{align*}
where in the second last inequality we have used Lemma~\ref{lem:D:N:a:psi}.  We see that the first terms in the above bound obeys the same estimate as $N_1$. Then again using $\delta_{q+1}<\delta_q$ we obtain
\begin{align*}
\norm{R_N}_{C^0}\lesssim \frac{\lambda_q^2}{\lambda_{q+1}} \delta_{q}^{3/4}  \delta_{q+1}^{1/4} = \lambda_0^{-3+\frac{7 \beta}{2}}\lambda_{q+2} \delta_{q+2}
\end{align*}
Thus we obtain the desired estimate so long as $\beta<\frac 67$ and $\lambda_0$ is sufficiently large.

\subsubsection{Material derivative of the Nash error}
Recall that $u_q$ has frequency support inside of the ball of radius $2 \lambda_q \leq \lambda_{q+1}/8$, and therefore
\begin{align*}
N_1= \BB \left( (\nabla u_q)^T \cdot w_{q+1}  \right) = \BB \tilde P_{\approx \lambda_{q+1}}  \left( (\nabla u_q)^T \cdot w_{q+1}  \right)
\end{align*}
where we denote by $\tilde P_{\approx \lambda_{q+1}}$ the Fourier multiplier operator whose symbol is supported on frequencies $\{\xi \colon \lambda_{q+1}/4 \leq |\xi| \leq 4 \lambda_{q+1}\}$, and is identically $1$ on the annulus $\{\xi \colon 3 \lambda_{q+1}/8 \leq |\xi| \leq 3 \lambda_{q+1} \}$. Therefore, appealing to Corollary~\ref{cor:Dt:commutator} (with $s=-1$ and $\lambda = \lambda_{q+1}$), we have that 
\begin{align*}
\norm{D_{t,q} N_1}_{C^0} 
&\leq \norm{\BB \tilde P_{\approx \lambda_{q+1}}  D_{t,q} \left( (\nabla u_q)^T \cdot w_{q+1} \right)}_{C^0} + \norm{[D_{t,q}, \BB \tilde P_{\approx \lambda_{q+1}}] \left(   (\nabla u_q)^T \cdot w_{q+1} \right)}_{C^0} \notag\\
&\lesssim \frac{1}{\lambda_{q+1}} \norm{D_{t,q} \left( (\nabla u_q)^T \cdot w_{q+1} \right)}_{C^0}
+ \frac{1}{\lambda_{q+1}} \norm{\nabla u_q}_{C^0} \norm{(\nabla u_q)^T \cdot w_{q+1}}_{C^0} \notag\\
&\lesssim  \frac{1}{\lambda_{q+1}} \left( \norm{D_{t,q}  (\nabla u_q)^T}_{C^0} \norm{w_{q+1} }_{C^0} 
+ \norm{\nabla u_q}_{C^0} \norm{D_{t,q} w_{q+1}}_{C^0}
+  \norm{\nabla u_q}_{C^0}^2 \norm{w_{q+1}}_{C^0} \right).
\end{align*}
Using the inductive hypothesis \eqref{eq:ind:q:3:b}, we have that 
\begin{align*}
 \norm{D_{t,q}  (\nabla u_q)}_{C_0} 
 &\lesssim \norm{u_q}_{C^1}^2 + \norm{D_{t,q} u_q}_{C^1} \lesssim \lambda_q^4 \delta_q.
\end{align*}
From Lemma~\ref{lem:w:q+1:bounds},  we conclude that
\begin{align*}
\norm{D_{t,q} N_1}_{C^0} 
&\lesssim \frac{1}{\lambda_{q+1}} 
\left(  \lambda_q^4 \delta_q \delta_{q+1}^{1/2} 
+ \lambda_q^2 \delta_q^{1/2} \tau_{q+1}^{-1} \delta_{q+1}^{1/2} \right)
\\
&\lesssim \lambda_{q+1}^{-1} \lambda_q^2 \delta_q^{1/2} \tau_{q+1}^{-1} \delta_{q+1}^{1/2} \\
&= \lambda_q^3\delta_q^{3/4}\delta_{q+1}^{3/4}\\
&= \lambda_0^{-4+\frac{7\beta}{2}}\lambda_{q+1}^2 \delta_{q+1}^{1/2} \lambda_{q+2} \delta_{q+2} \,.
\end{align*}

In order to estimate the material derivative of $N_2$,  we recall \eqref{eq:N2:rewriting}, and 
as above,  using the compact support of $v_q$, we find that
\begin{align*}
- \lambda_{q+1} D_{t,q} N_2 
&= D_{t,q} \BB \tilde P_{\approx \lambda_{q+1}} \left( \nabla (\nabla^\perp \cdot v_q) \sum_{j,k} \Lambda \PP_{q+1,k} \big(\chi_j  a_{k,j}  c_k(\lambda_{q+1} \Phi_j )  \big) \right)\notag\\
&\quad + D_{t,q} \BB \tilde P_{\approx \lambda_{q+1}}  \left( (\nabla^\perp \cdot v_q) \sum_{j,k} \Lambda \PP_{q+1,k} \big(\chi_j  \nabla \big( a_{k,j} \psi_{q+1,j,k} \big) c_k(\lambda_{q+1} x)  \big) \right) \notag\\
&= \left(\BB \tilde P_{\approx \lambda_{q+1}} D_{t,q} + [D_{t,q} ,\BB \tilde P_{\approx \lambda_{q+1}}] \right)  \left( \nabla (\nabla^\perp \cdot v_q) \sum_{j,k} \Lambda \PP_{q+1,k} \big(\chi_j  a_{k,j}  c_k(\lambda_{q+1} \Phi_j )  \big) \right)\notag\\
&\quad + \left(\BB \tilde P_{\approx \lambda_{q+1}} D_{t,q} + [D_{t,q} ,\BB \tilde P_{\approx \lambda_{q+1}}] \right) \left( (\nabla^\perp \cdot v_q) \sum_{j,k} \Lambda \PP_{q+1,k} \big( \nabla \big(\chi_j  a_{k,j} \psi_{q+1,j,k} \big) c_k(\lambda_{q+1} x)  \big) \right).
\end{align*}
We now appeal to the commutator estimate of Corollary~\ref{cor:Dt:commutator}: first with  $\lambda = \lambda_{q+1}$ and $s=0$ for $[D_{t,q}, \BB \tilde P_{\approx \lambda_{q+1}}]$, second with $\lambda = \lambda_{q+1}$ and $s=1$ for $[D_{t,q}, \Lambda \PP_{q+1,k}]$, and third with $\lambda = \lambda_{q+1}$ and $s=-1$ for $[D_{t,q}, \BB \tilde P_{\approx \lambda_{q+1}}]$. We obtain that
\begin{align*}
\lambda_{q+1} \norm{D_{t,q} N_2}_{C^0} 
&\lesssim  \sum_{k,j} \norm{D_{t,q} \nabla (\nabla^\perp \cdot v_q)}_{C^0}   \norm{\chi_j  a_{k,j}}_{C^0} +   \norm{\nabla (\nabla^\perp \cdot v_q)}_{C^0}  \left(  \norm{\chi_j' a_{k,j}}_{C^0} + \norm{\nabla u_q}_{C^0} \norm{\chi_j a_{k,j}}_{C^0} \right) \notag\\
&\quad + \norm{\nabla u_q}_{C^0}  \norm{\nabla (\nabla^\perp \cdot v_q)}_{C^0} \sum_{k,j} \norm{\chi_j a_{k,j}}_{C^0}  +  \sum_{k,j} \norm{D_{t,q} (\nabla^\perp \cdot v_q)}_{C^0}   \norm{\chi_j  \nabla (a_{k,j} \psi_{q+1,j,k})}_{C^0} \notag\\
&\quad +  \sum_{k,j} \norm{\nabla^\perp \cdot v_q}_{C^0}  \left(  \norm{\chi_j' \nabla (a_{k,j} \psi_{q+1,j,k})}_{C^0} + \norm{\nabla u_q}_{C^0} \norm{\chi_j \nabla(a_{k,j} \psi_{q+1,j,k})}_{C^0} \right) \notag\\
&\quad + \norm{\nabla u_q}_{C^0} \sum_{k,j} \norm{\nabla^\perp v_q}_{C^0} \norm{\chi_j \nabla(a_{k,j} \psi_{q+1,j,k})}_{C^0} \,.
\end{align*}
Using the previously established bounds and the inductive estimates yields
\begin{align*}
 \norm{D_{t,q} N_2}_{C^0}   
&\lesssim \lambda_{q+1}^{-1} \lambda_q^4 \delta_q    \delta_{q+1}^{1/2} + \lambda_{q+1}^{-1}  \lambda_q^2 \delta_q^{1/2}  \left(  \tau_{q+1}^{-1} \delta_{q+1}^{1/2} +  \lambda_{q}^2 \delta_q^{1/2} \delta_{q+1}^{1/2} \right) 
\notag\\
&\quad +  \lambda_{q+1}^{-1} \lambda_q \delta_q^{1/2}  \left( \tau_{q+1}^{-1} + \lambda_q^{2} \delta_q^{1/2}   \right)  \left( \lambda_q \delta_{q+1}^{1/2} + \delta_{q+1}^{1/2} \lambda_{q+1} \tau_{q+1} \lambda_q^2 \delta_{q}^{1/2} \right) 
\notag \\
&\lesssim  \lambda_q^3 \delta_q  \delta_{q+1}^{1/2} \,.
\end{align*}
Combining the above estimates shows that
\begin{align*}
\norm{D_{t,q} R_n}\lesssim  \lambda_q^3\delta_q^{3/4}\delta_{q+1}^{3/4}+\lambda_q^3 \delta_q  \delta_{q+1}^{1/2}
\lesssim \lambda_0^{-4+\frac{7\beta}{2}}\lambda_{q+1}^2 \delta_{q+1}^{1/2} \lambda_{q+2} \delta_{q+2}\,.
\end{align*}
Therefore \eqref{eq:Dt_Nash} holds so long as $\beta< \frac{8}{7}$.

\subsection{Dissipation error}

\begin{lemma}[Dissipation error]
\label{lem:disip:error}
For any $\eps>0$, if $\lambda_0$ and $q$ are sufficiently large then for $R_D$ as defined in \eqref{eq:R:new:split}, we have that
\begin{align*}
\norm{R_{D}}_{C^0} &\leq \eps \lambda_{q+2} \delta_{q+2} \\
\norm{D_{t,q} R_D}_{C^0} &\leq \eps \lambda_{q+1}^2 \delta_{q+1}^{1/2}  \lambda_{q+2} \delta_{q+2}
\end{align*}
\end{lemma}

\subsubsection{Amplitude of the dissipation error}
By definition,
\begin{align}
R_D = \BB \Lambda^\gamma w_{q+1} = \BB \Lambda^\gamma \tilde P_{\approx \lambda_{q+1}} w_{q+1}.
\label{eq:RD:def}
\end{align}
Therefore, it follows from Bernstein's inequality for Fourier multipliers~\cite{Le2002} that 
\begin{align*}
\norm{R_D}_{C^0} 
&\lesssim \lambda_{q+1}^{\gamma-1} \norm{w_{q+1}}_{C^0}  \notag\\
&\lesssim \lambda_{q+1}^{\gamma-1} \delta_{q+1}^{1/2} \notag\\
&\leq \lambda_{0}^{2\beta -2} \lambda_{q+1}^{\gamma + \beta -2 } \lambda_{q+2} \delta_{q+2}.
\end{align*}
Thus we obtain the stated estimate so long as $\gamma < 2-\beta$, $\beta<1$ and $\lambda_0$ is sufficiently large.

\subsubsection{Material derivative of the dissipation error}
The estimate on the material derivative of the dissipation error follows directly from \eqref{eq:RD:def}, the previously established bound \eqref{eq:Dt:w:q+1} for the material derivative of the perturbation, and the commutator estimate of Corollary \eqref{cor:Dt:commutator} in which we set $\lambda = \lambda_{q+1}$, and $s=  \gamma - 1$, which is the order of the Fourier multiplier operator $\BB \Lambda^\gamma \tilde P_{\approx \lambda_{q+1}}$. We obtain that
\begin{align*}
\norm{D_{t,q} R_{D}}_{C^0} 
&\leq \norm{(\BB \Lambda^\gamma \tilde P_{\approx \lambda_{q+1}}) D_{t,q} w_{q+1}}_{C^0} + \norm{ \left[ D_{t,q}, \BB \Lambda^\gamma \tilde P_{\approx \lambda_{q+1}}\right]w_{q+1}}_{C^0}\notag\\
&\lesssim  \lambda_{q+1}^{\gamma-1} \norm{D_{t,q} w_{q+1}}_{C^0} +   \lambda_{q+1}^{\gamma-1} \norm{\nabla u_q}_{C^0} \norm{w_{q+1}}_{C^0} \notag\\
&\lesssim \lambda_{q+1}^{\gamma-1} (\tau_{q+1}^{-1} + \lambda_q^2 \delta_q^{1/2} ) \delta_{q+1}^{1/2} \notag\\
&\lesssim \lambda_q \lambda_{q+1}^\gamma \delta_{q}^{1/4} \delta_{q+1}^{3/4} \notag\\
&= \lambda_0^{\gamma + 7 \beta/2 - 4} \lambda_q^{\gamma+\beta-2}\lambda_{q+1}^{2} \delta_{q+1}^{1/2} \lambda_{q+2} \delta_{q+2}.
\end{align*}
Thus again we obtain the stated estimate if $\gamma < 2-\beta$ and $q$ is sufficiently large.

\subsection{Oscillation error}
\label{sec:osc}
\begin{lemma}[Oscillation error]
\label{lem:osc:error}
For any $\eps>0$, if $\lambda_0$ and $q$ are sufficiently large, then for $R_O$ as defined in \eqref{eq:R:new:split}, we have that
\begin{align*}
\norm{R_{O}}_{C^0} &\leq \eps \lambda_{q+2} \delta_{q+2} \,, \\
\norm{D_{t,q} R_O}_{C^0} &\leq \eps \lambda_{q+1}^2 \delta_{q+1}^{1/2}  \lambda_{q+2} \delta_{q+2} \,.
\end{align*}
\end{lemma}

\subsubsection{Decomposition of the oscillation error}
Recall from \eqref{eq:R:new:split} that the oscillation error $R_O$ is defined so that the following equality is satisfied:
\begin{align}
\div R_O 
&=  \div \mathring R_q + \Lambda w_{q+1} \cdot \nabla w_{q+1} - (\nabla w_{q+1})^T \cdot \Lambda w_{q+1}  \notag\\
&=  \div \left( \sum_j \chi_j^2 ( \mathring R_{q} - \mathring R_{q,j}) \right) + \div \left( \sum_j \chi_j^2( \mathring R_{q,j} + \rho_j \lambda_{q+1} \Id) \right) \notag\\  
&\qquad + \Big( \Lambda w_{q+1} \cdot \nabla w_{q+1} - (\nabla w_{q+1})^T \cdot \Lambda w_{q+1}\Big)  . \label{RO}
\end{align}

\begin{remark}
Note that in the above formula, as well as throughout this paper,  we  somewhat abuse notation and write $\sum_{j}$ to mean
the summation $\sum_{\{j:\rho_j\neq 0\}} $.\footnote{We also denote by $\sum_{ j,k} $ the double sum $\sum_{\{j:\rho_j\neq 0\}} \sum_{k \in \Omega_j}$, and similarly $\sum_{j,j',k,k'}$ denotes a quadruple sum.}
It this then important to note that in view of \eqref{eq:zero_reynolds}, the decomposition \eqref{RO} is valid so long as
\begin{align*}
\ee(t) - \int_{\TT^2}\abs{\Lambda^{\sfrac12} v_q}^2~dx\leq  \frac{\lambda_{q+1}\delta_{q+1}}{8}
\end{align*}
on the support of $\chi_j$ for $\rho_j=0$. The proof of this fact will be delayed to Lemma \ref{lem:rho_diff}, equation \eqref{eq:rho_vanish}.
\end{remark}

Recalling that 
\begin{align*}
w_{q+1}(x,t) = \sum\limits_{\substack{  \{j:\rho_j\neq 0\} \\ k \in \Omega_{j}}}
 \PP_{q+1,k} \tilde w_{q+1,j,k} \qquad \mbox{with} \qquad \tilde w_{q+1,j,k} = \chi_{j}(t)  a_{k,j}(x,t) b_k( \lambda_{q+1} \Phi_j(x,t) )  \,,
\end{align*}
 the term $ \Lambda w_{q+1} \cdot \nabla w_{q+1} - (\nabla w_{q+1})^T \cdot \Lambda w_{q+1}$ in \eqref{RO} has  both high and low
 frequency components, depending  whether  $k \neq - k'$ or $k=-k'$. 
We notice that due to the frequency localization induced by $\PP_{q+1,k}$, for $k,k' \in \Omega_j$ with $k+k' \neq 0$,  we have that $1/2 \leq |k+k'| \leq 2$, and due to the localization in the angular frequency variable, we obtain that  
\begin{align}
\Lambda \PP_{q+1,k} \tilde w_{q+1,j,k} \cdot \nabla \PP_{q+1,k'} \tilde w_{q+1,j',k'} = \tilde P_{\approx \lambda_{q+1}} \left( \Lambda \PP_{q+1,k} \tilde w_{q+1,j,k} \cdot \nabla \PP_{q+1,k'} \tilde w_{q+1,j',k'} \right)
\label{eq:Lambda:angle:local}
\end{align}
and
\begin{align*}
(\nabla \PP_{q+1,k} \tilde w_{q+1,j,k})^T \otimes \Lambda \PP_{q+1,k'} \tilde w_{q+1,j',k'} = \tilde P_{\approx \lambda_{q+1}} \left( (\nabla \PP_{q+1,k} \tilde w_{q+1,j,k})^T \otimes \Lambda \PP_{q+1,k'} \tilde w_{q+1,j',k'} \right).
\end{align*}
We shall thus isolate the 
 high-frequency part of $R_O$ due to the nonlinear interactions in $\Lambda w_{q+1} \cdot \nabla w_{q+1} - (\nabla w_{q+1})^T \cdot \Lambda w_{q+1}$, as
\begin{align}
 R_{O,{\rm high}} &= \BB \tilde P_{\approx \lambda_{q+1}}  \Bigg(\sum_{\underset{k+k'\neq 0 } {j,j',k,k'}} \big(\Lambda \PP_{q+1,k} \tilde w_{q+1,j,k}\big) \cdot \nabla \big(\PP_{q+1,k'} \tilde w_{q+1,j',k'}\big)  \Bigg)  
\notag\\
&\quad -  \BB  \tilde P_{\approx \lambda_{q+1}} \Bigg(\sum_{\underset{k+k'\neq 0 } {j,j',k,k'}} \big(\nabla \PP_{q+1,k} \tilde w_{q+1,j,k}\big)^T \cdot \big( \Lambda  \PP_{q+1,k'} \tilde w_{q+1,j',k'} \big) \Bigg).
\label{eq:R:O:high}
\end{align}
Similarly, we need to isolate the low-frequency part of $\Lambda w_{q+1} \cdot \nabla w_{q+1} - (\nabla w_{q+1})^T \cdot \Lambda w_{q+1}$ 
which occurs when  $k+k'=0$. 
Since $\Omega_1 \cap \Omega_2 = \emptyset$, $k + k' = 0$ implies that $j = j'$, so that upon symmetrizing, we may define the low frequency part of the nonlinear term as
\begin{align}
\TTT_{j,k} &= \frac{1}{2}\Bigg(  
\big(\Lambda \PP_{q+1,k} \tilde w_{q+1,j,k}\big) \cdot \nabla \big(\PP_{q+1,-k} \tilde w_{q+1,j,-k}\big)  
+\big(\Lambda \PP_{q+1,k} \tilde w_{q+1,j,k}\big) \cdot \nabla \big(\PP_{q+1,-k} \tilde w_{q+1,j,-k}\big)  \notag \\
&\quad -  \big(\nabla \PP_{q+1,k} \tilde w_{q+1,j,k}\big)^T \cdot \big( \Lambda  \PP_{q+1,-k} \tilde w_{q+1,j,-k} \big)  
- \big(\nabla \PP_{q+1,-k} \tilde w_{q+1,j,-k}\big)^T \cdot \big( \Lambda  \PP_{q+1,k} \tilde w_{q+1,j,k} \big) \Bigg).
\label{eq:TTT:j:k:def}
\end{align}
The challenge now is to obtain the decomposition
\begin{align}
{\TTT}_{j,k} = \div ({\mathcal Q}_{j,k}) + \nabla {\mathcal P}_{j,k}
\label{eq:R:O:low}
\end{align}
for a suitably defined $2$-tensor ${\mathcal Q}_{j,k}$ which gains one derivative over ${\TTT}_{j,k}$ and obeys good transport estimates, 
and a scalar function ${\mathcal P}_{j,k}$. 
This is achieved in Section~\ref{sec:T:j:k:div}, equation \eqref{eq:Q:j:k:def} below.
 In turn, the decomposition \eqref{eq:R:O:low}  allows us to write the oscillation stress as
\begin{align}
R_O 
&=\sum_j \chi_j^2 (\mathring R_q - \mathring R_{q,j}) +\left(\sum_j \chi_j^2 \mathring R_{q,j} + \sum_{j,k}\mathring{ \mathcal Q}_{j,k} \right) + R_{O,{\rm high}}\notag\\
&=: R_{O,{\rm approx}} + R_{O,{\rm low}} + R_{O,{\rm high}} \,,
\label{eq:RO:decompose}
\end{align}
where we have used the notation $\mathring R_{q,j}$ and $\mathring{ \mathcal Q}_{j,k}$ to denote the traceless parts of 
$R_{q,j}$ and ${ \mathcal Q}_{j,k}$,  respectively. We have also used the fact that $\BB$  already contains the Leray projector, so that 
it annihilates gradients.

\subsubsection{The definition of ${\mathcal Q}_{j,k}$}
\label{sec:T:j:k:div}

Before explaining how we obtain the $2$-tensor ${\mathcal Q}_{j,k}$ and  prior to estimating the three terms in \eqref{eq:RO:decompose}, two technical remarks are in order. First, 
since $\PP_{q+1,k} b_k(\lambda_{q+1} x) = b_k(\lambda_{q+1} x)$, we may write
\begin{align}
\PP_{q+1,k} \tilde w_{q+1,j,k} 
&= \tilde w_{q+1,j,k}  + \chi_{j}  \big[\PP_{q+1,k}, a_{k,j}  \psi_{q+1,j,k} \big] b_k( \lambda_{q+1} x )
\label{eq:local:w:q+1}
\end{align}
and thus
\begin{align}
w_{q+1} 
&=  \sum_{j,k} \tilde w_{q+1,j,k}  + \sum_{j,k} \chi_{j}  \big[\PP_{q+1,k}, a_{k,j}  \psi_{q+1,j,k} \big] b_k( \lambda_{q+1} x ).
\label{eq:local:w:q+1:*}
\end{align}
And second,  
since $\Lambda$ and $\PP_{q+1,k}$ commute, using \eqref{eq:bk-eigenvalue}  we may write
\begin{align}
\Lambda \PP_{q+1,k} \tilde w_{q+1,j,k} 
&= \chi_{j}     a_{k,j} \psi_{q+1,j,k}   \PP_{q+1,k} \Lambda b_k( \lambda_{q+1} x )  
+  \chi_{j}   [\PP_{q+1,k} \Lambda, a_{k,j} \psi_{q+1,j,k} ] b_k( \lambda_{q+1} x ) 
\notag \\
&= \lambda_{q+1}   \tilde w_{q+1,j,k}  +  \chi_{j}   [\PP_{q+1,k} \Lambda, a_{k,j}  \psi_{q+1,j,k} ] b_k( \lambda_{q+1} x ).
\label{eq:Lambda:w:q+1}
\end{align}
and thus
\begin{align}
\Lambda w_{q+1} = \lambda_{q+1} \sum_{j,k} \tilde w_{q+1} + \sum_{j,k}  \chi_{j}    [\PP_{q+1,k} \Lambda, a_{k,j}  \psi_{q+1,j,k} ] b_k( \lambda_{q+1} x ).
\label{eq:Lambda:w:q+1:*}
\end{align}

Let us define the potential vorticity associated to the perturbation $\PP_{q+1,k} \tilde w_{q+1,j,k}$ as
\begin{align}
\vartheta_{j,k} =   \nabla^{\perp} \cdot \PP_{q+1,k} \tilde w_{q+1,j,k}.
\label{eq:vartheta:def}
\end{align}
Using the identity
\begin{align}
\Lambda f \cdot \nabla g - \left(\nabla g\right)^T \Lambda f = \Lambda f^\perp (\nabla^\perp \cdot g) =  \left( \RSZ \big( \nabla^{\perp}\cdot f\big) \right) (\nabla^{\perp}\cdot g)
\label{eq:2D:MAGIC}
\end{align}
which holds for any vector fields $f,g \colon \TT^2 \to \CC^2$  with $\nabla \cdot f = 0$, 
we may write $\TTT_{j,k}$, as defined in \eqref{eq:TTT:j:k:def}, in the convenient form
\begin{align}
\TTT_{j,k}=\frac 12\left((\RSZ \vartheta_{j,k})\vartheta_{j,-k}+\vartheta_{j,k}(\RSZ\vartheta_{j,-k})\right) =: \TTT(\vartheta_{j,k},\vartheta_{j,-k})
\label{eq:TTT:useful}
\end{align}
where as usual $\RSZ = (\RSZ_1, \RSZ_2)$ is the Riesz-transform vector. Our goal next is to rewrite the operator $\TTT$ defined by \eqref{eq:TTT:useful} as a sum of a pressure gradient and a divergence of a $2$-tensor.

By an abuse of notation concerning Fourier transforms and Fourier series, we can rewrite $\TTT$ as a bilinear Fourier operator whose $\ell^{\rm th}$ component is given by 
\begin{align*}
2 \left(\TTT^\ell (f,g)\right)^\wedge (\xi) =& \int_{\RR^2} (\RSZ^\ell f)^\wedge(\xi-\eta)\hat g(\eta)~d\eta+
\int_{\RR^2} (\RSZ^\ell g)^\wedge(\eta)\hat f(\xi-\eta)~d\eta \\
=&\int_{\RR^2} \frac{i(\xi-\eta)^\ell}{\abs{\xi-\eta}} \hat f(\xi-\eta) \hat g (\eta)~d\eta+
\int_{\RR^2} \frac{i\eta^\ell}{\abs{\eta}}\hat g(\eta) \hat f(\xi-\eta)~d\eta \\
=&\int_{\RR^2} \left(\frac{i(\xi-\eta)^\ell}{\abs{\xi-\eta}} + \frac{i\eta^\ell}{\abs{\eta}}\right) \hat f(\xi-\eta) \hat g(\eta)~d\eta~.
\end{align*}
Rearranging the symbol,  we have that
\begin{align*}
\frac{i(\xi-\eta)^\ell}{\abs{\xi-\eta}} + \frac{i\eta^\ell}{\abs{\eta}}
=&\frac{i\left((\xi-\eta)^\ell \abs{\eta} + \eta^\ell \abs{\xi-\eta} \right)}{\abs{\xi-\eta} \abs{\eta}} \notag \\
=&\frac{i\xi^\ell \abs{\eta}}{\abs{\xi-\eta} \abs{\eta}} + \frac{i \eta^\ell \left(\abs{\xi-\eta}-\abs{\eta}\right) }{\abs{\xi-\eta} \abs{\eta}} \notag \\
=&\frac{i\xi^\ell}{\abs{\xi-\eta}} + \frac{i\eta^\ell}{\abs{\xi-\eta} \abs{\eta}} \int_0^1\left(\frac{d}{dr} \abs{\eta-r\xi}\right)~dr
\notag \\
=&\frac{i\xi^\ell}{\abs{\xi-\eta}} - \frac{i\eta^\ell}{\abs{\xi-\eta} \abs{\eta}} \xi^m \int_0^1 \frac{\left(\eta-r\xi\right)^m} {\abs{\eta-r\xi}}~dr \notag \\
=&\left(i\xi^\ell\right) \frac{1}{\abs{\xi-\eta} } + \left(i\xi^m\right)  \frac{i\eta^\ell}{\abs{\eta}} \frac{1}{\abs{\xi-\eta} } s^m(\xi-\eta,\eta) 
\end{align*}
where we  define the symbol $s \colon \RR^2 \times \RR^2 \to \CC$ by
\begin{align}
s^{m}(\zeta,\eta) = \int_0^1 \frac{i\left((1-r)\eta-r\zeta\right)^m} {\abs{(1-r)\eta-r\zeta}}~dr.
\label{eq:sm:symbol:def}
\end{align}
An important property of the symbol $s^m$ (which will  later be essential to our proof)  is that 
\begin{align}
s^m(-\eta,\eta) = \frac{i \eta^m}{|\eta|}
\label{eq:sm:property}
\end{align}
which is the symbol of the Riesz transform $\RSZ^m$.
As a result of the above computations,  may write
\begin{equation}\label{eq:T_symb}
\begin{split}
 \left( \TTT^\ell(f,g)\right)^\wedge (\xi)  =&  \frac{i\xi^\ell}{2} \int_{\RR^2} \frac{\hat f(\xi-\eta)}{\abs{\xi-\eta}}  \hat{g}(\eta)~d\eta + \frac{i\xi^m}{2} \int_{\RR^2} s^m(\xi-\eta,\eta) \frac{\hat f(\xi-\eta)}{\abs{\xi-\eta}} \frac{i\eta^l}{\abs{\eta}} \hat g(\eta)~d\eta.
\end{split}
\end{equation}
Upon defining the bilinear pseudo-product operator $\SSS^m$ in Fourier space as
\begin{equation}
\left(\SSS^m (f,g)\right)^\wedge(\xi):= \int_{\RR^2} s^m(\xi-\eta, \eta) \hat f(\xi-\eta)\hat g(\eta) ~d\eta~,
\label{eq:SSS:def}
\end{equation}
we then obtain from \eqref{eq:T_symb} the following formula for $\mathcal T$:
\begin{align}
\TTT^\ell(f,g)  = \frac 12 \partial_\ell (\Lambda^{-1}f g)+ \frac 12 \partial_m (\SSS^m(\Lambda^{-1}f,\RSZ^\ell g))~.
\label{eq:T:grad:div}
\end{align}
The
representation \eqref{eq:SSS:def} of the bilinear operator $\SSS^m$ is not very convenient to  estimate;  instead, we compute
 the inverse Fourier transform with respect to $\xi$ and rewrite $\SSS^m$ as
\begin{align}
\SSS^m(f,g)(x) :=  \frac{1}{(2\pi)^2}\intint_{\RR^2\times \RR^2} s^m(\zeta,\eta)  \hat{f}(\zeta) \hat{g}(\eta) e^{i x \cdot (\zeta + \eta)} d\zeta \, d\eta\, ,
\label{eq:SSS:def:real}
\end{align}
and upon further computing the inverse Fourier transform with respect to $(\eta,\zeta) \in \RR^4$, we rewrite $\SSS^m$ as follows:
\begin{align}
\SSS^m(f,g)(x) = \intint_{\RR^2\times \RR^2} K_{s^m}(x-y,x-z) f(y) g(z) dy dz \,,
\label{eq:SSS:def:conv}
\end{align}
where $K_{s^m}$ is the inverse Fourier transform in $\RR^4$ of $s^m$. Thus, another way to view $\SSS^m$ is as a bilinear convolution operator. We refer to~\cite{CoMe1978,GrTo2002,MuSh2013} and Appendix~\ref{app:pseudo:product}  for further properties of 
pseudo-product operators of  the type \eqref{eq:sm:symbol:def}, and for the equivalence of the definitions \eqref{eq:SSS:def:real}--\eqref{eq:SSS:def:conv}.

In view of \eqref{eq:TTT:useful} and \eqref{eq:T:grad:div}, we have now defined the tensor ${\mathcal Q}_{j,k}$ and the scalar ${\mathcal P}_{j,k}$ in \eqref{eq:R:O:low}, namely
\begin{align*}
\TTT_{j,k} = \frac 12 \nabla \left(\Lambda^{-1} \vartheta_{j,k} \, \vartheta_{j,-k} \right) + \frac 12 \div \left( \SSS\Big( \Lambda^{-1} \vartheta_{j,k}, \RSZ \vartheta_{j,-k} \Big) \right),
\end{align*}
so that 
\begin{align}
\left({\mathcal Q}_{j,k}\right)^{m \ell} 
&= \frac 12 \SSS^m\left( \Lambda^{-1} \vartheta_{j,k}, \RSZ^\ell \vartheta_{j,-k} \right) \,,
\label{eq:Q:j:k:def} \\
{\mathcal P}_{j,k} 
&= \Lambda^{-1} \vartheta_{j,k} \vartheta_{j,-k} \,, \notag
\end{align}
with the bilinear pseudo-product operator $\SSS^m$ being defined by \eqref{eq:SSS:def:real}.

\subsubsection{Canceling the principal part of the $R_{O,{\rm low}}$ stress}
Before estimating $R_{O,{\rm low}}$, we need to extract the leading order term in the matrices ${\mathcal Q}_{j,k}$ defined by \eqref{eq:Q:j:k:def}. For this purpose recall cf.~\eqref{eq:PP:q+1:k:def} and \eqref{eq:vartheta:def} that 
\begin{align*}
\vartheta_{j,k} 
&= \nabla^\perp \cdot \left( P_{\approx k\lambda_{q+1} } \tilde w_{q+1,j,k} + (\RSZ \otimes \RSZ) P_{\approx k \lambda_{q+1} } \tilde w_{q+1,j,k} \right) \notag\\
&= P_{\approx k\lambda_{q+1} }  \left(\nabla^\perp \cdot  \tilde w_{q+1,j,k} \right).
\end{align*}
Here $\hat K_{\approx k \lambda_{q+1}}(\xi) = \hat K_{\approx 1}(\xi/\lambda_{q+1} - k)$ is the Fourier symbol of $P_{\approx k \lambda_{q+1}}$. 
Using the precise definition of $\tilde w_{q+1,j,k}$ in \eqref{eq:tilde:w:q+1:def}, the definition of $b_k$ and $c_k$ in \eqref{eq:bk:ck:def}, and the notation \eqref{eq:psi:q+1:j:k:def}, we obtain
\begin{align}
\Lambda^{-1} \vartheta_{j,k} 
&= \Lambda^{-1} P_{\approx k \lambda_{q+1}} \left(\nabla^\perp \cdot  \tilde w_{q+1,j,k} \right) =  P_{\approx k \lambda_{q+1}} \RSZ^\perp \cdot \tilde w_{q+1,j,k} \notag\\
&=  \chi_j   (i k^\perp) \cdot \RSZ^\perp P_{\approx k \lambda_{q+1}} \Big( a_{k,j}  \psi_{q+1,j,k} c_k(\lambda_{q+1}x ) \Big) 
\label{eq:Lambda:-1:vartheta}
\end{align}
and 
\begin{align}
\RSZ^\ell \vartheta_{j,-k} 
&=  \RSZ^\ell P_{\approx - k \lambda_{q+1}} \left(\nabla^\perp \cdot  \tilde w_{q+1,j,-k} \right) \notag\\
&= - \chi_j (i k^\perp) \cdot \nabla^\perp  \RSZ^\ell P_{\approx - k \lambda_{q+1}} \Big( a_{k,j}  \psi_{q+1,j,-k} c_{-k}(\lambda_{q+1}x ) \Big) .
\label{eq:Riesz:vartheta}
\end{align}
We note that since multiplication by $c_k(\lambda_{q+1} x)$ results in a shift by $\lambda_{q+1} k$ in frequency, the Fourier analogues of \eqref{eq:Lambda:-1:vartheta} and \eqref{eq:Riesz:vartheta} are
\begin{align}
\left(\Lambda^{-1} \vartheta_{j,k} \right)^\wedge(\xi)
&=  - \chi_j(t)     \frac{k \cdot \xi}{|\xi|} \hat K_{\approx 1}\left(\frac{\xi}{\lambda_{q+1}} - k\right) \left( a_{k,j}  \psi_{q+1,j,k}\right)^\wedge (\xi - k\lambda_{q+1})
\label{eq:Lambda:-1:vartheta:F}
\end{align}
and 
\begin{align}
\left(\RSZ^\ell \vartheta_{j,-k} \right)^\wedge(\xi)
&=  \chi_j(t)  i \xi^\ell  \frac{( k \cdot \xi)}{|\xi|} K_{\approx 1}\left(\frac{\xi}{\lambda_{q+1}} + k\right)\left(a_{k,j}  \psi_{q+1,j,-k}\right)^\wedge(\xi+ k \lambda_{q+1}).
\label{eq:Riesz:vartheta:F}
\end{align}
Inserting  \eqref{eq:Lambda:-1:vartheta:F}--\eqref{eq:Riesz:vartheta:F} in  formula \eqref{eq:Q:j:k:def}, and recalling \eqref{eq:SSS:def}--\eqref{eq:SSS:def:real}, we obtain that 
\begin{align}
{\mathcal Q}_{j,k}^{m\ell}(x)
&=    \frac{\chi_j^2(t)}{2 (2 \pi)^2} \intint_{\RR^2\times\RR^2} (- i s^m(\zeta,\eta)) \frac{k \cdot \zeta}{|\zeta|} \hat K_{\approx 1}\left(\frac{\zeta - k \lambda_{q+1}}{\lambda_{q+1}} \right) \left( a_{k,j}  \psi_{q+1,j,k}\right)^\wedge (\zeta - k\lambda_{q+1})  \notag\\
&\qquad \qquad \quad \times \eta^\ell  \frac{  k \cdot \eta}{|\eta|} \hat K_{\approx 1}\left(\frac{\eta + k\lambda_{q+1}}{\lambda_{q+1}} \right)\left(a_{k,j}  \psi_{q+1,j,-k}\right)^\wedge(\eta+ k \lambda_{q+1}) e^{i x \cdot(\zeta + \eta)} d\zeta d\eta \notag\\
&=    \frac{\chi_j^2(t)}{2 (2 \pi)^2} \intint_{\RR^2\times\RR^2} (- i s^m(\zeta + k \lambda_{q+1},\eta - k \lambda_{q+1})) \frac{k \cdot (\zeta + k \lambda_{q+1})}{|\zeta + k \lambda_{q+1}|} \hat K_{\approx 1}\left(\frac{\zeta}{\lambda_{q+1}} \right) \left( a_{k,j}  \psi_{q+1,j,k}\right)^\wedge (\zeta)  \notag\\
&\qquad \qquad \quad \times (\eta^\ell - k^\ell \lambda_{q+1})  \frac{  k \cdot (\eta - k \lambda_{q+1}) }{|\eta - k \lambda_{q+1}|} \hat K_{\approx 1}\left(\frac{\eta}{\lambda_{q+1}} \right)\left(a_{k,j}  \psi_{q+1,j,-k}\right)^\wedge(\eta) e^{i x \cdot(\zeta + \eta)} d\zeta d\eta \notag\\
&= \frac{\chi_j^2}{2}  \frac{1}{(2\pi)^2} \int\int_{\RR^2\times \RR^2} M_k^{m\ell}(\zeta,\eta) \left( a_{k,j}  \psi_{q+1,j,k}\right)^\wedge (\zeta) 
\left(a_{k,j}  \psi_{q+1,j,-k}\right)^\wedge(\eta) e^{i x \cdot(\zeta + \eta)} d\zeta d\eta
\label{eq:Q:j:k:var1}
\end{align}
where in the second to last line we have used the change of variables in $\zeta$ and $\eta$ by shifting with $\pm k \lambda_{q+1}$, and in the last line we have denoted
\begin{align}
&M_k^{m\ell}(\zeta,\eta) \notag\\
&=   - i s^m(\zeta + k \lambda_{q+1},\eta - k \lambda_{q+1})  \frac{k \cdot (\zeta + k \lambda_{q+1})}{|\zeta + k \lambda_{q+1}|} \hat K_{\approx 1}\left(\frac{\zeta}{\lambda_{q+1}} \right)   (\eta^\ell - k^\ell \lambda_{q+1})  \frac{  k \cdot (\eta - k \lambda_{q+1}) }{|\eta - k \lambda_{q+1}|} \hat K_{\approx 1}\left(\frac{\eta}{\lambda_{q+1}} \right)
\notag\\
&=: \int_0^1 M^{m\ell}_{k,r}(\zeta,\eta) dr.
\label{eq:Q:j:k:var2}
\end{align}
and
\begin{align}
M^{m\ell}_{k,r}(\zeta,\eta) 
&=   \frac{\big( (1-r) \eta   - r \zeta - k \lambda_{q+1} \big)^m}{| (1-r) \eta   - r \zeta - k \lambda_{q+1}|}  (\eta^\ell - k^\ell \lambda_{q+1}) \notag\\
&\qquad \qquad \times    \frac{k \cdot (\zeta + k \lambda_{q+1})}{|\zeta + k \lambda_{q+1}|}  \frac{  k \cdot (\eta - k \lambda_{q+1}) }{|\eta - k \lambda_{q+1}|} \hat K_{\approx 1}\left(\frac{\zeta}{\lambda_{q+1}} \right)  \hat K_{\approx 1}\left(\frac{\eta}{\lambda_{q+1}} \right) .
\label{eq:Q:j:k:var22}
\end{align}
We observe here that the multiplier $M_{kr}^{m\ell}$ defined in \eqref{eq:Q:j:k:var22} has two important features: the 
first concerns smoothness and  will allow us to establish bounds on the induced bilinear pseudo-product operator, while the second concerns
structure, and  allows us to define the principal term in ${\mathcal Q}_{j,k}^{m\ell}$ and cancel the leading order term in the
oscillation stress $R_{O,{\rm low}}$.

First, we note that by \eqref{eq:Q:j:k:var2}, we have that 
\begin{align}
M_{k,r}^{m\ell}(\zeta,\eta) =  \lambda_{q+1} (M_{k,r}^*)^{m\ell}\left(\frac{\zeta}{\lambda_{q+1}}, \frac{\eta}{\lambda_{q+1}}\right)
\label{eq:M:k:r:rescale}
\end{align}
where 
\begin{align}
(M_{k,r}^*)^{m\ell}(\xi_1, \xi_2) =   \frac{\big(  (1-r) \xi_2   - r \xi_1 - k \big)^m}{| (1-r) \xi_2   - r \xi_1 -k |}  (\xi_2^\ell - k^\ell)  \frac{k \cdot (\xi_1 + k )}{|\xi_1 + k|}  \frac{  k \cdot (\xi_2 - k ) }{|\xi_2 - k|} \hat K_{\approx 1}\left( \xi_1 \right)  \hat K_{\approx 1}\left(\xi_2 \right) 
\label{eq:M:k:r:*:def}
\end{align}
for $\xi_1,\xi_2\in\RR^2$. We notice here that $M_{k,r}^*$ is independent of $\lambda_{q+1}$, and that by the definition of $\hat K_{\approx 1}$, the multiplier $M_{k,r}^*$ is supported on $(\xi_1,\xi_2) \in B_{1/8}(0) \times B_{1/8}(0)$. The latter property ensures that $|\xi_1+ k|\geq 1/2$, $|\xi_2-k| \geq 1/2$, and $|k+ (1-r)\xi_1 - r \xi_2|\geq 1/8$. This ensures that the multiplier $M_{k,r}^*$ is infinitely many times differentiable, with bounds that are uniform in $r \in (0,1)$.

Second, we note that from \eqref{eq:Q:j:k:var22} it follows that 
\begin{align}
M_{k,r}^{m\ell}(0,0) =  - \lambda_{q+1} k^m k^\ell
\label{eq:MAGIC:*}
\end{align}
whenever $k \in \SS^1$.
Moreover, by the definition of the inverse Fourier transform, we have that 
\begin{align}
&\frac{1}{(2\pi)^2} \int\int_{\RR^2\times \RR^2}  \left( a_{k,j}  \psi_{q+1,j,k}\right)^\wedge (\zeta)
\left(a_{k,j}  \psi_{q+1,j,-k}\right)^\wedge(\eta) e^{i x \cdot(\zeta + \eta)} d\zeta d\eta  \notag\\
&\qquad = a_{k,j}(x) \psi_{q+1,j,k}(x) a_{k,j}(x) \psi_{q+1,j,-k}(x) \notag\\
&\qquad = a_{k,j}^2(x) \,,
\label{eq:Q:j:k:var3}
\end{align}
since by \eqref{eq:psi:q+1:j:k:def},  $\psi_{q+1,j,k}(x) \psi_{q+1,j,-k}(x) = 1$.

Therefore, combining \eqref{eq:Q:j:k:var1}--\eqref{eq:Q:j:k:var3}, we decompose ${\mathcal Q}_{j,k}^{m\ell}$ as a principal and commutator term:
\begin{align}
{\mathcal Q}_{j,k}^{m\ell}
&=  - \frac{\lambda_{q+1}}{2} \chi_j^2  (k \otimes k)^{ m \ell}  a_{k,j}^2  + \tilde {\mathcal Q}_{j,k}^{m\ell} \notag \\
&=   \frac{\lambda_{q+1}}{2} \chi_j^2  \left( k^{\perp}\otimes k^{\perp} - \Id\right)^{ml} a_{k,j}^2  + \tilde {\mathcal Q}_{j,k}^{m\ell}.
\label{eq:Q:j:k:decompose}
\end{align}
where,t since $|k|=1$,  $\Tr(k^\perp \otimes k^\perp) = 1$, and where 
\begin{align*}
\tilde {\mathcal Q}_{j,k}^{m\ell}(x)  
&= \frac{\chi_j^2}{2(2\pi)^2} \int_0^1  \intint_{\RR^2\times\RR^2} \left( M_{k,r}^{m\ell}(\zeta,\eta) - M_{k,r}^{m\ell}(0,0) \right) \notag\\
&\qquad \qquad \qquad \qquad \times \left( a_{k,j}  \psi_{q+1,j,k}\right)^\wedge(\zeta)  \left(a_{k,j}  \psi_{q+1,j,-k}\right)^\wedge(\eta) e^{i x \cdot(\zeta + \eta)} d\zeta d\eta dr.
\end{align*}
Next, we use
the mean value theorem together with \eqref{eq:M:k:r:rescale}, and find that
\begin{align}
\tilde {\mathcal Q}_{j,k}^{m\ell}(x) 
&= \frac{\chi_j^2}{2(2\pi)^2} \int_0^1 \!\! \int_0^1 \intint_{\RR^2\times\RR^2}\left( (\zeta \cdot \nabla_\zeta + \eta \cdot \nabla_\eta) M_{k,r}^{m\ell}\right) (\bar r \zeta,\bar r \eta) \notag\\
&\qquad \qquad \qquad \qquad \times   \left( a_{k,j}  \psi_{q+1,j,k}\right)^{\wedge}(\zeta)  \left(a_{k,j}  \psi_{q+1,j,-k}\right)^\wedge(\eta) e^{i x \cdot(\zeta + \eta)} d\zeta d\eta d\bar r dr  \notag\\
&= \frac{ \chi_j^2}{2(2\pi)^2}\int_0^1 \!\! \int_0^1\intint_{\RR^2\times\RR^2} \left(  - i \nabla_{\xi_1}  (M_{k,r}^*)^{m\ell}\right) \left(\frac{\bar r \zeta}{\lambda_{q+1}}, \frac{\bar r \eta}{\lambda_{q+1}}\right)  \cdot  \left(\nabla (a_{k,j}  \psi_{q+1,j,k})\right)^\wedge(\zeta) \notag\\
&\qquad \qquad \qquad \qquad \times \left(a_{k,j}  \psi_{q+1,j,-k}\right)^\wedge(\eta) e^{i x \cdot(\zeta + \eta)} d\zeta d\eta d\bar r  dr \notag\\
&\quad + \frac{ \chi_j^2}{2(2\pi)^2} \int_0^1 \!\! \int_0^1 \intint_{\RR^2\times\RR^2}\left(-i  \nabla_{\xi_2} (M_{k,r}^*)^{m\ell}\right) \left(\frac{\bar r \zeta}{\lambda_{q+1}}, \frac{\bar r \eta}{\lambda_{q+1}}\right)  \cdot \left(\nabla(a_{k,j}  \psi_{q+1,j,-k})\right)^\wedge(\eta) \notag\\
&\qquad \qquad \qquad \qquad \times \left( a_{k,j}  \psi_{q+1,j,k}\right)^\wedge(\zeta) e^{i x \cdot(\zeta + \eta)} d\zeta d\eta d\bar r  dr \notag\\
&= \left(\tilde {\mathcal Q}_{j,k}^{(1)}\right)^{m\ell}(x) + \left(\tilde {\mathcal Q}_{j,k}^{(2)}\right)^{m\ell}(x).
\label{eq:Q:j:k:decompose:nasty}
\end{align}
Note that $\tilde Q^{(1)}_{j,k}$ and $\tilde Q^{(2)}_{j,k}$ are both bilinear pseudo-product operators, and thus similarly to the equivalence between \eqref{eq:SSS:def:real}--\eqref{eq:SSS:def:conv} we may take the inverse Fourier transform of \eqref{eq:Q:j:k:decompose:nasty} with respect to the variable $(\zeta,\eta) \in \RR^4$.
For $(z_1,z_2) \in \RR^2 \times \RR^2$,  we denote the inverse Fourier transforms of the above vectors of multipliers as
\begin{subequations}
\label{eq:multilinear:kernels}
\begin{align}
({\mathcal K}_{k,r,\bar r}^{(1)})^{m\ell} (z_1,z_2) &= \frac{\lambda_{q+1}^4}{\bar r^4} \left(  - i \nabla_{\xi_1}  (M_{k,r}^*)^{m\ell}\right)^{\vee}\left(\frac{\lambda_{q+1} z_1}{\bar r},\frac{\lambda_{q+1} z_2 }{\bar r} \right) , \\
({\mathcal K}_{k,r,\bar r}^{(2)})^{m\ell} (z_1,z_2) &= \frac{\lambda_{q+1}^4}{\bar r^4} \left(  - i \nabla_{\xi_2}  (M_{k,r}^*)^{m\ell}\right)^{\vee}\left(\frac{\lambda_{q+1} z_1}{\bar r},\frac{\lambda_{q+1} z_2 }{\bar r} \right).
\end{align}
\end{subequations}
It follows from basic scaling properties of the Fourier transform that
\begin{align}
\left(\tilde {\mathcal Q}_{j,k}^{(1)}\right)^{m\ell}(x) 
&= \frac{\chi_j^2}{2} \int_0^1\!\!\int_0^1 \!\! \intint_{\RR^2\times\RR^2}  ({\mathcal K}_{k,r,\bar r}^{(1)})^{m\ell} (x-z_1 ,  x-z_2)   \cdot \nabla (a_{k,j}  \psi_{q+1,j,k}) (z_1) \left(a_{k,j}  \psi_{q+1,j,-k}\right)(z_2) dz_1 dz_2 d\bar r dr \notag\\
&=: \chi_j^2 \tilde \SSS_k^{(1),m\ell}\big (\nabla (a_{k,j}  \psi_{q+1,j,k}), a_{k,j}  \psi_{q+1,j,-k} \big)
\label{eq:Q:j:k:1:convolution} \\
\left(\tilde {\mathcal Q}_{j,k}^{(2)}\right)^{m\ell}(x) 
&= \frac{\chi_j^2}{2} \int_0^1\!\!\int_0^1 \!\! \intint_{\RR^2\times\RR^2}  ({\mathcal K}_{k,r,\bar r}^{(2)})^{m\ell} (x-z_1 ,  x-z_2)   \cdot \nabla (a_{k,j}  \psi_{q+1,j,-k}) (z_2) \left(a_{k,j}  \psi_{q+1,j,k}\right)(z_1) dz_1 dz_2 d\bar r dr \notag\\
&=: \chi_j^2 \tilde \SSS_k^{(2),m\ell}\big (a_{k,j}  \psi_{q+1,j,k} , \nabla (a_{k,j}  \psi_{q+1,j,-k})\big)
\label{eq:Q:j:k:2:convolution}
\end{align}
Here as usual we have identified the $\TT^2$-periodic functions of $z_1$ and $z_2$ with their periodic extensions to all of $\RR^2$.
The precise form of the above kernels appearing in \eqref{eq:Q:j:k:1:convolution}--\eqref{eq:Q:j:k:2:convolution} is not  important. The only important property of these kernels, which we will use repeatedly when bounding these bilinear convolution operators, is that 
for $i \in \{1,2\}$, with the notation $z = (z_1,z_2) \in \RR^2\times \RR^2$, we have that
\begin{align}
\norm{ z^a  \nabla_z^b ({\mathcal K}_{k,r,\bar r}^{(i)})^{m\ell}}_{L^1_{z_1,z_2}(\RR^2 \times\RR^2)}
&\leq C_{a,b} \left(\frac{\lambda_{q+1}}{\bar r}\right)^{|b|-|a|}
\label{eq:K:k:r:bounds}
\end{align}
uniformly for $r\in (0,1)$, and for all $0\leq |a|,|b|\leq 1$. The bound \eqref{eq:K:k:r:bounds} follows upon rescaling from the fact that the multiplier $M_{k,r}^*$ defined in \eqref{eq:M:k:r:*:def} is in $C_0^\infty( B_{1/8}(0) \times B_{1/8}(0))$, and thus so are $\partial_{\xi_1}M_{k,r}^*$ and $\partial_{\xi_2} M_{k,r}^*$.

In summary, the  decomposition \eqref{eq:Q:j:k:decompose}--\eqref{eq:Q:j:k:decompose:nasty}, allows us to split the low frequency part of the oscillation error, defined by \eqref{eq:RO:decompose}, as
\begin{align*}
R_{O,{\rm low}} = \mathring O_1 + \mathring O_2
\end{align*}
where the principal term $\mathring O_1$ is the traceless part of
\begin{align}
O_1 = \sum_j \chi_j^2  \mathring R_{q,j} + \frac{\lambda_{q+1}}{2} \sum_{j,k}   \left(  k^{\perp}\otimes k^{\perp} - \Id \right)  \chi_j^2 a_{k,j}^2  
\label{eq:O1:def}
\end{align}
while $\mathring O_2$ is the traceless part of the commutator terms,  given by
\begin{align}
O_2 =   \sum_{j,k} \tilde {\mathcal Q}_{j,k}^{(1)} + \tilde {\mathcal Q}_{j,k}^{(2)} = O_{21} + O_{22}.
\label{eq:O2:def}
\end{align}

The key observation here is that the $\mathring O_1$ term vanishes. Indeed,  by the definition of the $a_k$ in \eqref{eq:ak:def}, and of the functions $\gamma_k$ in \eqref{e:split}, we have
\begin{align*}
\frac{\mathring R_{q,j}}{\lambda_{q+1}} + \frac 12 \sum_{k \in \Omega_j} a_{k,j}^2 ( k^\perp \otimes k^\perp )=   \rho_j  \Id\,,
\end{align*}
which shows that $\mathring O_1=0$, since the traceless part of a multiple of the identity is the zero matrix.
 Therefore, we may summarize our  computations in this section as
\begin{align}
R_{O,{\rm low}} =  \mathring O_{21} + \mathring O_{22}
\label{eq:RO:low:final}
\end{align}
with $O_{21}$ and $O_{22}$ as defined by \eqref{eq:Q:j:k:decompose:nasty}--\eqref{eq:O2:def}.

\subsubsection{Amplitude of the $R_{O,{\rm approx}}$ stress}

In order to bound $R_{O,{\rm approx}}$, we recall cf.~\eqref{eq:R:q:transported}   that
\begin{align*}
\left( \mathring R_q - \mathring R_{q,j}\right)(x,j \tau_{q+1}) = 0
\end{align*}
where $j \tau_{q+1}$ is the center of the time-support of $\chi_{j}$, and  moreover 
\begin{align}
D_{t,q} \left( \mathring R_q - \mathring R_{q,j}\right) = D_{t,q} \mathring R_q.
\label{eq:material:RO:approx}
\end{align}
We may thus appeal to the inductive assumption \eqref{eq:ind:q:4} and the transport estimate~\eqref{eq:max:prin} to find that
\begin{align}
\norm{ \mathring R_q - \mathring R_{q,j}}_{C^0(\supp \chi_j)} 
&\lesssim \tau_{q+1} \lambda_{q}^2 \delta_{q}^{1/2} \lambda_{q+1} \delta_{q+1} \notag\\
&\lesssim \lambda_q \delta_q^{1/4} \delta_{q+1}^{3/4} \,.
\label{eq:RO:approx:amplitude}
\end{align}
Upon summing over $j$, we arrive at 
\begin{align}
\label{eq:R_O_approx_est}
\norm{R_{O,{\rm approx}}}_{C^0} \leq \sum_j \chi_j^2 \norm{ \mathring R_q - \mathring R_{q,j}}_{C^0(\supp \chi_j)}  \leq \lambda_q \delta_q^{1/4} \delta_{q+1}^{3/4}\,.
\end{align}

\subsubsection{Amplitude of the $R_{O,{\rm low}}$ stress}

The map from matrices to their traceless part is clearly bounded and thus in view of \eqref{eq:RO:low:final}, we need to estimate $O_{21}$ and $O_{22}$. We only show the estimate for $O_{21}$, since the one for $O_{22}$ is identical, upon changing $(1)$ with $(2)$ below. For this purpose we use \eqref{eq:Q:j:k:1:convolution}--\eqref{eq:Q:j:k:2:convolution}  and the kernel estimate \eqref{eq:K:k:r:bounds} with $|a|=|b|=0$ to conclude  that
\begin{align}
\norm{O_{21}}_{C^0}  
&\lesssim   \sum_{j,k}  \chi_j^2 \norm{\tilde \SSS_k^{(1),m\ell}\big (\nabla (a_{k,j}  \psi_{q+1,j,k}), a_{k,j}  \psi_{q+1,j,-k} \big)}_{C^0} \notag\\
&\lesssim   \sum_{j,k}  \chi_j^2 \norm{\nabla (a_{k,j} \psi_{q+1,k,j})}_{C^0} \norm{a_{k,j} \psi_{q+1,j,-k}}_{C^0} \sup_{r,\bar r\in(0,1)}\norm{  ({\mathcal K}_{k,r,\bar r}^{(1)} )^{m\ell}}_{L^1(\RR^2 \times\RR^2)} \notag\\
&\lesssim  \delta_{q+1}^{1/2} \left(\lambda_q \delta_{q+1}^{1/2} + \delta_{q+1}^{1/2} \tau_{q+1} \lambda_{q+1} \lambda_q^2 \delta_{q}^{1/2} \right) \notag\\
&\lesssim  \delta_{q+1}  \tau_{q+1} \lambda_{q+1}  \lambda_q^2 \delta_{q}^{1/2} \notag\\
&=   \lambda_{q} \delta_q^{1/4} \delta_{q+1}^{3/4} \,. \notag
\end{align}
Therefore
\begin{equation}\label{eq:R_O_low}
\norm{R_{O,{\rm low}}}_{C^0}   \leq \lambda_{q} \delta_q^{1/4} \delta_{q+1}^{3/4}\,.
\end{equation}

\subsubsection{Amplitude of the $R_{O,{\rm high}}$ stress}
We recall cf.~\eqref{eq:R:O:high} that 
\begin{subequations}
\label{eq:RO:high:decompose}
\begin{align}
 R_{O,{\rm high}} &= O_3 - O_4 \\
 O_3 &=\BB \tilde P_{\approx \lambda_{q+1}}   \sum_{\underset{k+k'\neq 0 } {j,j',k,k'}} \big(\Lambda \PP_{q+1,k} \tilde w_{q+1,j,k}\big) \cdot \nabla \big(\PP_{q+1,k'} \tilde w_{q+1,j',k'}\big)   \\
O_4 &= \BB  \tilde P_{\approx \lambda_{q+1}}  \sum_{\underset{k+k'\neq 0 } {j,j',k,k'}} \big(\nabla \PP_{q+1,k} \tilde w_{q+1,j,k}\big)^T \cdot \big( \Lambda  \PP_{q+1,k'} \tilde w_{q+1,j',k'} \big) .
\end{align}
\end{subequations}

\noindent{\bf The $O_3$ estimate.} Appealing to \eqref{eq:local:w:q+1} and \eqref{eq:Lambda:w:q+1}, we have that
\begin{align}
O_3
&=\BB \tilde P_{\approx \lambda_{q+1}} \lambda_{q+1}\sum_{\underset{k+k'\neq 0 } {j,j',k,k'}}  \div \left( \big(\tilde w_{q+1,j,k}\big) \otimes \big( \tilde w_{q+1,j',k'}\right) \big)   \notag\\
&\quad + \BB \tilde P_{\approx \lambda_{q+1}} \div \Bigg(\sum_{\underset{k+k'\neq 0 } {j,j',k,k'}} \big(\lambda_{q+1} \tilde w_{q+1,j,k}\big) \otimes \big( \chi_{j'}  \big[\PP_{q+1,k'}, a_{k',j'}  \psi_{q+1,j',k'} \big] b_{k'}( \lambda_{q+1} x )\big)  \Bigg) \notag\\
&\quad + \BB \tilde P_{\approx \lambda_{q+1}} \div \Bigg(\sum_{\underset{k+k'\neq 0 } {j,j',k,k'}} \big(\chi_{j}    [\PP_{q+1,k} \Lambda, a_{k,j}  \psi_{q+1,j,k} ] b_k( \lambda_{q+1} x )\big) \otimes \big(\PP_{q+1,k'} \tilde w_{q+1,j',k'}\big)  \Bigg) \notag\\
&=: O_{31} + O_{32} + O_{33}.
\label{eq:O3:decompose}
\end{align}
For $O_{31}$, we need to compute carefully the divergence before estimating it. From \eqref{eq:Beltrami_identity},
\begin{align*}
O_{31}
=& \lambda_{q+1}  \BB  \tilde P_{\approx \lambda_{q+1}}  \Bigg(\sum_{  \underset{k+k'\neq 0 } {j,j',k,k'}    } \chi_{j}\chi_{j'} \left(b_{k'}( \lambda_{q+1} x) \otimes   b_{k}( \lambda_{q+1} x)-\frac12 b_{k'}( \lambda_{q+1} x) \cdot   b_{k}( \lambda_{q+1} x)\Id\right)
\notag \\
    &\qquad\qquad \qquad \qquad \times \nabla\left( a_{k,j} \psi_{q+1,j,k} a_{k'} \psi_{q+1,j',k'}\right) \Bigg)
\end{align*}
We thus obtain from Lemma~\ref{lem:D:N:a:psi} that 
\begin{align}
\|O_{31}\|_{C^0} 
&\lesssim \sum_{j,j',k,k'} \norm{ \chi_{j}\chi_{j'}  \nabla\left( a_{k,j} \psi_{q+1,j,k} a_{k'} \psi_{q+1,j',k'}\right)}_{C^0} \notag\\
&\lesssim \lambda_q \delta_{q+1} + \delta_{q+1} \tau_{q+1} \lambda_{q+1} \lambda_q^2 \delta_{q}^{1/2} \notag\\
&\lesssim \lambda_q \delta_{q}^{1/4}\delta_{q+1}^{3/4}\,.
\label{eq:sharp:2}
\end{align}
Using the commutator estimate \eqref{eq:comm:multiplication} (with $s=0$ and $\lambda = \lambda_{q+1}$; and respectively with $s=1$ and $\lambda = \lambda_{q+1}$) and Lemma~\ref{lem:D:N:a:psi}, we have
\begin{align*}
\norm{O_{32}}_{C^0} + \norm{O_{33}}_{C^0}
&\lesssim \sum_{j,j',k,k'}   \norm{\tilde w_{q+1,j,k}}_{C^0}   \norm{ \nabla( a_{k',j'} \psi_{q+1,j',k'})}_{C^0(\supp(\chi_{j'})} \notag \\
&\lesssim   \delta_{q+1}^{1/2}  \left( \lambda_q \delta_{q+1}^{1/2} + \delta_{q+1}^{1/2} \tau_{q+1} \lambda_{q+1} \lambda_q^2 \delta_{q}^{1/2} \right) \notag\\
&\lesssim  \lambda_q \delta_{q}^{1/4}\delta_{q+1}^{3/4}\,.
\end{align*}
Combining the above  estimates yields
\begin{align}
\label{eq:O_3_bound}
\norm{O_3}_{C^0} \leq \lambda_q \delta_{q}^{1/4}\delta_{q+1}^{3/4}.
\end{align}

\noindent{\bf The $O_4$ estimate.}
In order to bound the $O_4$ part of the oscillation error, we note that 
\begin{align}
\BB \left(  (\nabla \tilde w_{q+1,j,k})^T \cdot \tilde w_{q+1,j',k'}  + (\nabla \tilde w_{q+1,j',k'})^T \cdot \tilde w_{q+1,j,k} \right) = \frac{1}{2}  \BB  \left( \nabla \left( \tilde w_{q+1,j,k} \cdot \tilde w_{q+1,j',k'} \right) \right)= 0
\end{align}
since $\BB$ contains the Leray projector. Therefore, using \eqref{eq:local:w:q+1} and \eqref{eq:Lambda:w:q+1} we obtain
\begin{align}
O_4 
&=  \BB \tilde P_{\approx \lambda_{q+1}} \Bigg(\sum_{\underset{k+k'\neq 0 } {j,j',k,k'}} 
\left( \chi_j \nabla \big[ \PP_{q+1,k}, a_{k,j} \psi_{q+1,j,k}] b_k(\lambda_{q+1} x) \right)^T \cdot  
\left(\lambda_{q+1} \tilde w_{q+1,j',k'} \right) \Bigg) \notag\\
&\quad+ \BB \tilde P_{\approx \lambda_{q+1}} \Bigg(\sum_{\underset{k+k'\neq 0 } {j,j',k,k'}} 
\left( \nabla  \PP_{q+1,k} \tilde w_{q+1,j,k}\right)^T \cdot  \left( \chi_{j'}  \left( \big[ \PP_{q+1,k'} \Lambda, a_{k',j'} \psi_{q+1,j',k'} \big] b_{k'}(\lambda_{q+1} x) \right) \right) \Bigg) \notag\\
&= O_{41} + O_{42}.
\label{eq:O4:decompose}
\end{align}
Appealing to the formula
\begin{align*}
\nabla \big[ \PP_{q+1,k}, a_{k,j} \psi_{q+1,j,k}] b_k(\lambda_{q+1} x) =  \big[ \PP_{q+1,k}, \nabla(a_{k,j} \psi_{q+1,j,k})] b_k(\lambda_{q+1} x) + \big[ \PP_{q+1,k}, a_{k,j} \psi_{q+1,j,k}]  \nabla b_k(\lambda_{q+1} x) ,
\end{align*}
the commutator estimate \eqref{eq:comm:multiplication} (with $s= 0$ and $\lambda = \lambda_{q+1}$), and Lemma~\ref{lem:D:N:a:psi}, we
obtain the bound
\begin{align*}
\norm{O_{41}}_{C^0} 
&\lesssim \lambda_{q+1}^{-1} \sum_{j,j',k,k'} \left( \lambda_{q+1}^{-1} \norm{\nabla^2 (a_{k,j} \psi_{q+1,j,k})}_{C^0(\supp \chi_j)} +    \norm{\nabla (a_{k,j} \psi_{q+1,j,k})}_{C^0(\supp \chi_j)}   \right) \lambda_{q+1} \norm{\tilde w_{q+1,j',k'}}_{C^0} \notag\\
&\lesssim \delta_{q+1}^{1/2} \left(\lambda_{q+1}^{-1} \lambda_q^2 \delta_{q+1}^{1/2} + \lambda_{q+1}^{-1} \tau_{q+1}^2 \lambda_{q+1}^2 \lambda_q^4 \delta_q + \lambda_q \delta_{q+1}^{1/2} + \delta_{q+1}^{1/2}  \tau_{q+1} \lambda_{q+1} \lambda_q^2 \delta_q^{1/2} \right)  \notag\\
&\lesssim \frac{\lambda_q^2 \delta_q^{1/2} \delta_{q+1}^{1/2}}{\lambda_{q+1}} + \lambda_q \delta_q^{1/4} \delta_{q+1}^{3/4} \notag\\
&\lesssim \lambda_q \delta_q^{1/4} \delta_{q+1}^{3/4}.
\end{align*}
The term $O_{42}$ is bounded similarly, by appealing to the commutator estimate \eqref{eq:comm:multiplication} (with $s=1$ and $\lambda = \lambda_{q+1}$):
\begin{align*}
\norm{O_{42}}_{C^0}   
&\lesssim \lambda_{q+1}^{-1} \sum_{j,j',k,k'} \lambda_{q+1} \norm{\tilde w_{q+1,j,k}}_{C^0} \norm{\nabla ( a_{k',j'} \psi_{q+1,j',k'})}_{C^0 (\supp \chi_{j'})} \notag\\
&\lesssim \delta_{q+1}^{1/2} \left( \lambda_q \delta_{q+1}^{1/2} + \delta_{q+1}^{1/2} \tau_{q+1} \lambda_q^2 \delta_{q}^{1/2} \right) \notag\\
&\leq \lambda_q \delta_q^{1/4} \delta_{q+1}^{3/4}\,.
\end{align*}
Combining the above estimates, we arrive at the  bound
\begin{align}
\label{eq:O_4_bound}
\norm{O_4}_{C^0} \lesssim \lambda_q \delta_q^{1/4} \delta_{q+1}^{3/4}.
\end{align}

\subsubsection{Bound on $R_{O}$}

Combining the estimates \eqref{eq:R_O_approx_est}, \eqref{eq:R_O_low}, \eqref{eq:O_3_bound} and \eqref{eq:O_4_bound} yields
\begin{align*}
\norm{R_O}_{C^0}\lesssim \lambda_q \delta_q^{1/4} \delta_{q+1}^{3/4}= \lambda_0^{-2+\frac{5\beta}{2}}\lambda_{q+2} \delta_{q+2}.
\end{align*}
Assuming $\lambda_0$ is sufficiently large and $\beta<\sfrac 45$ we obtain our claim.

\subsubsection{Material derivative  of the $R_{O,{\rm approx}}$ stress}
We recall cf.~\eqref{eq:RO:decompose} and \eqref{eq:material:RO:approx} that
\begin{align*}
 D_{t,q} R_{O,{\rm approx}} 
 &= \sum_j \chi_j^2 D_{t,q}\Big( \mathring R_q - \mathring R_{q,j} \Big) + 2 \sum_j \chi_j \chi_j' \Big( \mathring R_q - \mathring R_{q,j} \Big) \notag\\
 &= \sum_j \chi_j^2 \Big(D_{t,q} \mathring R_q\Big) + 2  \sum_j \chi_j \chi_j' \Big( \mathring R_q - \mathring R_{q,j} \Big).
\end{align*}
Therefore, using the inductive estimate \eqref{eq:ind:q:4}, the fact that the $\chi_j^2$ form a partition of unity, and the previously established bound \eqref{eq:RO:approx:amplitude}, we obtain that
\begin{align*}
 \norm{D_{t,q} R_{O,{\rm approx}}}_{C^0}
 &\lesssim \lambda_q^2 \delta_q^{1/2} \lambda_{q+1} \delta_{q+1} + \tau_{q+1}^{-1} \left(\tau_{q+1} \lambda_q^2 \delta_q^{1/2} \lambda_{q+1} \delta_{q+1}\right) \notag\\
 &\lesssim  \lambda_q^2 \delta_q^{1/2} \lambda_{q+1} \delta_{q+1}\,.
\end{align*}

\subsubsection{Material derivative  of the $R_{O,{\rm low}}$ stress}
Recall cf.~\eqref{eq:RO:low:final} that $R_{O,{\rm low}} = \mathring O_{21}+\mathring O_{22}$, terms which are defined in \eqref{eq:O2:def}, with  \eqref{eq:Q:j:k:1:convolution}--\eqref{eq:Q:j:k:2:convolution}. Trivially,
\begin{align*}
\norm{D_{t,q} \left(\mathring O_{21}+\mathring O_{22}\right)}_{C^0}\leq \norm{D_{t,q}\left(O_{21}+ O_{22}\right)}_{C^0} \,.
\end{align*}
We only show the estimate for $O_{21}$, since the one for $O_{22}$ is identical, upon changing $(1)$ with $(2)$ below.  Note that in view of \eqref{eq:K:k:r:bounds} the bilinear convolution operators defining ${\tilde Q}_{j,k}^{(1)}$ has a kernel which obeys the conditions of Lemma~\ref{lem:multi:Dt:commutator}. Moreover, by construction we have
\[ 
D_{t,q}( a_{k,j} \psi_{q+1,j,k} ) = D_{t,q}( a_{k,j} \psi_{q+1,j,-k} ) = 0
\]
and thus, using the notation \eqref{eq:comm:Dt:SM} we obtain
\begin{align*}
&D_{t,q} \left( \tilde \SSS_{k}^{(1),m\ell} ( \nabla (a_{k,j}\psi_{q+1,j,k}), a_{k,j} \psi_{q+1,j,-k}) \right) \notag\\
&\quad = \tilde \SSS_{k}^{(1),m\ell} (D_{t,q}\left( \nabla (a_{k,j}\psi_{q+1,j,k})\right), a_{k,j} \psi_{q+1,j,-k}) 
+ \tilde \SSS_{k}^{(1),m\ell} ( \nabla (a_{k,j}\psi_{q+1,j,k}),D_{t,q}( a_{k,j} \psi_{q+1,j,-k})) \notag\\
&\quad \qquad + \left[ D_{t,q}, \tilde \SSS_{k}^{(1),m\ell} \right](\nabla (a_{k,j}\psi_{q+1,j,k}), a_{k,j} \psi_{q+1,j,-k})
\notag\\
&\quad = \tilde \SSS_{k}^{(1),m\ell} (\nabla u_q \cdot \nabla (a_{k,j}\psi_{q+1,j,k}) , a_{k,j} \psi_{q+1,j,-k}) 
+ \left[ D_{t,q}, \tilde \SSS_{k}^{(1),m\ell} \right](\nabla (a_{k,j}\psi_{q+1,j,k}), a_{k,j} \psi_{q+1,j,-k}).
\end{align*}
From \eqref{eq:K:k:r:bounds} and Lemma~\ref{lem:multi:Dt:commutator} it then follows that
\begin{align*}
\norm{D_{t,q} O_{21}}_{C^0}
&\lesssim \sum_{j,k} \norm{\chi_j}_{C^0} \norm{\chi_j'}_{C^0} \norm{ \tilde \SSS_{k}^{(1),m\ell} ( \nabla (a_{k,j}\psi_{q+1,j,k}), a_{k,j} \psi_{q+1,j,-k})}_{C^0} \notag\\
&\qquad + \sum_{j,k} \norm{\chi_j^2}_{C^0} \norm{ \tilde \SSS_{k}^{(1),m\ell} (\nabla u_q \cdot \nabla (a_{k,j}\psi_{q+1,j,k}), a_{k,j} \psi_{q+1,j,-k})}_{C^0} \notag\\
&\qquad + \sum_{j,k} \norm{\chi_j^2}_{C^0} \norm{\left[ D_{t,q}, \tilde \SSS_{k}^{(1),m\ell} \right](\nabla (a_{k,j}\psi_{q+1,j,k}), a_{k,j} \psi_{q+1,j,-k})}_{C^0} \notag\\
&\lesssim \sum_{j,k} \norm{\chi_j}_{C^0} \norm{\chi_j'}_{C^0} \norm{ \nabla (a_{k,j}\psi_{q+1,j,k})}_{C^0} \norm{a_{k,j} \psi_{q+1,j,-k}}_{C^0} \notag\\
&\qquad + \sum_{j,k} \norm{\chi_j^2}_{C^0} \norm{ \nabla u_q}_{C^0} \norm{\nabla (a_{k,j}\psi_{q+1,j,k})}_{C^0} \norm{a_{k,j} \psi_{q+1,j,-k}}_{C^0} \notag\\
&\lesssim (\tau_{q+1}^{-1} + \lambda_q^2 \delta_{q}^{1/2}) (\lambda_q \delta_{q+1}^{1/2} + \delta_{q+1}^{1/2} \tau_{q+1}\lambda_{q+1} \lambda_q^2 \delta_q^{1/2} ) \delta_{q+1}^{1/2} \notag\\
&\lesssim \lambda_q^2\lambda_{q+1}\delta_q^{1/2}\delta_{q+1}
\end{align*}
The estimate for $O_{22}$ is similar, and we obtain that
\begin{align*}
\norm{D_{t,q} R_{O,{\rm low}}}_{C^0} \leq \lambda_q^2\lambda_{q+1}\delta_q^{1/2}\delta_{q+1}\,.
\end{align*}

\subsubsection{Material derivative  of the $R_{O,{\rm high}}$ stress}

Recall cf.~\eqref{eq:RO:high:decompose}, the decomposition  of $R_{O,{\rm high}} = O_3 - O_4$.
Applying $D_{t,q} = \partial_t + u_q \cdot \nabla$ to the $O_3$ equation,  we find that
\begin{align*}
 D_{t,q} O_3 
 &= \left[D_{t,q},\BB \tilde P_{\approx \lambda_{q+1}}\right]   \sum_{\underset{k+k'\neq 0 } {j,j',k,k'}} \big(\Lambda \PP_{q+1,k} \tilde w_{q+1,j,k}\big) \cdot \nabla \big(\PP_{q+1,k'} \tilde w_{q+1,j',k'}\big)    \notag\\
 &\quad + \BB \tilde P_{\approx \lambda_{q+1}}  \sum_{\underset{k+k'\neq 0 } {j,j',k,k'}} \big(\left[D_{t,q}, \Lambda \PP_{q+1,k}\right] \tilde w_{q+1,j,k}\big) \cdot \nabla \big(\PP_{q+1,k'} \tilde w_{q+1,j',k'}\big)    \notag\\
 &\quad + \BB \tilde P_{\approx \lambda_{q+1}}  \sum_{\underset{k+k'\neq 0 } {j,j',k,k'}} \big(  \Lambda \PP_{q+1,k} (D_{t,q} \tilde w_{q+1,j,k}) \big) \cdot \nabla \big(\PP_{q+1,k'} \tilde w_{q+1,j',k'}\big)    \notag\\
&\quad + \BB \tilde P_{\approx \lambda_{q+1}} \sum_{\underset{k+k'\neq 0 } {j,j',k,k'}} \big( \Lambda \PP_{q+1,k}  \tilde w_{q+1,j,k}\big) \cdot \left[D_{t,q},\nabla \PP_{q+1,k'} \right]\big(\tilde w_{q+1,j',k'}\big)   \notag\\
&\quad + \BB \tilde P_{\approx \lambda_{q+1}} \sum_{\underset{k+k'\neq 0 } {j,j',k,k'}} \big( \Lambda \PP_{q+1,k}  \tilde w_{q+1,j,k}\big) \cdot  \nabla \PP_{q+1,k'} \big(D_{t,q} \tilde w_{q+1,j',k'}\big) \notag\\
&=: \tilde O_{31} + \tilde O_{32} + \tilde O_{33} + \tilde O_{34} + \tilde O_{35}.
\end{align*}
The term $\tilde O_{31}$ is bounded directly using Corollary~\ref{cor:Dt:commutator}, estimate \eqref{eq:comm:material:convolution}, with $\lambda = \lambda_{q+1}$ and $s=-1$ as
\begin{align*}
 \norm{\tilde O_{31}}_{C^0} 
 &\lesssim \lambda_{q+1}^{-1} \norm{\nabla u_q}_{C^0} \sum_{j,j',k,k'} \lambda_{q+1} \norm{\tilde w_{q+1,j,k}}_{C^0} \lambda_{q+1} \norm{\tilde w_{q+1,j',k'}}_{C^0} \notag\\
 &\lesssim \lambda_{q+1} \lambda_q^2 \delta_q^{1/2} \delta_{q+1} \,.
\end{align*}
Similarly, the terms $\tilde O_{32}$ and $\tilde O_{34}$ are bounded  using  \eqref{eq:comm:material:convolution} with $\lambda = \lambda_{q+1}$ and $s= 1$ as
\begin{align*}
\norm{\tilde O_{32}}_{C^0} +\norm{\tilde O_{34}}_{C^0}
&\lesssim  \lambda_{q+1}^{-1} \sum_{j,j',k,k'} \lambda_{q+1} \norm{\nabla u_q}_{C^0} \norm{\tilde w_{q+1,j,k}}_{C^0} \lambda_{q+1} \norm{\tilde w_{q+1,j,k}}_{C^0}\\
&\lesssim \lambda_{q+1} \lambda_q^2 \delta_q^{1/2} \delta_{q+1}  \,.
\end{align*}
In order to bound $\tilde O_{33}$ and $\tilde O_{35}$ we note that  by the construction of $w_{q+1}$ we have that 
\begin{align*}
D_{t,q}\tilde w_{q+1,j,k}(x,t)=\chi'_{j}(t)  a_{k,j}(x,t) b_k( \lambda_{q+1} \Phi_j(x,t) ) = \frac{\chi'_j(t)}{\chi_j(t)} \tilde w_{q+1,j,k}(x,t).
\end{align*}
This fact allows us to rewrite $\tilde O_{33}$ and $\tilde O_{35}$ as follows:
\begin{align*}
 \tilde O_{33} &=  \BB \tilde P_{\approx \lambda_{q+1}}  \sum_{\underset{k+k'\neq 0 } {j,j',k,k'}} \frac{\chi_j'}{\chi_j} \big(  \Lambda \PP_{q+1,k} (\tilde w_{q+1,j,k}) \big) \cdot \nabla \big(\PP_{q+1,k'} \tilde w_{q+1,j',k'}\big)\,, \\
 \tilde O_{35} &= \BB \tilde P_{\approx \lambda_{q+1}} \sum_{\underset{k+k'\neq 0 } {j,j',k,k'}} \frac{\chi_{j'}'}{\chi_{j'}} \big( \Lambda \PP_{q+1,k}  \tilde w_{q+1,j,k}\big) \cdot  \nabla \PP_{q+1,k'} \big( \tilde w_{q+1,j',k'}\big)\,,
\end{align*}
and upon noting that 
\begin{align*}
\norm{\frac{\chi_j'}{\chi_j}} + \norm{\frac{\chi_{j'}'}{\chi_{j'}}} \lesssim \tau_{q+1}^{-1} \,,
\end{align*}
we may use the bounds previously established for $O_3$ to conclude that
\begin{align*}
 \norm{\tilde O_{33}}_{C^0} +\norm{\tilde O_{35}}_{C^0} 
 &\lesssim \tau_{q+1}^{-1} \norm{O_3}_{C^0} \notag\\
 &\lesssim \tau_{q+1}^{-1} \lambda_q \delta_q^{1/4} \delta_{q+1}^{3/4} \notag\\
 &\leq  \lambda_{q+1} \lambda_q^2 \delta_q^{1/2} \delta_{q+1} \,.
\end{align*}

Bounding the material derivative of $O_4$ is very similar. We first  write $D_{t,q} O_4 $ as
\begin{align*}
 D_{t,q} O_4 
 &= \left[D_{t,q}, \BB \tilde P_{\approx \lambda_{q+1}}\right] \sum_{\underset{k+k'\neq 0 } {j,j',k,k'}} \big(\nabla \PP_{q+1,k} \tilde w_{q+1,j,k}\big)^T \cdot \big( \Lambda  \PP_{q+1,k'} \tilde w_{q+1,j',k'} \big) \notag\\
 &\qquad + \BB \tilde P_{\approx \lambda_{q+1}} \sum_{\underset{k+k'\neq 0 } {j,j',k,k'}} \big(\big[D_{t,q},\nabla \PP_{q+1,k}\big] \tilde w_{q+1,j,k}\big)^T \cdot \big( \Lambda  \PP_{q+1,k'} \tilde w_{q+1,j',k'} \big) \notag\\
&\qquad + \BB \tilde P_{\approx \lambda_{q+1}} \sum_{\underset{k+k'\neq 0 } {j,j',k,k'}} \frac{\chi_j'}{\chi_j} \big(\nabla \PP_{q+1,k} \tilde w_{q+1,j,k}\big)^T \cdot \big( \Lambda  \PP_{q+1,k'} \tilde w_{q+1,j',k'} \big) \notag\\
&\qquad + \BB \tilde P_{\approx \lambda_{q+1}} \sum_{\underset{k+k'\neq 0 } {j,j',k,k'}} \big( \nabla \PP_{q+1,k} \tilde w_{q+1,j,k}\big)^T \cdot \big( \big[ \Lambda  \PP_{q+1,k'}, D_{t,q} \big] \tilde w_{q+1,j',k'} \big) \notag\\
&\qquad + \BB \tilde P_{\approx \lambda_{q+1}} \sum_{\underset{k+k'\neq 0 } {j,j',k,k'}} \frac{\chi_{j'}'}{\chi_{j'}}  \big(\big[D_{t,q},\nabla \PP_{q+1,k}\big] \tilde w_{q+1,j,k}\big)^T \cdot \big( \Lambda  \PP_{q+1,k'} \tilde w_{q+1,j',k'} \big) \,,
\end{align*}
and using similar arguments as above, together with the bound previously established on $O_4$, we obtain that
\begin{align*}
\norm{D_{t,q}O_4}_{C^0} 
&\lesssim \tau_{q+1}^{-1} \norm{O_4}_{C^0} + \lambda_{q+1}^{-1} \norm{\nabla u_q}_{C^0} \sum_{j,j',k,k'} \lambda_{q+1} \norm{\tilde w_{q+1,j,k}}_{C^0} \lambda_{q+1} \norm{\tilde w_{q+1,j',k'}}_{C^0}
\notag\\ 
&\lesssim \tau_{q+1}^{-1}  \lambda_q \delta_q^{1/4} \delta_{q+1}^{3/4} + \lambda_{q+1} \delta_{q+1} \lambda_q^2 \delta_q^{1/2}
\notag\\
&\lesssim \lambda_q^2 \lambda_{q+1} \delta_q^{1/2}\delta_{q+1}  \,,
\end{align*}
which concludes the proof of the material derivative estimate for $R_{O,{\rm high}}$.

\subsubsection{Bound on the material derivative of $R_{O}$}
Combining all the estimates yields
\begin{align*}
\norm{D_{t,q} R_{O}}_{C^0}\lesssim \lambda_q^2 \lambda_{q+1} \delta_q^{1/2}\delta_{q+1}=
\lambda_0^{-3+3\beta}\lambda_{q+1}^{2}\ \delta_{q+1}^{1/2} \lambda_{q+2} \delta_{q+2}
\end{align*}
from which we obtain our desired estimate if $\beta<1$ and $\lambda_0$ is sufficiently large.

\subsection{Concluding bounds on the new Reynolds stress}
Combining all the estimates above,  we now show that for the new Reynolds stress $\mathring{R}_{q+1}$,
 estimate \eqref{eq:ind:q:2} holds with $q$ replaced by $q+1$ as follows:
\begin{proposition}
Assuming $\lambda_0$ is sufficiently large then \eqref{eq:ind:q:2} and \eqref{eq:ind:q:4} hold, and
\begin{align}
\| \mathring R_q\|_{C^0} &\leq \eps_R \lambda_{q+1} \delta_{q+1} \,, \label{eq:R_est_final}\\
\norm{ (\partial_t + u_q \cdot \nabla) \mathring R_q}_{C^0} &\leq  \lambda_q^2 \delta_q^{1/2} \lambda_{q+1} \delta_{q+1}.\label{eq:D_t_R_est_final}
\end{align}
\begin{proof}
The first estimate \eqref{eq:R_est_final} follows directly from Lemmas \ref{lem:transp:error}--\ref{lem:osc:error}. To prove 
\eqref{eq:D_t_R_est_final}, first note that since $\norm{\Lambda w_{q+1}}_{C^0} \lesssim \lambda_{q+1} \delta_{q+1}^{1/2}$, 
and $\mathring R_{q+1}$ has compact support in frequency in the ball of radius $4 \lambda_{q+1}$, we obtain from 
 Lemmas \ref{lem:transp:error}--\ref{lem:osc:error} that if $\lambda_0$ is sufficiently large, then
\begin{align*}
\norm{D_{t,q+1} \mathring R_{q+1}}_{C^0} 
&\leq \norm{D_{t,q} \mathring R_{q+1}}_{C^0} + \norm{\Lambda w_{q+1} \cdot \nabla \mathring R_{q+1}}_{C^0} \\
&\leq \frac12 \lambda_{q+1}^2 \delta_{q+1}^{1/2}  \lambda_{q+2} \delta_{q+2} + \frac12\lambda_{q+1} \delta_{q+1}^{1/2} \lambda_{q+1} \lambda_{q+2} \delta_{q+2}  = \lambda_{q+1}^2 \delta_{q+1}^{1/2}  \lambda_{q+2} \delta_{q+2} \,.
\end{align*}
\end{proof}
\end{proposition}

\section{The Hamiltonian increment}
\label{sec:hamiltonian}
In this section, we conclude the proof of Proposition \ref{prop:main} by showing  that  \eqref{eq:energy_ind} and  \eqref{eq:zero_reynolds} hold with $q$ replaced by $q+1$. We begin by stating some consequences of the inductive estimates in Section \ref{s:inductive_assump}:
\begin{lemma}\label{lem:rho_diff}
If $t$ is in the support of the cut-off function $\chi_j$, then
\begin{equation}\label{eq:rho_diff}
\abs{\int_{\TT^2}\left(\abs{\Lambda^{\sfrac12} v_q(x,t)}^2-\abs{\Lambda^{\sfrac12} v_q(x,\tau_{q+1}j)}^2\right)~dx}+\abs{\ee(t)+\ee(\tau_{q+1}j)}\leq \frac{\lambda_{q+
2}\delta_{q+2}}{16}~.
\end{equation}
Consequently, if $\rho_j\neq 0$ then on the support of $\chi_j$
\begin{equation}\label{eq:rho_j_t}
\lambda_{q+1}\abs{\rho(t)-\rho_j}\leq \frac{\lambda_{q+
2}\delta_{q+2}}{16}\,,
\end{equation}
and by the definition \eqref{eq:rho_def} 
\begin{equation}\label{eq:energy_lower}
\ee(t) - \int_{\TT^2}\abs{\Lambda^{\sfrac12} v_q}^2~dx\geq\frac{7\lambda_{q+
2}\delta_{q+2}}{16}.
\end{equation}\label{eq:rho_vanish}
If $\rho_j=0$  then by the definition \eqref{eq:rho_def}  and \eqref{eq:zero_reynolds} 
\begin{align*}
e(j\tau_{q+1}) - \int_{\TT^2}\abs{\Lambda^{\sfrac12} v_q(x,j\tau_{q+1})}^2~dx \leq \frac{9\lambda_{q+
2}\delta_{q+2}}{16}~\mbox{and}~\mathring R_q(\cdot, t)\equiv 0~.
\end{align*}
\end{lemma}
\begin{proof}
Using the equation for $v_q$,  we have that
\begin{align*}
\abs{\int_{\TT^2}\left(\abs{\Lambda^{\sfrac12} v_q(x,t)}^2-\abs{\Lambda^{\sfrac12} v_q(x,\tau_{q+1}j)}^2\right)~dx}& = 2\abs{\int_{\tau_{q+1}j}^t\Lambda^{\sfrac12}v_q \cdot \div \Lambda^{\sfrac12}R_q~dx}\\
& \lesssim  \abs{t-\tau_{q+1} j}\lambda_q^2\delta_q^{\sfrac12}\lambda_{q+1}\delta_{q+1}\\
& \lesssim  \tau_{q+1}\lambda_q^2\delta_q^{\sfrac12}\delta_{q+1}\lambda_{q+1}\\
& =  \lambda_q\delta_q^{1/4}\delta_{q+1}^{3/4} \,,
\end{align*}
where in the second to last line,  we used that by our hypothesis $\abs{t-\tau_{q+1} j}\leq 4\tau_{q+1}$. Trivially, we have that
\begin{align*}
\abs{\ee(t)+\ee(\tau_{q+1}j)}\lesssim \tau_{q+1}\,.
\end{align*}
Thus,
\[
\abs{\int_{\TT^2}\left(\abs{\Lambda^{\sfrac12} v_q(x,t)}^2-\abs{\Lambda^{\sfrac12} v_q(x,\tau_{q+1}j)}^2\right)~dx}+\abs{\ee(t)+\ee(\tau_{q+1}j)}\lesssim \left(\lambda_0^{-2+5\beta/2}+\lambda_0^{-3\beta/2}\lambda_{q+2}^{-3+3\beta}\right)\lambda_{q+2}\delta_{q+2}
\]
Hence, if $\beta<\frac 45$, the lemma is proved.
\end{proof}

We are now in  position to prove  that \eqref{eq:energy_ind} and  \eqref{eq:zero_reynolds} hold with $q$ replaced by $q+1$.
\begin{proposition}
If $\rho(j)\neq 0$ and $t$ is in the support of $\chi_j$, then
\begin{equation}\label{eq:energy_inductive_step}
 \frac{\lambda_{q+
2}\delta_{q+2}}{4} \leq \ee(t) - \int_{\TT^2}\abs{\Lambda^{\sfrac12} v_{q+1}}^2~dx\leq\frac{3\lambda_{q+
2}\delta_{q+2}}{4}.
\end{equation}
Otherwise,  if  $t$ is \emph{not} in the support of the cut-off $\chi_j$ with $\rho(j)\neq 0$,  then
\begin{equation}\label{eq:upper_no_perturb}
\ee(t) - \int_{\TT^2}\abs{\Lambda^{\sfrac12} v_{q+1}}^2~dx\leq\frac{9\lambda_{q+
2}\delta_{q+2}}{16}~\mbox{and}~\mathring R_{q+1}(\cdot,t)\equiv 0\,.
\end{equation}
As a consequence of \eqref{eq:energy_inductive_step} and \eqref{eq:upper_no_perturb}, it follows that \eqref{eq:energy_ind} and  \eqref{eq:zero_reynolds} hold with $q$ replaced by $q+1$.
\end{proposition}
\begin{proof}
Assume that $t$ is not on the support of any cut-off $\chi_j$ with $\rho_j\neq 0$, then since  $w_{q+1}(\cdot, t)\equiv 0$ and $\mathring R_{q+1}(\cdot,t)=\mathring R_q(\cdot,t)$ (see Lemma \ref{lem:rho_diff} above) then $\mathring R_{q+1}(\cdot,t)\equiv 0$. Moreover, we have
\[
\ee(t) - \int_{\TT^2}\abs{\Lambda^{\sfrac12} v_{q+1}}^2~dx=\ee(t) - \int_{\TT^2}\abs{\Lambda^{\sfrac12} v_{q}}^2~dx,
\]
and thus by \eqref{eq:rho_vanish} we obtain \eqref{eq:upper_no_perturb}.

Now assume $\rho(j)\neq 0$ and $t$ is on the support of $\chi_j$. By computation it follows that
\begin{align*}
\ee(t)-\int_{\TT^2}\abs{\Lambda^{\sfrac12} v_{q+1}}^2~dx&=
\ee(t)-\int_{\TT^2}\abs{\Lambda^{\sfrac12} v_{q}}^2~dx+2\int_{\TT^2}\Lambda^{\sfrac12} v_{q} \cdot \Lambda^{\sfrac12} w_{q+1}~dx-\int_{\TT^2}\abs{\Lambda^{\sfrac12} v_{q}}^2~dx
\\&=\ee(t)-\int_{\TT^2}\abs{\Lambda^{\sfrac12} v_{q}}^2~dx-\int_{\TT^2}\abs{\Lambda^{\sfrac12} w_{q+1}}^2~dx
\end{align*}
where we used the disjoint frequency support of $v_q$ and $w_{q+1}$. Utilizing the frequency supports of $\tilde w_{q+1,j,k}$ we obtain
\begin{align*}
\int_{\TT^2}\abs{\Lambda^{\sfrac12} w_{q+1}}^2~dx=&\int_{\TT^2} \abs{\sum_{k,k',j,j'} \Lambda^{\sfrac12} \PP_{q+1,k}  \tilde w_{q+1,j,k}}^2~dx\\
=&\sum_{k,j'}\int_{\TT^2} \Lambda^{\sfrac12} \PP_{q+1,k}  \tilde w_{q+1,j,k}\cdot \Lambda^{\sfrac12} \PP_{q+1,-k}  \tilde w_{q+1,j,-k}~dx \\
=&\sum_{k,j}\int_{\TT^2} \Lambda^{\sfrac12} \PP_{q+1,k}  \left( \chi_{j} a_{k,j} b_k (\lambda_{q+1} \Phi_j) \right)
\cdot \Lambda^{\sfrac12} \PP_{q+1,-k}  \left( \chi_{j} a_{-k,j} b_{-k} (\lambda_{q+1} \Phi_j) \right)~.
~dx \end{align*}
Then, by the definition of $b_k$, estimates \eqref{eq:local:w:q+1:*} and \eqref{eq:Lambda:w:q+1:*}, and the fact that the mean of a high frequency object vanishes, it follows that
\begin{align*}
\int_{\TT^2}\abs{\Lambda^{\sfrac12} w_{q+1}}^2~dx 
&= \int_{\TT^2} w_{q+1} \cdot \Lambda w_{q+1} ~dx 
\notag\\
&=\sum_{k,j}\int_{\TT^2} \lambda_{q+1} \chi_j^2\abs{a_k}^2~dx 
\\
&\quad +\sum_{j,k} \int_{\TT^2}\chi_{j}  \big[\PP_{q+1,k}, a_{k,j}  \psi_{q+1,j,k} \big] b_k( \lambda_{q+1} x )\cdot \Lambda w_{q+1}~dx
\\
&\quad + \sum_{j,k}  \int_{\TT^2} \tilde w_{q+1,j,k}\cdot \chi_{j}    [\PP_{q+1,-k} \Lambda, a_{k,j}  \psi_{q+1,j,-k} ] b_{-k}( \lambda_{q+1} x )~dx
\\
&=: (2\pi)^2 \lambda_{q+1}\sum_j \chi_j^2 \rho_j+ E_1+E_2.
\end{align*}
Using Lemmas~\ref{lem:w:q+1:bounds},~\ref{lem:D:N:a:psi}, and~\ref{lem:u:grad:commutator} we may estimate
\begin{align*}
E_1+E_2  
&\lesssim \sum_{j,k} \chi_j  \norm{\nabla (a_{k,j} \psi_{q+1,j,k})} \delta_{q+1}^{1/2}
\notag\\
&\lesssim \lambda_q\delta_{q+1}+\lambda_q \delta_q^{1/4} \delta_{q+1}^{3/4} 
\notag\\
&\lesssim \lambda_q \delta_q^{1/4} \delta_{q+1}^{3/4}=\lambda_0^{-2+5\beta/2}\lambda_{q+1}\delta_{q+2}.
\end{align*}

Finally, applying \eqref{eq:rho_def}, \eqref{eq:rho_j_t} and \eqref{eq:energy_lower} and $\beta<4/5$ we obtain
\begin{align*}
\frac{\lambda_{q+
2}\delta_{q+2}}{4} \leq \ee(t) - \int_{\TT^2}\abs{\Lambda^{\sfrac12} v_{q+1}}^2~dx\leq\frac{3\lambda_{q+
2}\delta_{q+2}}{4}
\end{align*}
which concludes the proof.
\end{proof}

\appendix

\section{Appendix}

\subsection{Variational principle for hydrodynamical systems}\label{appendix:EP}
We provide a derivation of the system \eqref{eq:hydro}.
Many models of incompressible hydrodynamical systems can be written as geodesic equations of right-invariant metrics on the Lie
group 
$ \mathcal{D} _\mu$,  the group of volume-preserving diffeomorphisms, with group multiplication given by composition on the right.
The Lie algebra $ \mathcal{V} $ associated to this group 
is the vector space of divergence-free vector fields.   The Lie bracket  on $ \mathcal{V} $ is given by
$[u,v] = \partial_j u v^j - \partial _j v u^j$. 

On $ \mathcal{V} $ we (formally) define the metric 
$$
(u, w) = \int_{ \mathbb{T}  ^2} A u \cdot  w \, dx \,,
$$
where $A$ is a self-adjoint, positive operator.   This metric is then right-translated over the Lie group $ \mathcal{D} _\mu$.   We then define
the Lagrangian function $l$ on the Lie algebra $ \mathcal{V} $ by
\begin{equation}\label{eq:A-lag}
l(u) = {\frac{1}{2}} (u, u) =  \int_{ \mathbb{T}  ^2} A u \cdot  u \, dx \,.
\end{equation} 

The well-known Euler-Poincar\'{e} variational principle provides a simple procedure for computing the equations of motion associated to
the Lagrangian $l$ on the Lie algebra $ \mathcal{V} $.  We shall state this as the following proposition, whose proof can be found in
Chapter 13 of \cite{MaRa1999}).

\begin{proposition}[Euler-Poincar\'{e} Variational Principle]\label{prop:ep}
With  $l: \mathcal{V} \rightarrow{\mathbb R}$ given by \eqref{eq:A-lag}, the following are equivalent:
\begin{itemize} 
\item[\bf{(a)}]  $u(t):=u(t, \cdot )$ is a geodesic curve, solving
$$
\frac{d}{dt}\frac{\delta l}{\delta u} = -\operatorname{ad}^*_{u}
\frac{\delta l}{\delta u},
$$
where  $\operatorname{ad}^*_u$ is defined by
$$
(\operatorname{ad}^*_u v, w) = - (v, [u,w]),
$$
for $u,v,w$ in $\mathcal{V} $;
\item[\bf{(b)}] the curve $u(t)$ is an extremum
of the action function
$$
s(u) = \int l(u(t)) dt,
$$
for variations of the form
\begin{equation}\nonumber
\delta u = \partial_t w + [w, u],
\end{equation}
where $w \in \mathcal{V} $ vanishes at the endpoints $t=0$ and $t=T$.
\end{itemize}
\end{proposition}
We make use of Proposition \ref{prop:ep} {\bf(b)} to derive \eqref{eq:hydro}.   We define the potential velocity 
$$
v = A u,
$$
and compute the first variation of $s(u) = \int_0^T l(u) dt$:
\begin{align*}
\delta s(u) \cdot \delta u &= \int_0^T \int_{  \mathbb{T} ^2 }v  \cdot \delta u \, dx dt \\
 &= \int_0^T \int_{  \mathbb{T} ^2 } v ^i\, \left(  \partial_t w^i + \p_j w^i u^j - \p_ju^i w^j  \right) \, dx dt\\
  &= -\int_0^T \int_{  \mathbb{T} ^2 } \left(  \partial_t v^i + \p_j v^i u^j + \p_iu^j  v^j \right) \, w^i \, dx dt \,.
\end{align*} 
Setting $\delta s(u) \cdot \delta u=0$ for all divergence-free variations $w$, and applying the Hodge decomposition, we find that there exists
a pressure function $ \tilde p :  \mathbb{T} ^2 \to \mathbb{R}  $ such that
\begin{subequations}\label{eq:SQG-bad}
\begin{align} 
 \partial_t v^i + \p_j v^i u^j + \p_iu^j  v^j  & = -\p_i  \tilde p  \,, \\
 \operatorname{div} u &= 0 \,,
\end{align} 
\end{subequations}
which is the general hydrodynamical system \eqref{eq:hydro}.

\subsection{Transport and composition estimates}
\label{sec:transport}

In this section we gather some classical estimates for transport equations. For refer the reader to e.g.~\cite[Section 4.3]{Bu2014} or~\cite[Proposition D.1]{BuDLeIsSz2015} for proofs of these classical facts.

In this section we recall some well known results regarding smooth solutions of
the \emph{transport equation}:
\begin{subequations}
\label{eq:transport}
\begin{align}
\partial_t f + u\cdot  \nabla f &=g,\\ 
f|_{t_0}&=f_0,
\end{align}
\end{subequations}
where $u=u(t,x)$ is a given smooth vector field. 
We denote the corresponding material derivative by $D_{t} = \partial_t+u\cdot \nabla$. Moreover, define by $\Phi (t, \cdot)$ to be the inverse of the flow associated to the vector field $u$ starting at time $t_0$ as the identity, i.e., $\Phi = X^{-1}$, where $\frac{d}{dt} X(x,t) = u (X(x,t),t)$ and $X (x, t_0 )=x$. Then we have:
\begin{lemma}[Transport estimates]
\label{lem:transport}
 Assume $t>t_0$. Any solution $f$ of \eqref{eq:transport} satisfies
\begin{align}
\norm{f (t)}_{C^0} &\leq \norm{f_0}_{C^0} + \int_{t_0}^t \norm{g (\tau)}_{C^0}\, d\tau\,,\label{eq:max:prin}\\
\norm{D f(t)}_{C^0} &\leq \norm{D f_0}_{C^0} e^{(t-t_0) \norm{D u}_{C^0}} + \int_{t_0}^t e^{(t-\tau) \norm{D u}_{C^0}} \norm{ Dg (\tau)}_{C^0}\, d\tau\,,\label{eq:trans:est:0}
\end{align}
and, more generally, for any $N\geq 2$ there exists a constant $C=C(N)$ so that
\begin{align}
\norm{ D^N f (t)}_{C^0} & \leq \bigl( \norm{D^N f_0}_{C^0} + C(t-t_0) \norm{D^N u}_{C^0} \norm{D f_0}_{C^0} \bigr)e^{C(t-t_0) \norm{D u}_{C^0}} \notag \\
&\qquad +\int_{t_0}^t e^{C(t-\tau) \norm{D u}_{C^0} }\bigl(\norm{D^N g (\tau)}_{C^0} + C (t-\tau ) \norm{ D^N v}_{C^0} \norm{ D g (\tau)}_{C^0} \bigr)\,d\tau.
\label{eq:trans:est:1}
\end{align}
Moreover,
\begin{align}
\norm{D\Phi (t) -\Id}_{C^0} &\leq e^{(t-t_0)\norm{D u}_{C^0}}-1 \leq (t-t_0)\norm{D u}_{C^0} e^{(t-t_0)\norm{D u}_{C^0}}\,,  \label{eq:Dphi:near:id}\\
\norm{D^N \Phi (t)}_{C^0} &\leq C(t-t_0) \norm{D^N u}_{C^0} e^{C(t-t_0)\norm{D u}_{C^0}}, \label{eq:Dphi:N}
\end{align}
holds for all $N \geq 2$ and a suitable constant $C = C(N)$.
\end{lemma}

In order to take advantage of Lemma~\ref{lem:transport} we also need to appeal to the following standard composition estimate:
\begin{lemma}[Chain rule]
 \label{lem:composition}
 Let $\Psi: \Omega \to \RR$ and $u: \RR^d \to \Omega$ be two smooth functions, with $\Omega\subset \RR^D$. 
Then, for every $N \in \NN\setminus \{0\}$ there is a constant $C$ such that
\begin{align}
\norm{D^N \left( \Psi\circ u \right)}_{C^0} &\leq C \left(\norm{D \Psi}_{C^0} \norm{D^N u}_{C^0} +\|D\Psi\|_{C^{N-1}} \norm{u}_{C^0}^{N-1} \norm{D^N u}_{C^0} \right)\, .
\label{eq:chain:i} \\
\norm{D^N \left( \Psi\circ u \right)}_{C^0} &\leq C \left(\norm{D \Psi}_{C^0} \norm{D^N u}_{C^0} +\|D\Psi\|_{C^{N-1}} \norm{D u}_{C^0}^{N} \right)\, .
\label{eq:chain}
\end{align} 
where $C = C (N,d,D)$.
\end{lemma}

\subsection{Inverse of the divergence}
\label{sec:BB}

In this section we prove a number of estimates for the operator $\BB$ defined in Definition~\ref{def:BB}.
We take advantage of the frequency localization of our perturbation and establish the following lemma.
\begin{lemma}[Inverse divergence gains a derivative]
\label{lem:BB}
Let $\BB$ be as defined in Definition~\ref{def:BB}. For $f \colon \TT^2 \to \CC^2$ that is smooth,  we have that 
\[
\div (\BB f) = \PP \left(f - \frac{1}{|\TT^2|} \int_{\TT^2} f(x) dx\right), 
\]
and $\BB f$ is a symmetric, trace-free matrix. 
Fix $\lambda  \geq 1$, and denote by $P_{\approx \lambda}$ a Fourier multiplier operator with symbol that is supported on $\{ \xi \colon \lambda/2 \leq |\xi|\leq 2\lambda\}$ and is identically $1$ on $\{ \xi \colon 3\lambda/4 \leq |\xi| \leq 3\lambda/2\}$. Then, for a smooth functions $f,g \colon \TT^2 \to \CC$, with $\supp ( \hat g(\xi) ) \subset \{\xi \colon |\xi|\leq \lambda/4\}$, we have
\begin{align}
\norm{\BB \big( g(x)  P_{\approx \lambda} f(x) \big)}_{C^0} 
&\lesssim  \frac{\norm{g}_{C^0} \norm{P_{\approx \lambda} f}_{C^0} }{\lambda} \lesssim  \frac{\norm{g}_{C^0} \norm{f}_{C^0} }{\lambda} 
\label{eq:BB:L:infty} \\
\norm{[ \BB , g(x)]  P_{\approx \lambda} f(x) }_{C^0} 
&\lesssim  \frac{\norm{g}_{C^1} \norm{P_{\approx \lambda} f}_{C^0} }{\lambda^2}  \lesssim \frac{\norm{g}_{C^1} \norm{f}_{C^0} }{\lambda^2} 
\label{eq:BB:comm}
\end{align}
for some implicit universal constant $C> 0$.
\end{lemma}

\begin{proof}[Proof of Lemma~\ref{lem:BB}]
The first assertions follow directly from the definition of $\BB$. In order to prove \eqref{eq:BB:L:infty} we note that by the assumption on the frequency support of $g$, we have
\begin{align*}
g(x)  P_{\approx \lambda} f(x) = \tilde P_{\approx \lambda}  \big( g(x)  P_{\approx \lambda} f(x) \big)
\end{align*}
where we have denoted
\begin{align*}
\tilde P_{\approx \lambda}  = P_{\approx 2^{j-2} \lambda} + P_{\approx 2^{j-1} \lambda} + P_{\approx 2^j \lambda} + P_{\approx 2^{j+1} \lambda} .
\end{align*}
Thus, by the definition of $\BB$ and applying twice the Bernstein inequality on $L^\infty$, we obtain
\begin{align*}
\norm{\BB \big( g(x) P_{\approx \lambda} f(x) \big)}_{C^0} \lesssim \frac{1}{\lambda} \norm{g(x)  P_{\approx \lambda} f(x) }_{C^0} \lesssim \frac{1}{\lambda} \norm{g}_{C^0} \norm{f}_{C^0}.
\end{align*}

Lastly, in order to prove \eqref{eq:BB:comm} we note that since $\BB$ has components which are Fourier multipliers, we have
\begin{align*}
[\BB , g] P_{\approx \lambda} f 
&= \BB \big( g P_{\approx \lambda} f \big) - g \BB P_{\approx \lambda} f  \\
&= \BB \tilde P_{\approx \lambda} \big( g P_{\approx \lambda} f \big) - g \BB \tilde P_{\approx \lambda}  P_{\approx \lambda} f  \notag\\
&= [ \BB \tilde P_{\approx \lambda} , g] P_{\approx \lambda} f.
\end{align*}
Denoting by ${\mathcal K}_{\approx\lambda}$ the real convolution kernel corresponding to the Fourier multiplier operator $\BB \tilde P_{\approx \lambda} $ (which is a symbol of order $-1$), one may check that it obeys the bounds
\begin{align}
\norm{|x|^b \nabla^a {\mathcal K}_{\approx \lambda} (x)}_{L^1(\RR^2)} \leq C \lambda^{a-b-1}
\label{eq:BB:K:aux}
\end{align}
for all $a,b\geq 0$, and some constant $C = C(a,b)$. Therefore, by the mean value theorem
\begin{align*}
[ \BB \tilde P_{\approx \lambda} , g] P_{\approx \lambda} f(x) 
&= \int_{\RR^d} (g(y) - g(x)) {\mathcal K}_{\approx \lambda}(x-y) P_{\approx \lambda} f(y) dy \notag\\
&=- \int_{\RR^d} \left(\int_0^1 \nabla g(y - \tau (y-x)) d\tau\right)\cdot (y-x) {\mathcal K}_{\approx \lambda}(x-y) P_{\approx \lambda} f(y) dy
\end{align*}
and the kernel estimate \eqref{eq:BB:K:aux} with $b=1$ and $a=0$ we arrive at
\begin{align*}
\norm{[ \BB \tilde P_{\approx \lambda} , g] P_{\approx \lambda} f}_{C^0} 
&\lesssim \norm{D g}_{C^0} \norm{|x| {\mathcal K} _{\approx \lambda}(x)}_{L^1} \norm{P_{\approx \lambda} f}_{C^0} \notag\\
&\lesssim \lambda^{-2} \norm{g}_{C^1} \norm{f}_{C^0}.
\end{align*}
This concludes the proof of the lemma.
\end{proof}

\subsection{Calderon commutator}
\label{sec:commutator}
We recall cf.~\cite[Theorem 10.3, Page 99]{Le2002} and~\cite[Lemma 2.2, Page 50]{Ma2008}.
\begin{lemma}[Calderon commutator]
\label{lem:Calderon}
Let $p \in (1,\infty)$ and $\varphi \in W^{1,\infty}(\TT^2)$. Then for any $v \in L^p(\TT^2)$ with zero mean on $\TT^2$ we have that
\begin{align}
\norm{ [ \Lambda, \varphi] v}_{L^p} \lesssim_p \norm{\varphi}_{W^{1,\infty}} \norm{v}_{L^p}
\label{eq:Calderon:Lp}
\end{align}
where the constant implicitly depends on $p$. Moreover, for $s \in [0,1]$, if $\varphi \in W^{2,\infty} (\TT^2)$ and $v \in H^s(\TT^2)$ has zero mean on $\TT^2$, then\footnote{{The constant in \eqref{eq:Calderon:Hs} is not sharp. See e.g.~\cite{Ma2008} where the constant is given as $\max\left\{ \norm{\nabla \varphi}_{C^0} , \norm{\varphi}_{\dot{B}^{2}_{2,\infty}} \right\}$.}}
\begin{align}
\norm{ [ \Lambda, \varphi] v}_{H^s} \lesssim_s   \norm{\varphi}_{W^{2,\infty}} \norm{v}_{H^s}
\label{eq:Calderon:Hs}
\end{align}
\end{lemma}
Note that in the aforementioned references the results are stated for functions defined on $\RR^2$, whose Fourier support is at a positive distance from the origin. The same proofs work in the periodic case $\TT^2$, if the functions we consider have zero mean. To see that \eqref{eq:Calderon:Lp} holds one uses the Poisson summation convention to write the kernel associated to $\Lambda$ as in~\cite{CoCo2004}, so that off the diagonal, the singular integral kernel $K$ associated to $[\Lambda,\varphi]$ is given by
\[
K(x,y) = c  \sum_{k\in \ZZ^2}\frac{\varphi(x) - \varphi(y)}{|x-y-k|^3}  
\]
For $k = 0$ the proof closely follows the $\RR^2$ case, while for $ k\neq 0$ we are dealing with an $L^1$ kernel.
Assertion \eqref{eq:Calderon:Hs} spaces follows from the case $s=0$ (which holds by letting $p=2$ in \eqref{eq:Calderon:Lp}), the case $s=1$ (which holds since $\nabla [\Lambda, \varphi] = [\Lambda, \varphi] \nabla + [\Lambda, \nabla \varphi]$ and the bound \eqref{eq:Calderon:Lp} with $p=2$), and interpolation. We omit further details.

\subsection{Material derivatives and convolution operators}
We recall (similarly to~\cite[Lemma 7.2]{IsVi2015}) a commutator estimate involving convolution operators and material derivatives. 
\begin{lemma}
\label{lem:u:grad:commutator}
Let $s\in \RR$, $\lambda \geq 1$, and let $T_K$ be an order $s$ convolution operator localized at length scale $\lambda^{-1}$. That is, $T_K$ acts on smooth functions $f$ as
\[
T_K f(x) = \int_{\RR^2} K(y) f(x-y) dy
\]
for some kernel $K \colon \RR^2 \to \RR$ that obeys
\begin{align}
\norm{|x|^a \nabla^b K(x)}_{L^1(\RR^2)} \lesssim \lambda^{b-a + s}
\label{eq:K:convolution:assumption}
\end{align}
for all $0 \leq a, |b|\leq 1$ and some implicit constants $C  = C(a,b)$.
Then, for any smooth function $f \colon \TT^2 \to \CC$ and smooth incompressible vector field $u \colon \TT^2 \to \RR^2$ we have
\begin{align*}
\norm{[ u \cdot \nabla, T_K ] f}_{C^0} \leq \lambda^s \norm{\nabla u}_{C^0} \norm{f}_{C^0}.
\end{align*}
Similarly,  we have
\begin{align}
\norm{[ b, T_K ] f}_{C^0} \leq \lambda^{s-1} \norm{\nabla b}_{C^0} \norm{f}_{C^0}
\label{eq:comm:multiplication}
\end{align}
for smooth functions $b, f \colon \TT^2 \to \CC$.
\end{lemma}
\begin{proof}[Proof of Lemma~\ref{lem:u:grad:commutator}]
The proof is direct, and uses that $\div u=0$. We have
\begin{align*}
\left| T_K(u\cdot \nabla f)(x) - u(x) \cdot \nabla T_K f  (x) \right|
&= \left| \int_{\RR^2} \left( u(x) - u(x-y) \right) \cdot \nabla f(x-y) K(y) dy \right| \notag\\
&= \left| \int_{\RR^2} \left( u(x) - u(x-y) \right) \cdot \nabla K(y) f(x-y) dy \right| \notag\\
&\leq  \norm{\nabla u}_{C^0} \int_{\RR^2} |f(x-y)| |y| | \nabla K(y)| dy \notag\\
&\leq \norm{\nabla u}_{C^0} \norm{f}_{C^0} \norm{|y| \nabla K(y)}_{L^1}
\end{align*}
at which stage we use that the kernel is integrable cf.~\eqref{eq:K:convolution:assumption}.

For the second assertion, we use the mean value theorem. We have
\begin{align*}
\left| T_K(b\, f)(x) - b(x)  T_K f  (x) \right|
&= \left| \int_{\RR^2} \left( b(x) - b(x-y) \right)  f(x-y) K(y) dy \right| \notag\\
&= \left| \int_{\RR^2} \left( \int_0^1 \nabla b(x - \lambda y) d\lambda \right) \cdot y K(y)f(x-y) dy \right| \notag\\
&\leq \norm{\nabla b}_{C^0} \norm{f}_{C^0} \norm{|y| K(y)}_{L^1}
\end{align*}
so that the proof is completed upon using \eqref{eq:K:convolution:assumption}.
\end{proof}

\begin{remark}
When $s=0$, examples of such operators $T_K$ are given by zero-order Fourier multiplier operators with frequency support inside a shell at $|\xi| \approx \lambda$. For instance, the kernel associated to the operator $\PP_{q+1,k}$ obeys \eqref{eq:K:convolution:assumption} for $\lambda = \lambda_{q+1}$, and $s=0$. Similarly, the kernel associated to the Fourier multiplier operator  $\Lambda \PP_{q+1,k}$ obeys \eqref{eq:K:convolution:assumption} for $\lambda = \lambda_{q+1}$, and $s=1$, while the kernel of $\BB \PP_{q+1,k}$ obeys \eqref{eq:K:convolution:assumption} for $\lambda = \lambda_{q+1}$, and $s=-1$.
\end{remark}

An immediate consequence of Lemma~\ref{lem:u:grad:commutator} is:
\begin{corollary}
\label{cor:Dt:commutator}
Let $\lambda \geq 1$, $s\in \RR$, and let $K$ be a kernel which obeys \eqref{eq:K:convolution:assumption}. Given a smooth divergence free vector field  $u \colon \TT^2 \to\RR^2$, we have
\begin{align}
 \norm{[D_t , T_K ]   f }_{C^{0}} \lesssim \lambda^s \norm{\nabla u}_{C^0} \norm{f}_{C^0}.
 \label{eq:comm:material:convolution}
\end{align}
where as usual we denote $D_t = \partial_t + u \cdot \nabla$.
\end{corollary}

\subsection{Material derivatives and bilinear convolution operators}
\label{app:pseudo:product}
The bilinear convolution operators we are interested here arise from pseudo-product operators, as defined by Coifman and Meyer~\cite{CoMe1978}. Equivalently, these are translation invariant multilinear operators, see e.g.~\cite[Section 6]{GrTo2002} and~\cite[Chapter 2.13]{MuSh2013} for details.  Let $\xi = (\xi_1, \xi_2) \in \RR^2 \times \RR^2$. Let $M \colon \RR^2 \times \RR^2 \to \RR$ be a smooth multiplier. 
For two Schwartz functions $f_1, f_2 \colon \RR^2\to \RR$, define the bilinear pseudo-product operator $S_M (f_1,f_2)$ by 
\begin{align}
S_M (f_1,f_2)(x) = \frac{1}{(2\pi)^2} \intint_{\RR^2\times \RR^2} M(\xi_1,\xi_2) \hat{f_1}(\xi_1) \hat{f_2}(\xi_2) e^{i x \cdot (\xi_1+\xi_2)} d\xi_1 d\xi_2.
\label{eq:pseudo:product}
\end{align}
Equivalently, denoting by $K_M(z_1,z_2)= M^{\vee}(z_1,z_2)$ the inverse Fourier transform of $M$ in $\RR^2 \times \RR^2$, we may write 
\begin{align}
S_M(f_1,f_2)(x) = \intint_{\RR^2\times\RR^2} K_M(x-y_1,x-y_2) f_1(y_1) f_2(y_2) dy_1 dy_2
\label{eq:pseudo:product:real}
\end{align}
for Schwartz functions $f_1,f_2\colon \RR^2\to \RR$. As opposed to~\cite{CoMe1978,GrTo2002} which consider kernels of Calder\'on-Zygmund type, here we only need to consider kernels $K_M$ which obey
\begin{align}
\norm{|z|^{a} \partial_z^b K_M(z)}_{L^1(\RR^2\times\RR^2)} \leq C_{a,b} \lambda^{|b|-|a|}
\label{eq:K:convolution:assumption:multi}
\end{align}
for some $\lambda \geq 1$, all $0 \leq |a|, |b|\leq 1$ and some constants $C_{a,b}>0$. Examples are the kernels defined in \eqref{eq:multilinear:kernels}. Therefore, in contrast to~\cite[Theorem 1]{CoMe1978}, for us the boundedness of $S_M$ from $L^{p_1} \times L^{p_2} \to L^p$, where $1/p_1 + 1/p_2 = 1/p$ is automatic, and even includes the case $p_1=p_2=p = \infty$, which is the only case needed in this paper.

Instead, here we are interested in the commutator  between $S_M$ and $D_{t} = \partial_t + u\cdot\nabla_x$, where  $u\colon \RR^2\to\RR^2$, which is divergence free vector field. We have
\begin{lemma}
\label{lem:multi:Dt:commutator}
Let $\lambda \geq 1$, let $K_M$ be a kernel which obeys \eqref{eq:K:convolution:assumption:multi}, and let $S_M$ be the corresponding bilinear convolution operator given by \eqref{eq:pseudo:product:real}. Given a smooth divergence free vector field  $u \colon \TT^2 \to\RR^2$, we denote
\begin{align}
[D_t, S_M](f_1,f_2) = D_t (S_M(f_1,f_2)) - S_M(D_t f_1,f_2) - S_M(f_1, D_t f_2)  .
\label{eq:comm:Dt:SM}
\end{align}
Then we have 
\begin{align}
 \norm{[D_t, S_M](f_1,f_2)}_{C^{0}} \leq  C \norm{\nabla u}_{C^0} \norm{f_1}_{C^0} \norm{f_2}_{C^0}
 \label{eq:comm:material:convolution:multi}
\end{align}
for some constant $C>0$, which is independent of $\lambda$.
\end{lemma}
\begin{proof}[Proof of Lemma~\ref{lem:multi:Dt:commutator}]
Upon taking a material derivative of \eqref{eq:pseudo:product:real} and using the product rule, we obtain
\begin{align*}
&D_{t}(S_M(f_1,f_2))(x) - S_M(D_t f_1,f_2)(x) - S_M(f_1,D_t f_2)(x) \notag\\
&= \intint_{\RR^2\times\RR^2} u^j(x) \partial_{x}^j K_M(x-y_1,x-y_2)  f(y_1) f(y_2) dy_1 dy_2 \notag\\
&\qquad - \intint_{\RR^2\times\RR^2}  K_M(x-y_1,x-y_2)  u^j(y_1) \partial_{y_1}^j f(y_1) f(y_2) dy_1 dy_2 \notag\\
&\qquad - \intint_{\RR^2\times\RR^2}   K_M(x-y_1,x-y_2)  f(y_1) u^j(y_2) \partial_{y_2}^j f(y_2) dy_1 dy_2 \notag\\
&= \intint_{\RR^2\times\RR^2} \left(u^j(x) \partial_{z_1}^j K_M(x-y_1,x-y_2) +u^j(x) \partial_{z_2}^j K_M(x-y_1,x-y_2)  \right) f(y_1) f(y_2) dy_1 dy_2 \notag\\
&\qquad + \intint_{\RR^2\times\RR^2} u^j(y_1)  \partial_{y_1}^j K_M(x-y_1,x-y_2)   f(y_1) f(y_2) dy_1 dy_2 \notag\\
&\qquad + \intint_{\RR^2\times\RR^2} u^j(y_2)  \partial_{y_2}^j K_M(x-y_1,x-y_2)  f(y_1)  f(y_2) dy_1 dy_2 \notag\\
&= \intint_{\RR^2\times\RR^2} \left( \big(u^j(x)-u^j(y_1)\big) \partial_{z_1}^j K_M(x-y_1,x-y_2) + \big(u^j(x) - u^j(y_2) \big) \partial_{z_2}^j K_M(x-y_1,x-y_2)  \right) \notag\\
&\qquad \qquad \times f(y_1) f(y_2) dy_1 dy_2 \notag\\
&= \intint_{\RR^2\times\RR^2} \left(\frac{u^j(x)-u^j(y_1)}{x-y_1}\right) (x-y_1)\partial_{z_1}^j K_M(x-y_1,x-y_2) f(y_1) f(y_2) dy_1 dy_2  \notag\\
&\qquad + \intint_{\RR^2\times\RR^2} \left(\frac{u^j(x)-u^j(y_2)}{x-y_2}\right) (x-y_2) \partial_{z_2}^j K_M(x-y_1,x-y_2)    f(y_1) f(y_2) dy_1 dy_2.
\end{align*}
The lemma now follows from condition \eqref{eq:K:convolution:assumption:multi} with $|a|=|b|=1$.
\end{proof}

\vspace{.1in}

\noindent {\bf Acknowledgments.} 
TB was supported by the National Science Foundation  grant DMS-1600868.
SS was supported by the National Science Foundation  grant  DMS-1301380,
and by the Royal Society Wolfson Merit Award.   
VV was partially supported by the National Science Foundation  grant DMS-1514771 and by an 
Alfred P. Sloan Research Fellowship.




\newcommand{\etalchar}[1]{$^{#1}$}

\end{document}